\documentclass[12pt, BCOR=12mm, 
	DIV=calc, 
	headinclude,
	bibliography=oldstyle,
	bibliography=totoc]{scrbook}

\usepackage{amsthm,amsfonts,amsmath, amssymb}
\usepackage[mathscr]{eucal}

\KOMAoptions{DIV=calc}

\usepackage{calligra} 

\usepackage{fourier} 
\usepackage{pbsi}
\usepackage{yfonts}
\usepackage{aurical}
\usepackage{pgothic}
\usepackage[T1]{fontenc}

\usepackage[scaled]{helvet} 

\usepackage[pdftex,bookmarks=true]{hyperref}
\usepackage{graphicx}
\usepackage{enumerate}
\usepackage{scrpage2}
\setheadsepline{.1pt}
\theoremstyle{plain}
\newtheorem{thm}{Theorem}[chapter]
\newtheorem{cor}[thm]{Corollary}

\newtheorem{lem}[thm]{Lemma}
\newtheorem{prop}[thm]{Proposition}

\newtheorem{rem}[thm]{Remark}

\theoremstyle{definition}

\newtheorem{defn}[thm]{Definition}

\newtheorem{prob}[thm]{Problem}
\newtheorem{Qn}[thm]{Question}

\theoremstyle{remark}

\theoremstyle{plain}

  {\end{enumerate}}

\newcommand{\pa}{\partial}

\newcommand{\D}{\mathbb{D}}
\newcommand{\C}{\mathbb{C}}

\newcommand{\R}{\mathbb{R}}
\newcommand{\h}{\mathbb{H}}
\newcommand{\kh}{\mathbb{K}}
\newcommand{\W}{\Omega}
\newcommand{\w}{\omega}
\newcommand{\N}{\mathbb{N}}
\newcommand{\T}{\mathbb{T}}
\newcommand{\Z}{\mathbb{Z}}
\newcommand{\A}{\mathbb{A}}
\newcommand{\VK}{Varopoulos-Kaijser }

\hypersetup{
  colorlinks,
  citecolor=blue,
  linkcolor=blue,
  urlcolor=blue}
  
\title{Interpolation Problems and the von-Neumann Inequality}
\author{Rajeev Gupta}

\begin{document}

\frontmatter

\begin{titlepage}

  \begin{center}

\textbf{\LARGE  {The Carath\'{e}odory-Fej\'{e}r Interpolation Problems \\and the von-Neumann Inequality}  
	}\\
	\vspace{0.5in}
	{\large  A Dissertation \\
	submitted in partial fulfilment of the requirements \\
	\vspace{0.04in}
	for the award of the degree of} \\
	\vspace{0.1in}
	{{\Large \bsifamily 
	Doctor of Philosophy}}  \\
	\vspace{.5in}
	{\large by}\\
	\vspace{0.05in}
	{\large Rajeev Gupta}\\
	\vspace{1in}
\includegraphics[scale=.3]{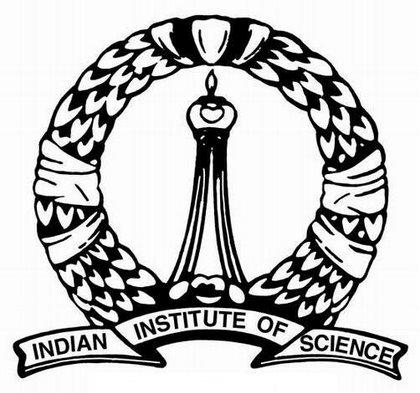} \\
\vspace{.2in}
{\large Department of Mathematics \\
Indian Institute of Science \\
Bangalore - 560012 \\
July 2015 \\}
\end{center}

\end{titlepage}

\newpage
\thispagestyle{empty}
{~}
\newpage

\thispagestyle{empty}
\vspace*{8cm}
\begin{center}
\begin{dedication}
TTo\\
my parents\\
and\\
my elder brother
\end{dedication}
\end{center}

\chapter*{Declaration}

\vspace*{.80in}

I hereby declare that the work reported in this thesis is entirely
original and has been carried out by me under the supervision of
Prof. Gadadhar Misra at the Department of Mathematics, Indian
Institute of Science, Bangalore. I further declare that this work
has not been the basis for the award of any degree, diploma,
fellowship, associateship or similar title of any University or
Institution.
\vspace*{.5in}

\begin{flushright}
Rajeev Gupta\\
S. R. No. 6910-110-101-07875
\end{flushright}
Indian Institute of Science,\\
Bangalore,\\
July, 2015.
\smallskip

\begin{flushright}
 Prof. Gadadhar Misra \\
 (Research advisor)
\end{flushright}

\chapter*{Acknowledgement}

I take this opportunity to thank my PhD supervisor, 
Prof. Gadadhar Misra, for his enthusiasm, motivation, immense 
patience and continuous support throughout my 
stay at Indian Institute of Science. 
He bestowed an ample amount of time for many long and 
fruitful discussions, 
despite being the Chairman of the department.
His joy and enthusiasm towards research was contagious,
thereby encouraging me to explore the subject with aggression
and novelty. 
A special thanks is due for his meticulous scrutiny of 
this thesis draft and a major reorganization of the same.
His support and guidance not only helped me grow as a researcher 
but also shaped me into a better human being.

\medskip

I would like to thank Prof. Pisier for some 
very insightful discussions during IWOTA 2013. 
I would 
like to thank him and Prof. Paulsen for some excellent lectures
at the Institute for Mathematical Sciences, Chennai.

\medskip
I would like to thank the faculty members at
the Department of Mathematics in IISc, especially
for offering such a wide variety of courses. The courses 
offered in topics in operator theory by Prof. Gadadhar Misra
have benefited me immensely. I would also like to thank Prof.
Manjunath Krishnapur for allowing me to simultaneously attend
and assist in the probability theory course.

\medskip

The primary reason for my embracing
research in mathematics is the summer program MTTS.
I would like to mention a special thanks to Prof.
Kumaresan for organizing the same.

\medskip 

I am extremely thankful to the administrative staff at the
Department of Mathematics for seeing me through the many
complicated formalities that have constantly risen during
my stint at IISc.

\medskip The many discussions in mathematics with 
Kartick, Nanda,  Pranav, Ramiz, Samya, Tulasi and 
Vikram have been very rewarding. A heartfelt thanks 
to all of them. Kartick, Pranav and Sayani deserve 
a special mention for pinpointing the mistakes in 
latex that emerged while penning this thesis.

\medskip

My life at IISc has been enjoyable, pleasant
and colourful because of my coexistence with 
folks like Arpan, Atreyee, Avijit, Bidyut, Dr. DD, 
Kartick, Mizanur, Nirupam, Pranav, Ramiz, 
Sayan, Sayani, Tulasi and Vikram. It has only been
more cheerful because of friends like
Amardeep, Bidhan, Manojit, Nanda, 
Rakesh, Samrat, Somnath, Sameer and Soumitra. 
It would be unfair not to mention 
Jiggy and Paddhu for they have made my
life here very charming.

\medskip
I acknowledge the National Board for Higher 
Mathematics (NBHM), Government of India for 
providing me the funding and fellowship to 
pursue research at IISc.  

\medskip

This acknowledgement would be incomplete without 
thanking my family for the support they have 
provided me all through. No words can thank my father 
and my elder brother for being such an immovable
shield for me at times of difficulty. Without their love, 
hard work and encouragement, this thesis would not 
have been possible.  

\chapter*{Abstract}
The validity of the von-Neumann inequality for commuting $n$ - tuples of $3\times 3$ matrices remains open for $n\geq 3$.  
We  give a partial answer to this question, which is used  to obtain a necessary condition for the Carath\'{e}odory-Fej\'{e}r 
interpolation  problem on the polydisc $\D^n.$ In the special case of $n=2$ (which follows from Ando's theorem as well), 
this necessary condition is made explicit. 

An alternative approach to the Carath\'{e}odory-Fej\'{e}r interpolation  problem, in the special 
case of $n=2,$ adapting a theorem of  Kor\'{a}nyi and Puk\'{a}nzsky is given. As a consequence, 
a class of polynomials are isolated for which a complete solution to the\\
Carath\'{e}odory-Fej\'{e}r 
interpolation problem is easily obtained. A natural generalization of the Hankel operators on the 
Hardy space of $H^2(\mathbb T^2)$ then becomes apparent. Many of our results remain valid for any 
$n\in \mathbb N,$ however, the computations are somewhat cumbersome for $n>2$ and are omitted. 

The inequality $\lim_{n\to \infty}C_2(n)\leq 2 K^\C_G$, where $K_G^\C$ is the complex Grothendieck constant and 
\[C_2(n)=\sup\big\{\|p(\boldsymbol T)\|:\|p\|_{\D^n,\infty}\leq 1, \|\boldsymbol T\|_{\infty} \leq 1 \big\}\]
is due to Varopoulos. Here the supremum is taken over all complex polynomials $p$ in $n$ variables of degree 
at most $2$ and commuting $n$ - tuples $\boldsymbol T:=(T_1,\ldots,T_n)$ of contractions. We  show that
\[\lim_{n\to \infty}C_2(n)\leq \frac{3\sqrt{3}}{4} K^\C_G\] obtaining a slight improvement in the inequality 
of Varopoulos.  

We show that the normed linear space $\ell^1(n),$  $n>1,$ has no isometric embedding into $k\times k$ complex matrices for any $k\in \N$
and discuss several infinite dimensional operator space structures on it. 

\tableofcontents

\mainmatter
\chapter{Introduction}
The fundamental inequality of von-Neumann saying that 
$\|T\|\leq 1$ if and only if $\|p(T)\|\leq \|p\|_{\D,\infty}$ 
for any polynomial $p,$ 
has lead to several new developments in modern operator theory. 
This inequality follows from the Sz.-Nagy dilation theorem, 
indeed, it is equivalent to it. 
The homomorphisms 
$\rho: \C[Z]\to \mathcal{B}(\h),$ 
where $\C[Z]$ is the polynomial ring 
and $\mathcal{B}(\h)$ is the algebra of 
bounded linear operators, on some complex separable Hilbert space $\h,$ 
are clearly in one-one correspondence 
with operators $T$ in $\mathcal{B}(\h).$ 
Thus given $T\in \mathcal{B}(\h),$ 
one defines the homomorphism $\rho_{\!_T}(p)=p(T)$ 
and conversely given $\rho,$ 
one may set $T:=\rho(z).$ 
An equivalent formulation of the von-Neumann inequality is the statement: 
A homomorphism $\rho$ is contractive, that is, 
$\|\rho(p)\|\leq \|p\|_{\D,\infty}$ 
for all $p\in \C[Z]$ 
if and only if $\|T\|:=\|\rho(z)\|\leq 1.$ 

The Sz.-Nagy dilation theorem for a homomorphism $\rho$ is the statement:

The homomorphism $\rho$ is contractive if and only if 
there exists a Hilbert space $\kh$ containing $\h$ 
and a $*$-homomorphism $\tilde{\rho}:C(\T)\to \mathcal{B}(\kh)$ 
such that 
\[P_{\h}\tilde{\rho}(p)_{|_{\h}}=\rho(p),\, p\in \C[Z].\]
Since $\sigma(\tilde{\rho}(z))\subset \T$ and $\tilde{\rho}$ is a $*$-homomorphism, it follows that 
\[\|\rho(p)\|\leq \|P_{\h}\tilde{\rho}(p)_{|_{\h}}\|
\leq \|\tilde{\rho}(p)\|\leq \|p\|_{\T,\infty}\leq \|p\|_{\D,\infty},\]
which is the von-Neumann inequality. 
The existence of the $*$-homomorphism 
$\tilde{\rho}$ can be obtained, among several other methods, 
following the Schaffer construction of the (unitary power) dilation. 

Over the past five or six decades, 
the question of the von-Neumann inequality and 
the Sz.-Nagy dilation has been studied vigorously. 
In explicit terms, these two questions are stated below. Let $\C[Z_1,\ldots,Z_n]$ denote the ring of complex valued polynomials in $n$ variables.
\begin{itemize}
\item[(1)] If $T_1,\ldots,T_n$ is a tuple of commuting contractions, 
does it follow that $\|p(T_1,\ldots,T_n)\|\leq \|p\|_{\D^n,\infty}$ for any polynomial $p\in \C[Z_1,\ldots,Z_n]?$ 
\item[(2)] If $\rho$ is a contractive homomorphism, that is, 
$\|\rho(p)\|\leq \|p\|_{\D^n,\infty},$ 
$p\in \C[Z_1,\ldots,Z_n],$ 
does it follow that 
$\rho(p)=P_{\h}\tilde{\rho}(p)_{|_{\h}}$ 
for some $*$-homomorphism 
$\tilde{\rho}:C(\T^n)\to \mathcal{B}(\kh),$ 
where $\kh$ is some Hilbert space containing $\h?$   
\end{itemize}
As is well known, via the foundational work of Arveson \cite{WA1,WA2}, 
the second question is equivalent 
to the complete contractivity of the homomorphism $\rho$: 
\[\|\rho(P)\|\leq \|P\|_{\D^n,\infty}^{\rm op},
\mbox{ where }P=\big(\!\!\big(p_{ij}\big)\!\!\big),\, p_{ij}\in \C[Z_1,\ldots Z_n]\]
\[\mbox{ and }\|P\|_{\D^n,\infty}^{\rm op}=\sup\big\{\big\|\big(\!\!\big(p_{ij}\big)\!\!\big)\big\|:z\in \D^n\big\}.\]  

If $n=1,$ as we have seen, 
an affirmative answer to both of these questions are obtained via the von-Neumann inequality and the Sz.-Nagy dilation theorem. Indeed, an affirmative answer to either of these questions 
gives an affirmative answer to the other. 
This continues to be the case even if $n=2,$ 
thanks to the celebrated theorem of Ando. 
However for $n=3,$ 
examples due to Varopoulos-Kaijser and Parrott 
show that neither $(1)$ nor $(2)$ has an affirmative answer. 

Varopoulos, in a second paper, showed that 
\begin{equation}\label{Varopoulos bounds}
K_{G}^{\C}\leq \sup\|p(T_1,\ldots,T_n)\|\leq 2 K_{G}^{\C},
\end{equation}
where $K_{G}^{\C}$ denote the complex Grothendieck constant 
and 
supremum is over all $n\in\N$, 
tuples of commuting contractions $T=(T_1,\ldots,T_n)$ 
and polynomials $p$ of degree 2 with 
$\|p\|_{\D^n,\infty}\leq 1.$ 
He lamented if $2$ appearing 
on the right hand side of this inequality, can be removed. 
The examples due to Varopoulos leaves the following question (cf. \cite[Chapter 1, Page 24]{Pisier} open: 
\begin{Qn}\label{unbounded}
For a fixed $n\in\N,n\geq 3$ and $M>0,$  
does there exist a commuting contractive $n$ - tuple of operators 
$T_1,\ldots,T_n$ 
such that 
\[\sup_{p\in\C[Z_1,\ldots,Z_n]}\frac{\|p(T_1,\ldots,T_n)\|}{\|p\|_{\D^n,\infty}}>M.\]
\end{Qn}

A class of homomorphism, 
which include the example of Parrott 
were studied further in \cite{GMParrott,VP}, 
where the question of 
contractivity vs. complete contractivity of these homomorphism 
was reduced to certain linear maps. The reason for this lies in showing that the contractivity(respectively complete contractivity) of these homomorphisms is determined by their restriction to the linear polynomials. To explain this in some detail and for use throughout this thesis, we introduced the following notations,

Let $\Omega$ be a bounded and polynomially convex domain in $\mathbb{C}^n$. 
Let $\mathcal{A}(\Omega)$ be the completion of 
$\C[Z_1,\ldots,Z_n]$ with respect to norm $\|\cdot\|_{\Omega,\infty},$ 
where 
$\|f\|_{\W,\infty}=\sup\{|f(\w)|:\w\in\W\}$ 
for every $f\in\C[Z_1,\ldots,Z_n].$ Let $\mathcal{P}_k[Z_1,\ldots,Z_n]$ denote the set of all polynomials in $n$ variables of degree at most $k.$ When number of variables is clear from the context we omit the variables $Z_1,\ldots,Z_n.$ Let $H^\infty(\W)$ denote the set of all complex valued bounded holomorphic functions on $\W$ and $\D$ be the unit disc in $\C$. For each $\w\in\W,$ set
\[H^\infty(\W,\D)=\big\{f\in H^\infty(\W):\|f\|_{\W,\infty}\leq 1\big\}\mbox{ and }H^{\infty}_{\omega}(\Omega,\D)=\big\{f\in H^\infty(\Omega,\D):f(\omega)=0\big\}.\] 
Let $T=(T_1,\ldots,T_n)$ be a tuple of bounded operators on 
some fixed separable Hilbert space $\mathbb{H}$ and 
$\omega=(\omega_1,\ldots,\omega_n)$ be a fixed point in $\Omega$. 
Define the Parrott homomorphism to be the map 
$\rho_{\!_T}^{(\w)}:H^{\infty}(\Omega)\to \mathcal{B}(\mathbb{H}\oplus\mathbb{H})$
given by the formula 
\[\rho_{\!_T}^{(\w)}(f)=
\left(
\begin{array}{cc}
f(\omega)I & D f(\omega)\cdot T\\
0 & f(\omega)I\\
\end{array}
\right),
\]
where $D f(\w)=
\left(\frac{\partial}{\partial z_1}f(\w),\ldots,\frac{\partial}{\partial z_m}f(\w)\right)$ and $D f(\omega)\cdot T=\frac{\partial}{\partial z_1}f(\w)T_1+\cdots+\frac{\partial}{\partial z_m}f(\w)T_m.$ 

The following lemma, called ``the zero lemma'', and several of its variants involving functions defined on domains in $\mathbb C^n$ and taking values in $k\times k$ matrices have been proved in \cite{Gmthesis, GMSastry, GMParrott, VP}. The proof below follows closely the one appearing in \cite{VP}.
\begin{lem}
	A Parrott homomorphism $\rho_{\!_T}^{(\w)}$ is contractive  
	if and only if $\|\rho_{\!_T}^{(\w)}(f)\|\leq 1$ 
	for all $f\in H^{\infty}_{\omega}(\Omega,\D)$.
\end{lem}
\begin{proof}
Let us assume that $\|\rho_{\!_T}^{(\w)}(f)\|\leq 1$ 
for all $f\in H^{\infty}_{\omega}(\Omega,\D).$ 
Suppose $g:\W\to\D$ is an analytic function 
and $\phi$ is the automorphism of $\D$ 
mapping $g(\w)$ to $0$. 
Then $\phi\circ g$ is an analytic map 
from $\W$ to $\D$ with $(\phi\circ g)(\w)=0$, 
therefore $\|\rho_{\!_T}^{(\w)}(\phi\circ g)\|\leq 1$. 
Now by von-Neumann's inequality 
we have 
$\|\phi^{-1}(\rho_{\!_T}^{(\w)}(\phi\circ g))\|\leq 1,$ 
which is equivalent to $\|\rho_{\!_T}^{(\w)}(g)\|\leq 1.$ 
Hence $\rho_{\!_T}^{(\w)}$ is a contraction. 
The converse is trivially true.
\end{proof}
Now, let us assume that $\Omega$ is a unit ball in $\mathbb{C}^n$ 
with respect to some norm and $\omega=0.$  
Let
\[\mathcal{L}[Z_1,\ldots,Z_n]=\{a_1 z_1+\cdots+a_n z_n:a_i\in\C~\forall i=1\mbox{ to } n\}\]
be the set of all linear polynomials in $m$ variables.  
Let $\rho_{\!_T}$ denote the homomorphism $\rho^{(0)}_{\!_T}.$ 
\begin{thm}
	 For the Parrott homomorphism $\rho_{\!_T},$ 
	 we have 
	\[\sup\left\{\|\rho_{\!_T}(\ell)\|:
	\ell\in\mathcal{L}[Z_1,\ldots,Z_n],\|\ell\|_{\W,\infty}\leq 1\right\}
	=\sup\left\{\|\rho_{\!_T}(f)\|:f\in H^\infty_0(\W,\D)\right\}.\]
\end{thm}
\begin{proof} 
If $f\in  H^\infty_0(\W,\D)$ is a holomorphic function, then from the Schwarz lemma \cite[Theorem 8.1.2]{RW}, 
$\ell:=D f(0)$ maps $\Omega$ in to the disc of radius 
$\|f\|_{\W,\infty}$ and thus $\|\ell\|_{\W,\infty}\leq \|f\|_{\W,\infty}$. 
From the definition of $\rho_{\!_T}$, 
we have $\|\rho_{\!_T}(\ell)\|=\|\rho_{\!_T}(f)\|$. 
Therefore
\[\frac{\|\rho_{\!_T}(\ell)\|}{\|\ell\|_{\Omega,\infty}}
\geq \frac{\|\rho_{\!_T}(f)\|}{\|f\|_{\Omega,\infty}}\]
and hence 
\begin{align*}
	\sup\left\{\|\rho_{\!_T}(\ell)\|:
	\ell\in\mathcal{L}[Z_1,\ldots,Z_n],\|\ell\|_{\W,\infty}\leq 1\right\}
	\geq\sup\left\{\|\rho_{\!_T}(f)\|:f\in H^\infty_0(\W,\D)\right\}.
\end{align*}
The other inequality is obvious. 
\end{proof}
This theorem says that if we wish to establish only the contractivity of Parrott homomorphism $\rho_{\!_T},$ it is enough to restrict $\rho_{\!_T}$ to the linear polynomials. 

It also says that if $\W=\D^n,$ then the Parrott homomorphisms $\rho_{\!_T}$ are contractive if and only if $T_1,\ldots,T_n$ are contractions. This follows from the Schwarz lemma (cf. \cite[Theorem 8.1.2]{RW}):
$$
\big \{Df(0)\in \mathbb C^n: f \in H_0^\infty(\mathbb D^n, \mathbb D)\big \} = \big \{\ell: \ell(\mathbb D^n) \subseteq \mathbb D\big \}.
$$
Consequently, these homomorphisms can not be used to answer the Question \ref{unbounded}.

We therefore investigate the homomorphism induced by the commuting triple of contractions $T_1,T_2,T_3$ given by Varopoulos and Kaijser with the property $\|p(T_1,T_2,T_3)\|>\|p\|_{\D^3,\infty}.$ This leads to a natural definition of a class of operators which we call Varopoulos operator of type I and II. We investigate the answer to the Question \ref{unbounded} assuming the homomorphism $\rho_{\!_T}$ is induced by $T,$ a tuple of these operators. It is useful to recall that Varopoulos, in a second paper, proved the following.
\[K_{G}^{\C}\leq \sup\|\rho_{{\!_T}{|_{\mathcal{P}_2}}}\|\leq 2 K_{G}^{\C},\]
where $K_{G}^{\C}$ denote the complex Grothendieck constant 
and 
supremum is over all $n\in\N$ and  
tuples of commuting contractions $T=(T_1,\ldots,T_n).$ Thus it is natural to ask if $\sup_{\|T\|_{\infty}\leq 1}\|\rho_{{\!_T}{|_{\mathcal{P}_2}}}\|,$ where $\|T\|_{\infty}:=\max\{\|T_1\|,\ldots,\|T_n\|\},$ is closer to the universal constant $K^{\C}_{G}$ of Grothendieck than indicated by the inequalities \eqref{Varopoulos bounds}. We show that inequality on the right can be considerably improved. 
    
Let $\h$ be a separable Hilbert space and 
$\{e_j\}_{j\in\mathbb{N}}$ be a set of orthonormal basis for $\h$. 
For any $x\in\h$, 
define $x^{\sharp}:\h\to\C$ by $x^{\sharp}(y)=\sum_{j}x_jy_j,$ 
where $x=\sum x_je_j$ and $y=\sum y_je_j$. 
For $x,y\in\h$, we set $\left[x^{\sharp},y\right]=x^{\sharp}(y).$
\begin{defn}[Varopoulos Operator of Type I (\sf{V\,I})]
	Let $\mathbb{H}$ be a separable Hilbert space. 
	For $x,y\in\mathbb{H}$, 
	define $T_{x,y}:\C \oplus \h \oplus \C \to \C \oplus \h \oplus \C$ by
	\[T_{x,y}=
	\left(
	\begin{array}{ccc}
		0 & x^{\sharp} & 0\\
		0 & 0 & y\\
		0 & 0 & 0\\
	\end{array}
	\right).\]
	The operator $T_{x,y}$ will be called Varopoulos operator of type I 
	corresponding to the pair of vectors $x,y$. 
	If $x=y$ then $T_{xy}$ will simply be denoted by $T_x$.
\end{defn}
\begin{defn}[Varopoulos Operator of Type II and of order $k$ (\sf{V\,II  of order k})]
	Let $\mathbb{H}$ be a separable Hilbert space. 
	For $X\in\mathcal{B}(\mathbb{H})$, let
	\[T_X=
	\left(
	\begin{array}{ccccc}
		0 & X & 0 & \cdots &0\\
		0 & 0 & X & \cdots & 0\\
		\texorpdfstring{\vdots}{TEXT} & \vdots & \vdots & \ddots & \vdots\\
		0 & 0 & 0 & \cdots & X\\
		0 & 0 & 0 & \cdots & 0\\
	\end{array}
	\right)\]
	be the operator in $\mathcal B(\mathbb H\otimes C^{k+1}).$
In analogy with the work of Varopoulos \cite{V2}, operators of the form $T_X,$ $X\in \mathcal B(\mathbb H),$ are called 	Varopoulos operator of type II and of order $k.$ 
\end{defn}
Let $\W$ be a bounded domain in $\C^n$ and $\w\in \W$ be a fixed but arbitrary point. 
As before let $\rho_{xy}^{(\w)}$ (respectively $\mu_{X}^{(\w)}$) denote the induced homomorphism on $H^\infty(\W),$ corresponding to a tuple of commuting contractions $\w_1I+T_{x_1y_1},\ldots,\w_nI+T_{x_ny_n}$ (respectively $\w_1I+T_{X_1},\ldots,\w_nI+T_{X_n}$), which is defined as following. 
\[\rho_{xy}^{(\w)}(f)=\left(
 \begin{array}{ccc}
	f(\omega) & D f(\omega)\cdot x^{\sharp} & \frac{1}{2}D^2f(\omega)\cdot A_{x,y}\\
	0 & f(\omega)I & D f(\omega)\cdot y\\
	0 & 0 & f(\omega)\\
  \end{array}
\right)\]
for $f\in H^\infty(\W),$ where $x=(x_1,\ldots,x_n),y=(y_1,\ldots,y_n),x^\sharp=(x_1^\sharp,\ldots,x_n^\sharp)$ and $A_{x,y}=\big(\!\!\big([x_i^\sharp,y_j]\big)\!\!\big).$ As the definition of $\rho_{xy}^{(\w)}(f)$ includes only the terms of order at most $2$ from the Taylor series expansion of $f,$ therefore it is quite natural to ask the following question.
\begin{Qn}
We ask if the contractivity of $\rho_{xy}^{(\w)}$ on $H^\infty(\W)$ is equivalent to contractivity of the restriction to the polynomials of degree at most $2$.
\end{Qn}
Clearly to answer this question, one must first answer a related question generalizing the Carath\'{e}odory-Fej\'{e}r interpolation problem, namely: 
Given any polynomial $p$ in $n$ variables of degree $2$ with $p(0)=0,$ find necessary and sufficient conditions on the coefficients of $p$ to  ensure the existence of a holomorphic function $h$ defined on the polydisc $\mathbb D^n$ with $h^{(k)}(0)=0$ for all multi indices $k$ of length at most $2,$ such that $f:=p+h$ maps the polydisc $\mathbb D^n$ to the unit disc $\mathbb D.$   

However the absence of an explicit criterion, in spite of several  results which have been obtained recently \cite{FFE,EPP,Woe,HMLZ}, for the solution to this problem for $n>1$  makes it difficult to answer this question.

We combine a theorem due to Kor\'{a}nyi and Puk\'{a}nszky giving a criterion for determining if the real part of a holomorphic function defined on the polydisc $\mathbb D^n$ is positive with a theorem due to Parrott
to find a solution to the Carath\'{e}odory-Fej\'{e}r interpolation problem. We state these two theorems below. 

\begin{thm}[Kor\'{a}nyi-Puk\'{a}nszky Theorem]
If the power series 
	$\sum_{\alpha \in \mathbb{N}_{0}^{n}}a_{\alpha}z^{\alpha}$ represents a holomorphic function $f$ on the polydisc $\mathbb D^n,$ then   $\Re(f(z))\geq 0$ for all $z\in \mathbb{D}^n$ 
	if and only if the map 
	$\phi:\mathbb{Z}^n\to \mathbb{C}$ 
	defined by 
	\begin{eqnarray*}
		\phi (\alpha )=\left\{
		\begin{array}{ll}
		      2\Re a_{\alpha} & \mbox{\rm if }\alpha =0\\
		      a_{\alpha} & \mbox{\rm if }\alpha > 0\\
		      a_{-\alpha} & \mbox{\rm if }\alpha < 0\\
		      0 & \mbox{\rm otherwise }\\
		\end{array} 
		\right.
	\end{eqnarray*}
is positive, that is, the $k\times k$ matrix $\big (\!\!\big ( \phi(\scriptstyle{m_i-m_j}) \big )\!\!\big )$ is non-negative definite for every choice of $m_1, \ldots, m_k\in \mathbb Z^n.$  
\end{thm}
Let $f:\mathbb D^n \to \mathbb D$ be  a holomorphic function and $\chi$ be the Cayley map of the unit disc to the right half plane. Then in the matricial representation of $\phi_{\chi \circ f}$ with respect to the usual order in $\mathbb Z^2,$ it is not easy to isolate the coefficients of $f.$ We introduce a new order, to be called, the {\tt D-slice} ordering:
\begin{defn}[D-slice ordering]
Suppose $(x_1,y_1)\in P_l$ and $(x_2,y_2)\in P_m$ 
are two elements in $\mathbb{Z}^2$. 
Then 
\begin{enumerate}
	\item If $l=m,$ then $(x_1,y_1) < (x_2,y_2)$ is determined by the lexicographic ordering on $P_l\subseteq \mathbb Z^2$ and    
	\item  if $l < m$ (resp., if $l > m$), then $(x_1,y_1) < (x_2,y_2)$ (resp., $(x_1,y_1) > (x_2,y_2)$).  
\end{enumerate}
\end{defn}

The matricial representation of the function $\phi_{\chi \circ f}$ is then in the form of a block Toeplitz matrix with respect to the D-slice ordering. 

\begin{thm}[Parrott's Theorem]
	For $i=1,2$, let $\h_i,\,\mathbb{K}_i$ be Hilbert spaces and 
	$\h=\h_1\oplus \h_2,\,\mathbb{K}=\mathbb{K}_1\oplus \mathbb{K}_2.$ 
	If 
	\[
	\left(
	\begin{array}{c}
		A\\
		C\\
	\end{array}
	\right):\h_1\to \mathbb{K} \mbox{ \rm and } \big(C\,\,\,\,\,\,D\big):\h\to \mathbb{K}_2
	\]
	are contractions, then there exists $X\in\mathcal{B}(\h_2,\mathbb{K}_1)$ such that 
	$\left(
	\begin{smallmatrix}
		A & X\\
		C & D\\	
	\end{smallmatrix}
	\right):\h\to \mathbb{K}
	$ is a contraction.
\end{thm}
In this theorem, all the choices for $X$ are given by the formula:
\[(I-ZZ^*)^{1/2}V(I-Y^*Y)^{1/2}-ZC^*Y,\]
where $V$ is a contraction and $Y,\,Z$ are determined from the formulae:
\[D=(I-CC^*)^{1/2}Y,\,\, A=Z(I-C^*C)^{1/2}.\]

Our method gives only a (explicit) necessary condition for the existence of a solution to the Carath\'{e}odory-Fej\'{e}r interpolation problem in general. (Surprisingly, for the case of the bi-disc, this necessary condition is exactly the condition for contractivity of the homomorphisms induced by the Varopoulos operators.)

It also gives an algorithm for constructing a solution whenever such a solution exists. The algorithm involves finding, inductively, polynomials $p_n$ of degree at most $n$ such that a certain block Toeplitz operator, made up of multiplication by these polynomials is contractive. A solution to the Carath\'{e}odory-Fej\'{e}r interpolation problem exists if and only if this process is completed successfully.

If $n=1$ and the necessary condition we have obtained is met, then the algorithm completes successfully and produces a solution to the  Carath\'{e}odory-Fej\'{e}r interpolation problem. Thus in this case, we fully recover the solution to the Carath\'{e}odory-Fej\'{e}r interpolation problem.

We also isolate a class of polynomials for which our necessary condition is also sufficient. This is verified using the deep theorem of Nehari reproduced below (cf. \cite[Theorem 15.14, page 194]{NY}). 

Let $H^2(\T)$ denote the Hardy space, a closed subspace of $L^2(\T)$. 
Let $P_{-}$ denote the orthogonal projection 
of $L^2(\T)$ onto $L^2(\T)\ominus H^2(\T).$
\begin{defn}[Multiplication Operator]
	For $\phi\in L^{\infty}(\T),$ 
	we define multiplication operator 
	$M_{\phi}:L^2(\T)\to L^2(\T)$ 
	by $M_{\phi}(f)=\phi\cdot f$.  
\end{defn}
Since $\phi \cdot f\in L^2(\T)$ 
for any $\phi\in L^{\infty}(\T)$ and 
$f\in L^2(\T)$ 
therefore $M_{\phi}$ is well defined 
for all $\phi\in L^{\infty}(\T)$. 
Also $\|M_{\phi}\|=\|\phi\|_{\infty}$ (cf. Theorem 13.14 in \cite{NY}).

\begin{defn}[Hankel Operator]
	Let $\phi\in L^{\infty}(\T)$. 
	Hankel operator corresponding to $\phi$ 
	is the operator $P_{-}\circ M_{\phi}\vert_{H^2(\T)}$. 
	It is denoted by $H_{\phi}$.
\end{defn}
 
\begin{thm}[Nehari's Theorem]
	If $\phi\in L^{\infty}(\T)$ and 
	$H_{\phi}$ is the corresponding Hankel operator, 
	then
	\[\inf\left\{\|\phi-g\|_{\T,\infty}:g\in H^{\infty}(\T)\right\}
	=\|H_{\phi}\|_{op}.\]
\end{thm} 

All this is done for the bi-disc $\mathbb D^2$ with the understanding that these computations will go through for the polydisc $\mathbb D^n.$ 
Similarly, while we have discussed the Carath\'{e}odory-Fej\'{e}r interpolation problem for polynomials of degree at most $2,$ again, our methods remain valid for an arbitrary polynomial.  

What follows is a detailed description of the results proved in this thesis.

Following \cite{V2} and \cite[Page 24]{Pisier}, in chapter 2, we define the quantities:
\[C_k(n)=\sup\big\{\|p(\boldsymbol T)\|:\|p\|_{\D^n,\infty}\leq 1, p \mbox{~\rm is of degree at most~} k,\,\|\boldsymbol T\|_{\infty} \leq 1 \big\}\]
and 
\begin{equation}
C(n)= \lim_{k\to \infty} C_k(n), 
\end{equation}
where $\|\boldsymbol T\|_{\infty}=\max\big\{\|T_1\|,\ldots,\|T_n\|\big\}.$ 
In this notation, it follows from the von-Neumann inequality and Ando's theorem that $C(1), C(2) = 1.$ 
Also $C_2(3) >1,$ thanks to the example of Varopoulos and Kaijser \cite{V1} involving an (explicit) homogeneous polynomial of degree $2.$ 
Following this,  in the paper \cite{V2}, Varopoulos proves the inequality \eqref{Varopoulos bounds}. Consequently, the limit of the non-decreasing sequence $C_2(n)$ must be bounded below by  $K_G^\mathbb C.$ We show that $C_2(3) \geq 1.2$ by means of explicit examples. We were hoping to improve this  
inequality obtained earlier by Holbrook \cite{HJ} since our methods  appear to be somewhat more direct. In view of the known lower bound 
for $\lim_{n\to \infty}C_2(n)$ in \eqref{Varopoulos bounds}, we hoped that 
the lower bound for $C_2(3)$ itself will be closer to $K_G^\mathbb C.$ 
In this chapter, we also show that $\|p(T_1,\ldots , T_n)\| \leq \|p\|_{\mathbb D^n, \infty}$ for any $n$ commuting contractions of the form
\begin{equation}\label{von3}
\Big \{\Big ( \,\,\Big ( \begin{smallmatrix}
	\w_1 & \alpha_1 & 0\\
	0 & \w_1 & \beta_1\\
	0 & 0 & \w_1\\
\end{smallmatrix} \Big ), \ldots, \Big ( \begin{smallmatrix}
	\w_n & \alpha_n & 0\\
	0 & \w_n & \beta_n\\
	0 & 0 & \w_n\\
\end{smallmatrix} \Big ) \,\,\Big ): 
\alpha_i \beta_j = \alpha_j \beta_i, 1\leq i,j \leq n,\, \w:=(\w_1, \ldots , \w_n)\in \mathbb D^n \Big\},
\end{equation}
after assuming that $|\alpha_i| = |\beta_i|,$ $1\leq i\leq n.$ This is interesting considering that the von-Neumann inequality is valid for any commuting $n$ - tuple of $2\times 2$ contractions \cite{GMPati,Agler} and fails for $4\times 4$ contractions \cite{HJ}. 
As a corollary of the von-Neumann inequality for a subclass of the operators defined in \eqref{von3}, we get the following necessary condition for the Carath\'{e}odory-Fej\'{e}r interpolation problem for the polydisc $\D^n.$
\begin{thm}\label{necessary}
Let $p$ be a polynomial in $n$ variables of degree $2$ such that $p(0)=0.$ There exists a holomorphic function $q,$ defined on polydisc $\D^n,$ with $q^{(k)}(0)=0,\,|k|\leq d$ such that $\|p+q\|_{\infty}\leq 1$ only if 
\[\sup_{\|\alpha\|_{\infty}\leq 1}\Big\{\Big|\frac{D^2p(0)\cdot A_\alpha}{2}\Big|+\big|Dp(0)\cdot\alpha\big|^2\Big\}\leq 1,\]
where $\alpha=(\alpha_1,\ldots,\alpha_n),$ $A_{\alpha}=\big(\!\!\big(\alpha_i\alpha_j\big)\!\!\big),$ $D^2p(0)\cdot A_{\alpha}=\sum\frac{\partial^2 p}{\partial{z_i}\partial{z_j}}(0)\alpha_i\alpha_j$ and $Dp(0)\cdot \alpha=\sum\frac{\partial p}{\partial{z_i}}(0)\alpha_i.$
\end{thm}  
We also prove the following theorem, giving a considerable improvement on the upper bound \eqref{Varopoulos bounds} previously obtained in \cite{V2}. 
\begin{thm}
$\lim_{n\to \infty} C_2(n) \leq \frac{3\sqrt{3}}{4} K_G^\mathbb C.$
\end{thm}
Finally, in this chapter, we investigate in some detail, the contractivity of the homomorphisms 
$\rho_{x,y}$ induced by the Varopoulos operators of type I ({\sf{V\,I}}) and we come up with the same inequality as in the Theorem \ref{necessary} but for $n=2.$ 

In chapter 3, we study the homomorphisms induced by tuples of commuting Varopoulos operators of type II and order $2$ and solve the extremal problem (indeed a more general extremal problem obtained in the study of these homomorphism) occurring in the Theorem \ref{necessary} but for $n=2.$ We define 
\[p_1(z)=\frac{\partial }{\partial z_1}f(0)+\frac{\partial }{\partial z_2}f(0)z\mbox{ and }
p_2(z)=\frac{1}{2}\frac{\partial^2}{\partial z_1^2}f(0)+\frac{\partial^2}{\partial z_1\partial z_2}f(0)z+\frac{1}{2}\frac{\partial^2}{\partial z_2^2}f(0)z^2 
\]
for a holomorphic function $f$ in two variables and  we prove the following.
\begin{thm}
For $f\in{\rm H}^\infty_{0}(\D^2,\D)$, we have, 
$$\sup\limits_{\|X\|_{\infty}\leq 1}\left\|\mathscr{T}\left(D f(0)\cdot  X,
\frac{1}{2}D^2f(0)\cdot A_{ X}\right)\right\|
=\|\mathscr{T}(M_{p_1},M_{p_2})\|,$$
where $X=(X_1,X_2)$ is pair of commuting operators, $\|X\|_{\infty}=\max\{\|X_1\|,\|X_2\|\},$ $A_X=\big(\!\!\big(X_iX_j\big)\!\!\big)$ and $\mathscr T(A_1,A_2)=\Big(\begin{array}{cc}
A_1 & A_2\\
0 & A_1\\
\end{array}\Big)$ for any $A_1,A_2\in \mathcal{B}(\h).$
\end{thm}
In the process of solving the extremal problem occurring in this theorem, we reprove the von-Neumann inequality and Ando's theorem for a commuting pair of Varopoulos operators of type II.

In chapter 4, we give an alternative for solving the Carath\'{e}odory-Fej\'{e}r interpolation problem after adapting a theorem of Kor\'{a}nyi and Puk\'{a}nszky. This approach is important as it is independent of the commutant lifting theorem, whereas the method in chapter 3 strongly depends on it. For a polynomial $p$ in two variables we define 
\[p_1(z)=\frac{\partial }{\partial z_1}p(0)+\frac{\partial }{\partial z_2}p(0)z\mbox{ and }
p_2(z)=\frac{1}{2}\frac{\partial^2}{\partial z_1^2}p(0)+\frac{\partial^2}{\partial z_1\partial z_2}p(0)z+\frac{1}{2}\frac{\partial^2}{\partial z_2^2}p(0)z^2 .
\]
In the following theorem, we reformulate the Carath\'{e}odory-Fej\'{e}r interpolation problem for the bi-disc $\D^2.$
\begin{thm}
	For any polynomial $p$ of the form 
	$$p(z) = a_{10} z_1 + a_{01}z_2 + a_{20}z_1^2 + a_{11} z_1z_2 + a_{02} z_2^2,$$ 
	there exists a holomorphic function $q,$ 
	defined on the bi-disc $\D^2,$ with $q^{(k)}(0)=0$ for $|k|\leq 2,$ such that 
	$\|p+q\|_{\mathbb D^2, \infty} \leq 1$ if and only if
	 	$|p_2|\leq 1-|p_1|^2$ 
	and there exists a holomorphic function  
	$f:\mathbb{D}\to \mathcal B(L^2(\mathbb{T}))$
	with 
	\[\|f\|_{\D,\infty}^{\rm op}\leq 1 
	\mbox{ and }
	\frac{f^{(k)}(0)}{k!}=M_{p_k}\mbox{ for all }k\geq 0,\]
	where $p_0=0$ and for $k\geq 3,$ 
	$p_k\in\mathbb{C}[Z]$ 
	is a polynomial of degree less than or equal to $k.$ 
	Here $M_{p_k}$ is the multiplication operator on $L^2(\mathbb{T})$        
	induced by the polynomial $p_k.$ 
	\end{thm}

In this chapter, we show that $|p_1|^2+|p_2|\leq 1$ is a necessary condition for the solution to exist for the Carath\'{e}odory-Fej\'{e}r interpolation problem. In the following theorem, we also isolate a class of polynomials for which this is a sufficient condition for the existence of a solution. 
\begin{thm}
Let $p_1(z)=\gamma + \delta z$ and $p_2(z)=(\alpha + \beta z)(\gamma + \delta z)$ for some choice of complex numbers $\alpha,~\beta,~\gamma$ and $\delta.$ Assume that 
	$|p_1|^2+|p_2|\leq 1.$ 
	If either $\alpha \beta \gamma \delta =0$ or 
	$\arg(\alpha)-\arg(\beta)=\arg(\gamma)-\arg(\delta),$ 
	then $|p_1|^2+|p_2|\leq 1$ is 
	a sufficient condition also for the existence of a solution for the corresponding Carath\'{e}odory-Fej\'{e}r interpolation problem.
\end{thm}

We illustrate, by means of an example, that this necessary condition is not sufficient in general. In the end of this chapter, we give a proof of the Kor\'{a}nyi-Puk\'{a}nszky theorem for the bi-disc using the spectral theorem. This proof can be made to work for the polydisc as well.

In chapter 5, we give a generalization of Nehari's theorem to two variables. In this chapter we define the Hankel operator $H_{\phi}$ corresponding to any function $\phi\in L^\infty(\T^2).$ The following theorem shows that the norm of the Hankel operator $H_{\phi}$ is the norm of the symbol $\phi$ with respect to a quotient norm, described below.
\begin{thm}[Nehari's theorem for $L^2(\mathbb T^2)$] 
	If $\phi\in L^{\infty}(\T^2),$ 
	then $\|H_{\phi}\|=\mbox{\rm dist}_{\infty}(\phi,H_1)$.
\end{thm}     
In this theorem, $H_1:=\left\{f:=\!\!\!\!\!\sum\limits_{m+n\geq 0}\!\!\!\!\!\!a_{m,n}z_{1}^{m}z_{2}^{n}|f\in L^{\infty}(\T^2)\right\}$ and $\mbox{\rm dist}_{\infty}(\phi,H_1)$ is the distance of $\phi$ from $H_1$ in $L^\infty -$ norm.

In chapter 6, we study the operators space structures on $\ell^1(n)$. 
There is a canonical isometric embedding of $\ell^{\infty}(n)$ into the set of $n\times n$ matrices $M_n.$ However, we show that $\ell^1(n),\,n>1,$ has no isometric embedding into $M_k$ for any $k\in \N.$ 
\begin{thm}
There is no isometric embedding of $\ell^{1}(n),\,n>1,$ in to $M_k$ for any $k\in \N.$
\end{thm}
 The next theorem provides several isometric embeddings of $\ell^1(n)$ into $\mathcal{B}(\h)$ for each $n\in \N.$ 
 
Let $\h_1,\ldots,\h_n$ be Hilbert spaces and $T_i$ 
be a contraction on 
$\h_i$ for $i=1,\ldots,n.$ 
Assume that the unit circle $\T$ is contained in $\sigma(T_i),$ the spectrum of $T_i,$ for $i=1,\ldots,n$. 
Denote 
$$\tilde{T_1}=T_1\otimes I^{\otimes (n-1)},
\tilde{T_2}=I\otimes T_2\otimes I^{\otimes (n-2)},\ldots ,
\tilde{T_n}=I^{\otimes (n-1)}\otimes T_n.$$
\begin{thm}
Suppose the operators $\tilde{T}_1,\ldots,\tilde{T}_n$ are defined as above. Then, the function 
$$f:\ell ^1(n)\to \mathcal B(\h_1 \otimes\cdots \otimes \h_n)$$
defined by 
\begin{align*}
	f(a_1,a_2,\ldots , a_n):=a_1\tilde{T_1}+ a_2\tilde{T_2} +\cdots + a_n \tilde{T_n}.
\end{align*}
is an isometry.
\end{thm} 
 For $n=2,3,$ we show that all of these embeddings are completely isometric to the MIN structure.  In the end of this chapter, using these embeddings and Parrott's example in \cite{GMParrott}, we construct an operator space structure on $\ell^1(3)$ which is distinct from the MIN structure.

\chapter{Varopoulos Operators of Type I}
Let $\C[Z_1,\ldots,Z_n]$ denote 
the set of all polynomials in $n$ complex variables. 
For every contraction $T$ on a complex Hilbert space, 
the von-Neumann inequality \cite{vN} states that 
$\|p(T)\|\leq \|p\|_{\D,\infty}$ for every $p\in\mathbb{C}[Z]$. 
Ando \cite{ando} established an analogous inequality 
for any two commuting contractions 
$T_1,T_2,$ namely, $\|p(T_1,T_2)\|\leq \|p\|_{\D^2,\infty}$ 
for every $p\in\mathbb{C}[Z_1,Z_2]$. 
Varopoulos \cite{V1} constructed examples showing that the generalization of this inequality to three variables fails. He along with Kaijser also produced an explicit example of three commuting contractions $T_1,T_2,T_3$ and a polynomial $p$ with the property $\|p(T_1,T_2,T_3)\| > \|p\|_{\D^3,\infty}.$ Let $\|\boldsymbol T\|_{\infty}=\max\big\{\|T_1\|,\ldots,\|T_n\|\big\},$
\begin{equation} \label{C_k(n)}
C_k(n)=\sup\big\{\|p(\boldsymbol T)\|:\|p\|_{\D^n,\infty}\leq 1, p \mbox{~\rm is of degree at most~} k,\,\|\boldsymbol T\|_{\infty} \leq 1 \big\}.\\
\end{equation}
%
and 
\begin{equation} \label{C(n)}
C(n)= \lim_{k\to \infty} C_k(n). 
\end{equation}
In this notation, it follows from the von-Neumann inequality and Ando's theorem that $C(1), C(2) = 1.$ 

Also $C_2(3) >1,$ thanks to the example of Varopoulos and Kaijser \cite{V1} involving an (explicit) homogeneous polynomial of degree $2.$ 
Following this,  in the paper \cite{V2}, Varopoulos shows that the limit of the non-decreasing sequence $C_2(n)$ must be bounded above by  $2 K_G^\mathbb C,$ where $K_G^\mathbb C$ is the complex Grothendieck constant, the definition is given below.  He also showed that the lower bound for this limit is $K_G^\mathbb C.$ Thus he has proved
\begin{equation}\label{bound for C_2}
K_G^\mathbb C \leq \lim_{n\to \infty} C_2(n) \leq 2 K_G^\mathbb C.
\end{equation}
We show that $C_2(3) \geq 1.2$ by means of explicit examples. We were hoping to improve this  
inequality obtained earlier by Holbrook \cite{HJ} since our methods  appear to be somewhat more direct. In view of the known lower bound 
for $\lim_{n\to \infty}C_2(n)$ in \eqref{bound for C_2}, we hoped that 
the lower bound for $C_2(3)$ itself will be closer to $K_G^\mathbb C.$

We recall some of the details from the two papers \cite{V1, V2} of Varopoulos, which will be useful in what follows.  
Fix a Hilbert space $\mathbb{H}$ and 
a bilinear form $S$ on $\mathbb{H}$ with norm $1$. 
Let $e,f$ be two arbitrary but fixed vectors of length $1$ and set  $\mathcal{H}=\{e\}\oplus\mathbb{H}\oplus\{f\}.$ 
For any $x\in\mathbb{H}$, 
define $T_x:\mathcal{H}\to\mathcal{H}$ by the rule 
\begin{equation}\label{T(x)}
T_xf=x,~T_xy=S(x,y)e,~T_xe=0 \,\mbox{ for all }y\in\h
\end{equation}
and extend it linearly. 
It is then easily verified that 
for every $x,y\in\mathbb{H}$, $T_x$ and $T_y$ commute.
\begin{lem}\label{norm of T_x}
	For every $x\in\mathbb{H}$, $\|T_x\|=\|x\|$, 
	where the operator $T_x$ is defined according to \eqref{T(x)}.
\end{lem}

\begin{proof}
For $h\in\mathcal{H}$ and $\alpha,\beta\in\mathbb{C},$ 
we have $h=\alpha e+P_{\mathbb{H}}h+\beta f,$ 
where $P_{\mathbb{H}}:\mathcal{H}\to\mathcal{H}$ 
is the orthogonal projection on $\mathbb{H}$.
\[\langle T_x^*T_x h,h \rangle = 
|S(x,P_{\mathbb{H}}(h))|^2+|\beta|^2\|x\|^2\leq (\|P_{\mathbb{H}}h\|^2+|\beta|^2)\|x\|^2\]
therefore 
$\langle T_x^*T_x h,h \rangle \leq \|h\|^2\|x\|^2.$
Thus $\|T_x\|\leq \|x\|$. 
We already know that $\|T_x\|\geq \|x\|$ and 
hence $\|T_x\|=\|x\|.$ 
\end{proof}

\begin{defn}[Grothendieck Constant]
	Suppose $A:=\big ( \!\! \big (a_{jk}\big )\!\!\big )_{n\times n}$ is a complex(real) array satisfying

	\begin{equation}\label{hypothesis}
		\left|\sum_{j,k=1}^{n}a_{jk}s_jt_k \right|\leq 
		\max\left\{\left| s_j\right|\left| t_k\right|:1\leq j,k\leq n\right\},
	\end{equation}
	where $s_j,$ $t_k$ are complex(real) numbers. 
	Then there exists $K>0$ such that for any choice of sequence of vectors 
	$(x_j)_1^n$, $(y_k)_1^n$ in a complex(real) Hilbert space $\mathbb{H}$, 
	we have
	\begin{equation}\label{conclusion}
		\left|\sum_{j,k=1}^na_{jk}\langle x_j,y_k \rangle\right| 
		\leq K \max\left\{\|x_j\| \| y_k\|:1\leq j,k\leq n\right\}.
	\end{equation}
	The least constant $K$ satisfying inequality \eqref{conclusion} 
	is denoted by $K_G$ and called Grothendieck constant. 
	The constant $K_G$ is a universal constant independent of $n$ 
	and the matrices satisfying the hypothesis \eqref{hypothesis}. 
	Note that the definition of $K_G$ depends only on the underlying     field. 
	When it is the field $\C$ of complex numbers,   
	this constant is called the 
	complex Grothendieck constant and is denoted by $K_G^\mathbb C.$
	\end{defn}
Let $p_{\!_A}$ be the polynomial  $\sum_{i,j=1}^n a_{ij} z_i w_j.$  The inequality \eqref{hypothesis} is equivalent to saying that $\|p_{_{\!A}}\|_{\mathbb D^{2n},\infty} \leq 1.$ This follows from the equality $\|p_{_{\!A}}\|_{\mathbb D^{2n},\infty} = \|A\|_{\ell^\infty(n) \to \ell^1(n)}.$   Let $p_{_{\!A,\vartriangle}}$ be the restriction of $p_{\!_A}$ to the diagonal set 
$$\vartriangle = \big \{(z_1,\ldots, z_n, z_1,\ldots, z_n): |z_i| < 1, 1\leq i \leq n \big \},$$ which is the polydisc $\mathbb D^n.$
Thus $\|p_{\!_{A,\vartriangle}}\|_{\mathbb D^n,\infty}$ is also at most $1.$  
If $A$ is symmetric, then the second derivative $D^2 p_{_{\!A, \vartriangle}}(0)$ is $2A.$ It is therefore clear that 
\begin{equation} \label{supnotinfty1}
\|p_{_{\!{2A, \vartriangle}}}\|_{\mathbb D^n,\infty} \leq  2\|A\|_{\ell^\infty(n) \to \ell^1(n)}.
\end{equation}
We find examples where \eqref{supnotinfty1} is strict. Indeed, for this particular example, we show that 
$$\frac{\|A\|_{\ell^\infty(n) \to \ell^1(n)}}{\|p_{_{\!{A,\vartriangle}}}\|_{\mathbb D^n,\infty}} \geq 1.2$$
This observation will be important for us in what follows. 
As pointed out earlier, the following theorem is due to Varopoulos.

\begin{thm}[\cite{V2}]\label{lower bound for homogeneous degree 2}
$\lim_{n\to \infty} C(n) \geq K_G^\mathbb C.$ 
\end{thm}

\begin{proof}
It is a well known that $K_{G}^{\C}>1$. 
Let $\epsilon >0$ be a fixed real number 
such that $K_{G}^{\C}-\epsilon > 1$. 
Since $K_{G}^{\C}$ is the least constant in the inequality \eqref{conclusion}, 
therefore there exists a matrix $A:=\big ( \!\!\big (a_{jk}\big )\!\!\big )_{n\times n}$ satisfying the inequality \eqref{hypothesis} and unit vectors $x_i, y_i,$ in $\ell^2(k),$  $1\leq i \leq n,$ for some $k \in \mathbb N$ such that  
\[\sum_{j,k=1}^na_{jk}\langle x_j,y_k \rangle >(K_{G}^{\C}-\epsilon)\]
Let
	\[\tilde{A}=\big ( \!\!\big (\tilde{a}_{jk}\big )\!\!\big )_{2n\times 2n}
	:=\frac{1}{2}
	\left( 
	{\begin{array}{cc}
	   0 & A \\
 	  A^t & 0 \\
 \end{array}
  }
   \right).\]
It is easy to see that 
$\tilde{A}$ satisfies inequality \eqref{hypothesis}. 
Take $\tilde{x}_1=x_1,~\tilde{x}_2=x_2,...,~\tilde{x}_n=x_n,
~\tilde{x}_{n+1}=\bar{y}_1,...,~\tilde{x}_{2n}=\bar{y}_n$ and 
consider the bilinear form $S$ on $\ell^2(k)$ defined as follows: 
\[S(x,y)=\sum_{j=1}^{k}x_jy_j,\]
where $x=(x_1,\ldots,x_k)$ and $y=(y_1,\ldots,y_k)$. 
The operator $\tilde{A}:\ell^{\infty}(2n)\rightarrow \ell^1(2n)$ is of norm at most $1$ 
and
\[\sum_{j,k=1}^{2n}\tilde{a}_{jk}S(\tilde{x}_j,\tilde{x}_k)
=\frac{1}{2}\left\{\sum_{j,k=1}^{n}a_{jk}\langle x_j,y_k\rangle 
+\sum_{j,k=1}^{n}a_{kj}\langle \bar{y_j},\bar{x_k}\rangle\right\},\]
which implies
\begin{equation}\label{new conclusion}
	\sum_{j,k=1}^{2n}\tilde{a}_{jk}S(\tilde{x}_j,\tilde{x}_k)
	=\sum_{j,k=1}^{n}a_{jk}\langle x_j,y_k\rangle>K_{G}^{\C}-\epsilon.
\end{equation}
The polynomial $p(z_1,\ldots,z_{2n})=\sum_{j,k=1}^{2n}\tilde{a}_{jk}z_jz_k$ 
is a homogeneous polynomial of degree two. 
It is clear that $\|p\|_{\D^{2n},\infty}\leq 1$. 
Consider the operators (as defined in \eqref{T(x)}),
\begin{equation}\label{first Varopoulos operator}
	T_{\tilde{x}_j}=
	\left(
	\begin{array}{ccc}
		0 & \tilde{x}_j & 0\\
		0 & 0 & \tilde{x}_j^t\\
		0 & 0 & 0\\
	\end{array}
	\right)
\end{equation}
for $j=1,\ldots,2n$. 
Then $\|p(T_{\tilde{x}_1},\ldots,T_{\tilde{x}_{2n}})\|>K_{G}^{\C}-\epsilon$ 
is a direct implication of the inequality \eqref{new conclusion}.
\end{proof}

Let $\h$ be a separable Hilbert space and 
$\{e_j\}_{j\in\mathbb{N}}$ be a set of orthonormal basis for $\h$. 
For any $x\in\h$, 
define $x^{\sharp}:\h\to\C$ by $x^{\sharp}(y)=\sum_{j}x_jy_j,$ 
where $x=\sum x_je_j$ and $y=\sum y_je_j$. 
For $x,y\in\h$, we set $\left[x^{\sharp},y\right]=x^{\sharp}(y)$. 
From the definition it can be seen that 
$\left[x^{\sharp},y\right]=\left[y^{\sharp},x\right]$. 
Let $\h^{\sharp}:=\left\{x^{\sharp}:x\in\h\right\}.$ 
Let $\h^{\sharp}$ be equipped with the operator norm.
Since the map $\phi:\h\to\h^{\sharp}$ defined by 
$\phi(x)=x^{\sharp}$ is a linear onto isometry,  
therefore $\h^{\sharp}$ is linearly (as opposed to the usual anti-linear identification) isometrically isomorphic to $\h$.

Let $\h_1$ and $\h_2$ be two separable Hilbert spaces  
and  $\{e_j\}_{j\in\mathbb{N}}$,
$\{\tilde{e}_j\}_{j\in\mathbb{N}}$ 
be orthonormal bases 
of $\h_1$ and $\h_2$ respectively. 
Let $\{f_j\}_{j\in\mathbb{N}}$ and 
$\{\tilde{f}_j\}_{j\in\mathbb{N}}$ 
be the corresponding dual basis 
for $\h_1$ and $\h_2$ respectively. 
For a linear map $T:\h_1 \to \h_2$, 
define $T^{\sharp}:\h_2 \to \h_1$ by
\[T^{\sharp}\left(\tilde{e}_k\right)=\sum_{j}\tilde{f}_k(Te_j)e_j\]
and extend it linearly. We note that if $T$ is bounded then so is the operator $T^\sharp.$
We have the following lemma:

\begin{lem}\label{linear identification for H}
	Let $\h_1,\h_2$ and $\h_3$ be separable Hilbert spaces  
	and $\{e_{j}^{(p)}\}_{j\in\mathbb{N}}$	be an orthonormal basis 
	for $\h_p$ for $p=1,2,3$. 
	Let $\{f_{j}^{(p)}\}_{j\in\mathbb{N}}$ 
	be the corresponding dual basis for $\h_p$ for $p=1,2,3$.
	If $T:\h_1 \to \h_2$ and $S:\h_2 \to \h_3$
	are two bounded operators,  
	then $\left(S\circ T\right)^{\sharp}=T^{\sharp}\circ S^{\sharp}.$
\end{lem}
\begin{proof}
It is enough to check the equality 
$\left(S\circ T\right)^{\sharp}=T^{\sharp}\circ S^{\sharp}$ 
on the basis elements 
$\{e_{k}^{(3)}\}_{k\in\mathbb{N}}$. 
For any $k\in\mathbb{N}$,
\begin{align*}
	\left(S\circ T\right)^{\sharp}\big(e_{k}^{(3)}\big)
	&=\sum_{j} f_{k}^{(3)}\big(S\big(T(e_{j}^{(1)}\big)\big)e_{j}^{(1)}.\\
	\big(T^{\sharp}\circ S^{\sharp}\big)\big(e_{k}^{(3)}\big)
	&=T^{\sharp}\big(\sum_{j}f_{k}^{(3)}\big(Se_{j}^{(2)}\big)e_{j}^{(2)}\big)\\
	&=\sum_{j}f_{k}^{(3)}\big(Se_{j}^{(2)}\big)
	\big(\sum_{l}f_{j}^{(2)}(Te_{l}^{(1)})e_{l}^{(1)}\big).
\end{align*}
Thus
\begin{align*}
	\big(T^{\sharp}\circ S^{\sharp}\big)\big(e_{k}^{(3)}\big)
	&=\sum_{l}f_{k}^{(3)}
	\big(S\big(\sum_{j}f_{j}^{(2)}(Te_{l}^{(1)})e_{j}^{(2)}\big)\big)e_{l}^{(1)}\\
	&=\sum_{l}f_{k}^{(3)}\big((S\circ T)(e_{l}^{(1)})\big)e_{l}^{(1)}.
\end{align*}
Hence
$\big(S\circ T\big)^{\sharp}\big(e_{k}^{(3)}\big)
=\big(T^{\sharp}\circ S^{\sharp}\big)\big(e_{k}^{(3)}\big).$
\end{proof}

The form of the operator appearing in \eqref{first Varopoulos operator} 
and the operators used in the addendum of \cite{V1}, 
suggest the definition of the following two classes:
\begin{defn}[Varopoulos Operator of Type I (\sf{V\,I})]
	Let $\mathbb{H}$ be a separable Hilbert space. 
	For $x,y\in\mathbb{H}$, 
	define $T_{x,y}:\C \oplus \h \oplus \C \to \C \oplus \h \oplus \C$ by
	\[T_{x,y}=
	\left(
	\begin{array}{ccc}
		0 & x^{\sharp} & 0\\
		0 & 0 & y\\
		0 & 0 & 0\\
	\end{array}
	\right).\]
	The operator $T_{x,y}$ will be called Varopoulos operator of type I 
	corresponding to the pair of vectors $x,y$. 
	If $x=y$ then $T_{xy}$ will simply be denoted by $T_x$.
\end{defn}
\begin{defn}[Varopoulos Operator of Type II and of order $k$ (\sf{V\,II  of order k})]
	Let $\mathbb{H}$ be a separable Hilbert space. 
	For $X\in\mathcal{B}(\mathbb{H})$, let
	\[T_X=
	\left(
	\begin{array}{ccccc}
		0 & X & 0 & \cdots &0\\
		0 & 0 & X & \cdots & 0\\
		\texorpdfstring{\vdots}{TEXT} & \vdots & \vdots & \ddots & \vdots\\
		0 & 0 & 0 & \cdots & X\\
		0 & 0 & 0 & \cdots & 0\\
	\end{array}
	\right)\]
	be the operator in $\mathcal B(\mathbb H\otimes C^{k+1}).$
In analogy with the work of Varopoulos \cite{V2}, operators of the form $T_X,$ $X\in \mathcal B(\mathbb H),$ are called 	Varopoulos operator of type II and of order $k.$ 
\end{defn}


In the following section, we show that $\|p(T_1,\ldots , T_n)\| \leq \|p\|_{\mathbb D^n, \infty}$ for any $n$ commuting contractions of the form
\begin{equation}\label{3x3class}
\Big \{\Big ( \,\,\Big ( \begin{smallmatrix}
	\w_1 & \alpha_1 & 0\\
	0 & \w_1 & \beta_1\\
	0 & 0 & \w_1\\
\end{smallmatrix} \Big ), \ldots, \Big ( \begin{smallmatrix}
	\w_n & \alpha_n & 0\\
	0 & \w_n & \beta_n\\
	0 & 0 & \w_n\\
\end{smallmatrix} \Big ) \,\,\Big ): 
\alpha_i \beta_j = \alpha_j \beta_i, 1\leq i,j \leq n,\, \w:=(\w_1, \ldots , \w_n)\in \mathbb D^n \Big\},
\end{equation}
after assuming that $|\alpha_i| = |\beta_i|,$ $1\leq i\leq n.$ 
This is interesting considering that the von-Neumann inequality is valid for any commuting $n$ - tuple of $2\times 2$ contractions \cite{GMPati,Agler} and fails for $4\times 4$ contractions \cite{HJ}. 
 
Secondly, we show that $\lim_{n\to \infty} C_2(n) \leq \frac{3\sqrt{3}}{4} K_G^\mathbb C$ giving a considerable improvement on the upper bound previously obtained in \cite{V2}. 

Finally, we investigate in some detail, the contractivity of the homomorphisms 
$\rho_{x,y}$ induced by the Varopoulos operators of type I ({\sf{V\,I}}). 
In particular, for a pair of commuting contractions $T_{x_1}, T_{x_2},$ $x_1,x_2 \in \mathbb C,$ of type {\sf{V\,I}},  and any holomorphic function $f:\mathbb D^2 \to \mathbb D,$ $f(0,0)=0,$ applying the von-Neumann inequality, we must have 
$$		\sup\limits_{x_1,x_2\in\D}\left\|\mathscr{T} \left( \frac{\partial f}{\partial z_1}(0)x_1 
		+ \frac{\partial f}{\partial z_2}(0)x_2,\frac{1}{2}
		\sum_{i,j=1}^{2}\frac{\partial ^2 f}
		{\partial z_i \partial z_j}(0) x_i x_j\right)\right\|
		\leq 1,
$$ 
where $\mathscr{T}(\omega, \alpha) = \Big (\begin{array}{cc} \omega & \alpha\\
0&\omega \end{array}\Big ).$ The solution to this problem (indeed a generalization of it), which we obtain in Chapter \ref{Extremal}, therefore gives a necessary condition for the Carath\'{e}odory- Fej\'{e}r interpolation problem for polynomials of degree $2.$ Unfortunately, while the extremal problem can be stated for any $n$ in $\mathbb N,$ not just for $2,$ its solution depends on the commutant lifting theorem.	
\section{The von-Neumann Inequality}
The von-Neumann inequality  for a commuting  $n$ - tuple of $3\times 3$ matrices remains open for $n\geq 3.$ 
In this section, we establish this inequality for any $n$ - tuple of commuting contractions of the form prescribed in \eqref{3x3class} with 
the additional assumption that $|\alpha_i| = |\beta_i|,$ $1\leq i \leq n.$  

Let $\h$ be a separable Hilbert space. 
Given a set of $n$ operators $A_1,\ldots,A_n$ in $\mathcal{B}(\h),$ 
define the operator 
\[\mathscr{T}(A_1,\ldots,A_n):=\left(
{\begin{array}{ccccc}
	A_1 & A_2 & A_3 &\cdots & A_n \\
	0 & A_1 & A_2 & \cdots & A_{n-1}\\ 
	0 & 0 & A_1 & \cdots & A_{n-2}\\
	\vdots & \vdots & \vdots & \ddots & \vdots \\
	0 & 0 & 0 & \cdots & A_1 \\
\end{array} } 
\right),\]
which is in $\mathcal B(\mathbb H\otimes \mathbb C^n).$

The condition for the contractivity of any  $3\times 3$ matrix of the form   
\begin{equation}\label{Mod are same}
T=
\left(
\begin{array}{ccc}
	\w & \alpha & 0\\
	0 & \w & \beta\\
	0 & 0 & \w\\
\end{array}
\right),\, \w \in \mathbb D,\,\alpha \mbox{\rm~and~}\beta \mbox{\rm~in~} \mathbb C. 
\end{equation}
is given in the following lemma. It will be used repeatedly in what follows.  
\begin{lem} \label{OPnorm}
	The operator $T$ defined in \eqref{Mod are same} 
	is a contraction if and only if 
	\[|\alpha|\leq 1-|\w|^2,~ |\beta|\leq 1-|\w|^2\]
	and 
	\[|\alpha\beta\w|^2
	\leq \left(\big(1-|\w|^2\big)^2-|\alpha|^2\right)
	\left(\big(1-|\w|^2\big)^2-|\beta|^2\right).\]
	In particular, if $|\alpha| = |\beta|,$ then $T$ is contractive if and only if $|\alpha|\leq (1-|\w|)\sqrt{1+|\w|}.$ 
\end{lem}

\begin{proof} Suppose $T$ is contraction.   
Then $\mathscr{T}(\w,\alpha):= \Big( \begin{smallmatrix} \omega & \alpha\\ 0 & \omega \end{smallmatrix} \Big )$ and $\mathscr{T}(\w,\beta):= \Big( \begin{smallmatrix} \omega & \beta\\ 0 & \omega \end{smallmatrix} \Big )$ must be contractions.  
Hence, we have 
$|\alpha|\leq 1-|\w|^2$ and $|\beta|\leq 1-|\w|^2.$  
By Parrott's theorem \cite{SP}, 
there exists $a\in \C$ such that the operator $T_a$ is a contraction, where 
\[T_a=\left(
\begin{array}{ccc}
	\w & \alpha & a\\
	0 & \w & \beta \\
	0 & 0 & \w\\
\end{array}
\right).\]
Every possible choice of $a$ in $\mathbb C,$ ensuring contractivity of the operator $T_a$ is given by  
\[a=(I-ZZ^*)^{1/2}V(I-Y^*Y)^{1/2}-ZS^*Y,\]
where $V:\C\to \C$ is an arbitrary contraction, 
$S=\mathscr{T}(0,\w)$, $R=(\beta,\w)^{\rm t}$ 
and $Q=(\w,\alpha)$. 
The operators $Y$ and $Z$ are explicitly determined by the formulae
$R=(I-SS^*)^{1/2}Y$ and $Q=Z(I-S^*S)^{1/2}$.
The reader is referred to (cf. \cite[Chapter 12]{NY}) 
for more information on the Parrott's theorem and related topics.
Thus  
\[Y=
\left(
	\frac{\beta}{(1-|\w|^2)^{1/2}}, \w\\
\right)^{\rm t}
\mbox{ and } Z=\left(\w,\frac{\alpha}{(1-|\w|^2)^{1/2}}\right).\]
Therefore 
\begin{equation}\label{Choices for a}
	a=\left((1-|\w|^2)-\frac{|\alpha|^2}{1-|\w|^2}\right)^{1/2}V\left((1-|\w|^2)-\frac{|\beta|^2}{1-|\w|^2}\right)^{1/2}
	-\frac{\overline{\w}\beta\alpha}{1-|\w|^2}.
\end{equation}
Since $T$ is a contraction, it follows that 
$a=0$ is a valid choice in \eqref{Choices for a} 
for some contraction $V$. This forces  
\[V=\frac{\overline{\w}\beta\alpha}
{1-|\w|^2}\Big(\big(1-|\w|^2\big)-\frac{|\alpha|^2}{1-|\w|^2}\Big)^{-1/2}\Big(\big(1-|\w|^2\big)-\frac{|\alpha|^2}{1-|\w|^2}\Big)^{-1/2}\]
to be of absolute value at most 1. 
Thus we have 
\[\frac{|\w|^2|\alpha|^2 |\beta|^2}{\big(1-|\w|^2\big)^2}
\leq \Big(\big(1-|\w|^2\big)-\frac{|\alpha|^2}{1-|\w|^2}\Big)
\Big(\big(1-|\w|^2\big)-\frac{|\beta|^2}{1-|\w|^2}\Big).\]
Hence we get 
\begin{align*}
	|\alpha\beta\w|^2
	\leq \left(\big(1-|\w|^2\big)^2-|\alpha|^2\right)
	\left(\big(1-|\w|^2\big)^2-|\beta|^2\right).
\end{align*}
All the steps in the proof given above are reversible. Therefore, 
the converse statement is valid as well.

The condition for contractivity assuming $|\alpha| = |\beta|$ is easily seen to be $$|\alpha|\leq (1-|\w|)\sqrt{1+|\w|}.$$ 
\end{proof}

\begin{rem}
	It is known that $T$ is a contraction if and only if 
	$\|f(T)\|\leq 1$ for all $f$ in the disc algebra $\A(\D)$ 
	with $f(\w)=0$ for an arbitrary but fixed $\w$ in $\mathbb D$ and $\|f\|_{\D,\infty}\leq 1$. 
	Therefore $T$ in \eqref{Mod are same} is a contraction if and only if 
	\begin{align*}
		|f'(\w)|^2|\alpha|^2+|f{'\!'}(\w)/2||\alpha|^2\leq 1
	\end{align*}
	for all $f\in\A(\D)$ with $f(\w)=0$ and $\|f\|_{\D,\infty}\leq 1.$	
	Thus 
	\[|\alpha|^2\leq 
	\frac{1}{\sup \left(|f'(\w)|^2+|\frac{f{'\!'}(\w)}{2}|\right)},\]
	where the supremum is over the set 
	$\{f\in \A(\D):f(\w)=0,\|f\|_{\D,\infty}\leq 1\}$. 
	From the Lemma \ref{OPnorm}, 
	we  conclude (for some arbitrary but fixed $\w \in \mathbb D$) that 
	\begin{align*}
	\sup \left\{|f'(\w)|^2+\big |\frac{f''(\w)}{2}\big |:f \in \mathbb A(\D),  f(\w)=0, \|f\|_{\D,\infty}\leq 1 \right\}
	=\frac{1}{(1-|\w|^2)(1-|\w|)}.
	\end{align*}
	\end{rem}

For $\alpha_j,\beta_j\in \mathbb D,$ define the operators 
\[T_j=
\left(
\begin{array}{ccc}
	0 & \alpha_j & 0\\
	0 & 0 & \beta_j\\
	0 & 0 & 0\\
\end{array}
\right), \, 1\leq j \leq n,
\]
%
and assume that $\alpha_j\beta_k=\alpha_k\beta_j,$ $j,k=1,\ldots, n.$ This commuting $n$-tuple of contractions $T=(T_1,\ldots ,T_n)$ is in the set \eqref{3x3class} with $\w=0.$ It defines a homomorphism $\rho_{\!_T}: \mathbb C[Z_1,\ldots,Z_n] \to \mathcal B(\mathbb C^3)$ given by the formula 
$ \rho_{\!_T}(p) := p(T).$ Explicitly evaluating $p(T),$ we obtain 
\begin{equation}\label{Hom with 0 eigenvalue}
\rho_{\!_T}(f):=\left(
\begin{array}{ccc}
	f(0) & Df(0)\cdot\alpha & \frac{1}{2}D^2f(0)\cdot A_{\alpha\beta}\\
	0 & f(0) & Df(0)\cdot\beta\\
	0 & 0 & f(0)\\ 
\end{array}
\right),\,\, f\in \mathbb C[Z_1,\ldots,Z_n], 
\end{equation}
where $A_{\alpha\beta}=\big ( \!\! \big (\alpha_j\beta_k\big )\!\!\big )_{n\times n}.$ Here  
\[Df(0)\cdot\alpha=\sum_{j} \frac{\partial f}{\partial z_j}(0)\alpha_j,
~Df(0)\cdot\beta=\sum_{j} \frac{\partial f}{\partial z_j}(0)\beta_j\] 
and 
\[D^2f(0)\cdot A_{\alpha\beta}=
\sum_{j,k}\frac{\partial^2f}{\partial z_j\partial z_k}(0)\alpha_j\beta_k.\]
Clearly, the formula \eqref{Hom with 0 eigenvalue} makes sense and defines a homomorphism of the algebra ${\rm H}^\infty(\mathbb D^n)$ consisting of all bounded holomorphic functions on the polydisc ${\mathbb D}^n.$  

The following lemma and several of its variants involving functions defined on domains in $\mathbb C^n$ and taking values in $k\times k$ matrices have been proved in \cite{Gmthesis, GMSastry, GMParrott, VP}. The proof below follows closely the one appearing in \cite{VP}.

\begin{lem}[The zero lemma] \label{zeropolydisk}
	The homomorphism $\rho_{\!_T}$ is a contraction 
	if and only if $\|\rho_{\!_T}(f)\|\leq 1$ 
	for all $f\in {\rm H}^\infty(\D^n)$ 
	with $f(0)=0$ and $\|f\|_{\D^n,\infty}\leq 1$.
\end{lem}
\begin{proof}
Let us assume that $\|\rho_{\!_T}(f)\|\leq 1$ 
for all $f:\D^n\to\D$ with $f(0)=0$. 
Let $g:\D^n\to\D$ be an analytic function 
and $\phi$ be an automorphism of $\D$ 
mapping $g(0)$ to $0$. 
Then $\phi\circ g$ is an analytic map 
from $\D^n$ to $\D$ with $(\phi\circ g)(0)=0$, 
therefore $\|\rho_{\!_T}(\phi\circ g)\|\leq 1$. 
Now by von-Neumann's inequality 
we have 
$\|\phi^{-1}(\rho_{\!_T}(\phi\circ g))\|\leq 1$ 
which is equivalent to $\|\rho_{\!_T}(g)\|\leq 1.$ 
Hence $\rho_{\!_T}$ is a contraction. 
The converse is trivially true.
\end{proof}
 
\begin{thm}\label{Theorem for zero as only eigenvalues}
	The homomorphism $\rho_{\!_T}$, 
	as defined in \eqref{Hom with 0 eigenvalue} for the commuting tuple of contractions $T,$ is contractive. 
\end{thm} 
\begin{proof} 
Assume that the supremum norm of the polynomial 
\[p(z_1,\ldots,z_n)
=\sum_{j=1}^{n}a_j z_j
+\sum_{i,j=1}^{n}a_{ij}z_iz_j
+\sum_{k=3}^{d}\sum_{|I|=k}a_Iz^I\]
over the polydisc $\D^n$ is at most $1.$
Then 
\[\left\|\rho_{\!_T}(p)\right\|=\left\|\left(
\begin{array}{cc}
	\sum a_i\alpha_{i} & \sum a_{ij}\alpha_i\beta_j\\
	0 & \sum a_i\beta_i\\ 
\end{array}
\right)
\right\|
\leq 
1 
\]
if and only if 
\begin{eqnarray}\label{Functional calculus for order 3 with diagonal 0}
	\left|\sum_{i,j=1}^{n} a_{ij}\alpha_i\beta_j\right|^2\leq
	\left(1-\left|\sum_{j=1}^{n}a_j\alpha_j \right|^2\right)
	\left(1-\left|\sum_{j=1}^{n}a_j\beta_j\right|^2\right).
\end{eqnarray} 
Without loss of generality we assume 
$0<|\beta_1|\leq |\alpha_1|$. 
Let $|\beta_1|/|\alpha_1|=\mu $. 
We have $\alpha_j\beta_k=\alpha_k\beta_j$ 
for all $j,k=1,\ldots,n$ 
therefore inequality \eqref{Functional calculus for order 3 with diagonal 0} 
is equivalent to
\begin{eqnarray}\label{Contractivity condition for taken operators with diagonal 0}
	\left|\sum_{i,j=1}^{n}a_{ij}\alpha_i \alpha_j\right|^2\leq
	\left(1-\left|\sum_{j=1}a_j\alpha_j\right|^2\right)\left(\frac{1}{\mu ^2}-\left|\sum_{j=1}a_j		\alpha_j\right|^2\right).
\end{eqnarray}
Define $q_{\alpha}(t):=p(t\alpha_1,\ldots,t\alpha_n)$ 
for all $t\in \D$. 
Since $\|p\|_{\D^n,\infty}\leq 1$ 
therefore $\|q_{\alpha}\|_{\D,\infty}\leq 1$ and 
hence 
$\mathscr{T}\big(\sum a_i\alpha_{i},\sum a_{ij}\alpha_i\alpha_j\big)$ 
is of norm at most 1. 
Therefore 
\[\left|\sum_{i,j=1}^{n}a_{ij}\alpha_i \alpha_j\right|^2\leq
\left(1-\left|\sum_{j=1}a_j\alpha_j\right|^2\right)
\left(1-\left|\sum_{j=1}a_j\alpha_j\right|^2\right)\]
and since $\mu \leq 1$, it follows that 
\eqref{Contractivity condition for taken operators with diagonal 0} holds.  
Hence $\rho_{\!_T}$ is a contractive homomorphism.
\end{proof}
%
%

We now prove the von-Neumann inequality for any contractive $n$ - tuple in the set \eqref{3x3class} assuming $|\alpha_i| = |\beta_i|,$ $1\leq i \leq n.$ (We no longer assume that $\w=0.$) As before, such a $n$-tuple defines a homomorphism 
$\rho_{\!_{T,\w}}:{\rm H}^\infty(\D^n)\to \mathcal{B}(\C^3)$ by 
\begin{equation}\label{Homomophism corresponding to general operators}
	\rho_{\!_{T,\w}}(f):=\left(
	\begin{array}{ccc}
		f(\w) & Df(\w)\cdot\alpha & \frac{1}{2}D^2f(\w)\cdot A_{\alpha\beta}\\
		0 & f(\w) & Df(\w)\cdot\beta\\
		0 & 0 & f(\w)\\ 
	\end{array}
	\right), 
\end{equation}
where 
\[Df(\w)\cdot\alpha=\sum_{j} \frac{\partial f}{\partial z_j}(\w)\alpha_j,
~Df(\w)\cdot\beta=\sum_{j} \frac{\partial f}{\partial z_j}(\w)\beta_j\]
and 
\[D^2f(\w)\cdot A_{\alpha\beta}
=\sum_{j,k}\frac{\partial^2f}{\partial z_j\partial z_k}(\w)\alpha_j\beta_k.\]
Here is ``the zero lemma{'\!'} again, now adapted to work for the homomorphism $\rho_{\!_{T,\w}}.$ For the proof, we compose $f$ with an automorphism of the disc taking $f(\w)$ to $0,$ whenever $f(\w) \not = 0.$
\begin{lem}
The homomorphism 	$\rho_{\!_{T,\w}}$ is contractive if and only if 
	$\|\rho_{\!_{T,\w}}(f)\|\leq 1$ for all $f\in {\rm H}^\infty(\D^n)$ 
	with $f(\w)=0$ and $\|f\|_{\D^n,\infty}\leq 1$.
\end{lem}

\begin{thm}\label{Same alpha and beta}
	The homomorphism $\rho_{\!_{T,\w}}$ induced by  a contractive $n$ - tuple  $T$ in the set \eqref{3x3class} 
	with $\alpha_j=\beta_j,$  $j=1,\ldots, n,$ 
	 	is itself contractive.
\end{thm} 
\begin{proof}
Assume that the supremum norm of the polynomial 
\[p(z_1,\ldots,z_n)=\sum_{j=1}^{n}a_j(z_j-\w_j)
+\sum_{i,j=1}^{n}a_{ij}(z_j-\w_j)(z_i-\w_i)
+\sum_{k=3}^{d}\sum_{|I|=k}a_I(z-\w)^I\]
over the polydisc $\D^n$ is at most 1,  
where $d$ is the degree of $p$. 
Then  
\[\left\|\rho_{\!_{T,\w}}(p)\right\|
=\left\|\left(
\begin{array}{cc}
	\sum a_i\alpha_{i} & \sum a_{ij}\alpha_i\alpha_j\\
	0 & \sum a_i\alpha_i\\ 
\end{array}
\right)
\right\|
\leq 
1 
\]
if and only if 
\begin{eqnarray}\label{Functional calculus for order 3}
	\left|\sum_{i,j=1}^{n} a_{ij}\alpha_i\alpha_j\right|\leq
	\left(1-\left|\sum_{j=1}^{n}a_j\alpha_j \right|^2\right).
\end{eqnarray}
For $j=1,\ldots,n,$ applying the Lemma \ref{OPnorm} to the operator  
\[
\left(
\begin{array}{ccc}
	\w_j & (1-|\w_j|)\sqrt{1+|\w_j|} & 0\\
	0 & \w_j & (1-|\w_j|)\sqrt{1+|\w_j|}\\
	0 & 0 & \w_j\\
\end{array}
\right)
\] 
we conclude that it must be contractive. 
Therefore, using Nehari's  theorem (cf. \cite[Chapter 15, Theorem 15.14]{NY}), we obtain a holomorphic function $h_j,$ $h_j^{(k)}(0) = 0, \, k=0,1,2,$ defined in the unit disc $\mathbb D$ such that the supremum norm of the function \[f_j(z)=\w_j + (1-|\w_j|)\sqrt{1+|\w_j|}z+ h_j(z)\] over the unit disc $\D$ is at most $1.$
Define $f=(f_1,\ldots,f_n):\D^n\to \D^n$ by 
$f(z_1,\ldots,z_n)=\big(f_1(z_1),\ldots,f_n(z_n)\big)$. 
Now $p\circ f$ maps $\D^n$ to $\D$ 
with $p\circ f(0)=0$. 
Define the following contractive operators 
\[S_j:=\left(
\begin{array}{ccc}
	0 & \frac{\alpha_j}{(1-|\w_j|)\sqrt{1+|\w_j|}} & 0\\
	0 & 0 & \frac{\alpha_j}{(1-|\w_j|)\sqrt{1+|\w_j|}}\\
	0 & 0 & 0\\
\end{array}
\right)\]
for $j=1,\ldots,n$.    
Then $S=(S_1,\ldots,S_n)$ is a tuple of commuting contractions. 
From Theorem \ref{Theorem for zero as only eigenvalues} 
it is clear that $\|p\circ f (S)\|\leq 1$. 
Therefore \eqref{Functional calculus for order 3} holds and  
hence $\|p(T_1,\ldots,T_n)\|\leq 1$.
\end{proof}
As a corollary of this theorem, we get the following necessary condition for the Carath\'{e}odory-Fej\'{e}r interpolation problem for the polydisc $\D^n.$
\begin{thm}
Let $p$ be a polynomial in $n$ variables of degree $2$ such that $p(0)=0.$ There exists a holomorphic function $q,$ defined on polydisc $\D^n,$ with $q^{(k)}(0)=0,\,|k|\leq d$ such that $\|p+q\|_{\infty}\leq 1$ only if 
\[\sup_{\|\alpha\|_{\infty}\leq 1}\Big\{\Big|\frac{D^2p(0)\cdot A_\alpha}{2}\Big|+\big|Dp(0)\cdot\alpha\big|^2\Big\}\leq 1,\]
where $\alpha=(\alpha_1,\ldots,\alpha_n),$ $A_{\alpha}=\big(\!\!\big(\alpha_i\alpha_j\big)\!\!\big),$ $D^2p(0)\cdot A_{\alpha}=\sum\frac{\partial^2 p}{\partial{z_i}\partial{z_j}}(0)\alpha_i\alpha_j$ and $Dp(0)\cdot \alpha=\sum\frac{\partial p}{\partial{z_i}}(0)\alpha_i.$
\end{thm} 

\begin{rem}
This proof of the von-Neumann inequality works without having to make  the assumption that $|\alpha_i| = |\beta_i|,$ if instead, we assume that   
  	\[
	\left(
	\begin{array}{ccc}
		\w_j & \alpha_j & 0\\
		0 & \w_j & \alpha_j\\
		0 & 0 & \w_j\\
	\end{array}
	\right) 
	\mbox{ and }
		\left(
	\begin{array}{ccc}
		\w_j & \beta_j & 0\\
		0 & \w_j & \beta_j\\
		0 & 0 & \w_j\\
	\end{array}
	\right)
	\]
are contractions for $j=1,\ldots,n.$ 
Unfortunately, there are contractive $n$ - tuples $T$ in the set \eqref{3x3class} for which this condition is not met, for example, 
take $\w=\alpha=2/5$ and $\beta=4/5,$ here $n$ is just $1!$ 
\end{rem}

\section{An improvement in the bound for \texorpdfstring{$C_2(n)$}{TEXT}}
The explicit example in \cite{V1} showing that a commuting triple of contractions need not define a contractive homomorphism of the tri-disc algebra uses the interesting polynomial 
$$p_{\!_V}(z_1,z_2,z_3):=z_{1}^{2}+z_{2}^{2}+z_{3}^{2}-2z_1z_2-2z_2z_3-2z_3z_1.$$ 
This polynomial will be referred as the \VK polynomial. 
The supremum norm of $p_{\!_V}$ over the tri-disc is shown to be $5.$  
Let 
\begin{equation}\label{ptoA}
	A_{\!_V}:=
	\left(
	\begin{array}{rrr}
		1 & -1 & -1\\
		-1 & 1 & -1\\
		-1 & -1 & 1\\
	\end{array}
	\right)
\end{equation}
be the matrix of co-efficients of the polynomial $p_{\!_V}.$
\begin{lem}
$\|A_{\!_V}\|_{\ell^{\infty}(3)\to \ell^1(3)}\geq 6.$
\end{lem}
\begin{proof}
Suppose $a_{ij}$ denote the 
$(i,j)$ entry of $A_{\!_V}$ and $z_j=e^{i\theta_j}$ for $i,j=1,2,3$.
Then  
\begin{align*}
	\big|\sum a_{ij}z_i\overline{z}_j\big|&=|z_1|^2+|z_2|^2+|z_3|^2
	-2Re\left(z_1\overline{z}_2+z_2\overline{z}_3+z_3\overline{z}_1\right)\\
	&= 3-2(cos(\theta_1-\theta_2)+cos(\theta_2-\theta_3)+cos(\theta_3-\theta_1))\\
	&\leq 3-2\left(\frac{-3}{2}\right)\\
	&= 6. 
\end{align*}
Using the Lemma \ref{minimizing sum of inner products}, the above inequality can easily be deduced . 
For $\theta_1-\theta_2=\frac{2\pi}{3}=\theta_2-\theta_3$, 
the inequality in this computation becomes an equality. 
Thus 
\[\|A_{\!_V}\|_{\ell^{\infty}(3)\to \ell^1(3)}
\geq \sup_{|z_j|=1}\left|\sum a_{ij}z_i\overline{z}_j\right|=6.\]
\end{proof} 

Thus $\|A_{\!_V}\|_{\ell^{\infty}(3)\to \ell^1(3)}>\|p_{\!_V}\|_{\D^3,\infty}$. 
Here $\frac{\|A_{\!_V}\|_{\ell^{\infty}(3)\to \ell^1(3)}}{\|p_{\!_V}\|_{\D^3,\infty}}\geq 1.2.$\\
%

\noindent \textbf{Question:} 
Does there exists $k > 0$ such that 
$\|A\|_{\ell^{\infty}(n)\to \ell^1(n)}\leq k \|p_{\!_{A,\vartriangle}}\|_{\D^n,\infty}$ for all symmetric matrices $A$ of size $n,$
$n\in\N?$

We have just seen that $k$ is bounded below by $1.2.$ Now, we show that $\frac{3\sqrt{3}}{4}$ is an upper bound for $k.$ This will be an immediate corollary of the following theorem giving an upper bound for the second derivative. 
\begin{thm}\label{secondderbd}
If $f:\mathbb D^n \to \mathbb D$ is a holomorphic function, then $\|D^2f(0)\|_{\ell^{\infty}(n)\to \ell^1(n)}$ is bounded above by $\frac{3\sqrt{3}}{2}.$  
\end{thm} 
\begin{proof}
Let $f$ be a complex valued analytic function on $\D^n$ 
with $\|f\|_{\D^n,\infty}\leq 1$. 
Let $a=(a_1,\ldots,a_n)\in \D^n$ be an arbitrary point. 
Let $\Phi_j$ be the automorphism of the unit disc 
defined by 
\[\Phi_j(z)=\frac{z+a_j}{1+\overline{a}_jz}\]
for $j=1,\ldots,n$. 
Let $\Phi(z_1,\ldots,z_n)=(\Phi_1(z_1),\ldots,\Phi_n(z_n))$ 
and $\varphi$ be the automorphism of the unit disc 
such that $\varphi(f(a))=0$. 
Due to chain rule we have 
\[D(\varphi\circ f \circ \Phi)(0)=\varphi '(f(a))Df(a)D\Phi(0).\]
As 
$g:=\varphi\circ f \circ \Phi:\D^n\to \D$ 
is an analytic map 
therefore due to Schwarz's lemma 
$Dg(0)$ is a contractive linear functional on $(\C^n,\|\cdot\|_{\D^n,\infty})$. 
Also 
\[D\Phi(0)=
\left(
\begin{array}{cccc}
	1-|a_1|^2 & 0 & \cdots & 0\\
	0 & 1-|a_2|^2 & \cdots & 0\\
	\vdots & \vdots & \ddots & \vdots\\
	0 & 0 & \cdots & 1-|a_n|^2\\
\end{array}
\right)\]
therefore 
\[Df(a)=\varphi'(f(a))^{-1}\left(\sum_{j=1}^{n}\frac{\partial_jg(0)}{1-|a_j|^2}\right).\]
Thus we have 
\[\left\|Df(a)\right\|_1
\leq (1-|f(a)|^2)\max_{j}\frac{1}{1-|a_j|^2}.\]    
Suppose $r\in (0,1)$ is such that 
$|a_i|<r$ for all $i=1,\ldots,n$. 
Then we have 

\begin{equation}
	\|Df(a)\|_1\leq \frac{1}{1-r^2}.
\end{equation}

\noindent Let $g:=Df$ then $g$ is a map 
from $r\D^n$ to $\frac{1}{1-r^2}(\D^n)^*$ 
where $(\D^n)^*$ denotes the 
dual unit ball of $(\C^n,\|\cdot\|_{\D^n,\infty})$. 
Now due to Schwarz's lemma 
$Dg(0)$ is a linear operator on $\C^n$ 
which maps $r\D^n$ into $\frac{1}{1-r^2}(\D^n)^*$. 
Hence we have 

\begin{equation}\label{Raw bound on second derivative}
	\|D^2f(0)\|_{\ell^{\infty}(n)\to \ell^1(n)}
	\leq \frac{1}{r(1-r^2)}.
\end{equation}

\noindent Inequality \eqref{Raw bound on second derivative} 
is true for every $r\in (0,1)$ 
and maximum of $r(1-r^2)$ is attained at $r=1/\sqrt{3}$. 
Therefore we can conclude that
 
$$	\|D^2f(0)\|_{\ell^{\infty}(n)\to \ell^1(n)}
	\leq \frac{3\sqrt{3}}{2}.$$
\end{proof}  
\noindent Let $p(z_1,z_2,\ldots,z_n)=\sum a_{ij}z_iz_j$ 
be a homogeneous polynomial of degree $2$ in $n$ variables 
with $\|p\|_{\D^n,\infty}\leq 1$. 
As $D^2p(0)=(\!\!(2a_{ij})\!\!)_{n\times n}$ 
therefore from \eqref{secondderbd},  
we have

\begin{equation}\label{Limit to the norm of A}
	\|(a_{ij})\|_{\ell^{\infty}(n)\to \ell^1(n)}
	\leq \frac{3\sqrt{3}}{4}\approx 1.3.
\end{equation}
This leads to a considerable improvement in one of the theorems of \cite{V2},     which is exactly the same as the theorem below except that the constant obtained in \cite{V2} is $2K_{G}^{\C}.$
\begin{thm}
	Suppose $p$ be a polynomial of degree atmost 2 in $n$ variables 
	and $T=(T_1,\ldots,T_n)$ be a tuple of commuting contractions 
	on a Hilbert space $\h$. 
	Then
	\[\|p(T_1,\ldots,T_n)\|
	\leq \frac{3\sqrt{3}}{4}K_{G}^{\C}\|p\|_{\D^n,\infty},\]
	where $K_{G}^{\C}$ is the complex Grothendieck constant.
\end{thm} 
\begin{proof}
Let 
\[p(z_1,\ldots,z_n)=a_0+
\sum_{j=1}^{n}a_jz_j+
\sum_{j,k=1}^{n}a_{jk}z_jz_k.\]
For $x,y\in \h$ arbitrary vectors of norm at most 1, we have 
\[\left|\langle p(T_1,\ldots,T_n)x,y\rangle\right|=
\big|a_0\langle x,y\rangle + \sum_{j=1}^{n}\langle a_jT_jx,y\rangle 
+\sum_{j,k=1}^{n} \langle a_{jk}T_jx,T_{k}^{*}y\rangle\big|.\]
Let 
\[B=
\left(
\begin{array}{ccccc}
	a_0 & a_1/2 & a_2/2 & \cdots & a_n/2\\
	a_1/2 & a_{11} & a_{12} & \cdots & a_{1n}\\
	\vdots & \vdots & \vdots & & \vdots\\
	a_n/2 & a_{n1} & a_{n2} & \cdots & a_{nn}\\
\end{array}
\right)
\]
and $q$ be the corresponding  homogeneous polynomial of degree $2$  
in $n+1$ variables defined by
\[q(z_0,z_1,\ldots,z_n)=a_0 z_{0}^{2}
+\sum_{j=1}^{n}a_jz_jz_0
+\sum_{j,k=1}^{n}a_{jk}z_jz_k.\]
It can easily be seen that $\|q\|_{\D^{n+1},\infty}=\|p\|_{\D^n,\infty}$.  
Suppose $v_0=x,v_j=T_jx$ and $w_0=y,w_j=T_{j}^{*}y$ 
for $j=1,\ldots,n$. 
Then
\[\sum_{j=0}^{n} b_{jk}\langle v_j,w_k \rangle 
= a_0\langle x,y\rangle 
+ \sum_{j=1}^{n}\langle a_jT_jx,y\rangle 
+\sum_{j,k=1}^{n} \langle a_{jk}T_jx,T_{k}^{*}y\rangle,\]
where $b_{jk}$ is the $(j,k)$ entry in $B$. 
Now from the definition of the complex Grothendieck constant, we get
\[\Big|\sum_{j=0}^{n} b_{jk}\langle v_j,w_k \rangle\Big|
\leq 
K_G^\C\|B\|_{\ell^{\infty}(n+1)\to \ell^1(n+1)}.
\]
Now, to complete the proof, one merely has to apply the inequality  \eqref{Limit to the norm of A}.
\end{proof}

\section{Homomorphisms induced by operators of type  \texorpdfstring{${\sf{V\,I}}$}{TEXT}}\label{VOTO}
Let $\Omega$ be a bounded domain in $\mathbb{C}^m$ 
and $\omega=(\omega_1,...,\omega_m)\in\Omega$ be fixed. 
Let $\mathbb{H}$ be a separable Hilbert space. 
Let $x=(x_1,\ldots,x_m),y=(y_1,\ldots,y_m)$, 
where $x_j,y_j\in\mathbb{H}$ for all $j=1,\ldots,m$, 
be such that $[x_{j}^{\sharp},y_k]=[x_{k}^{\sharp},y_j]$ 
for all $j,k=1,\ldots,m$.  
Let the operator $T_{x_j,y_j}$ be of type {\sf{V\,I}} 
corresponding to the pair $x_j,y_j,$ $j=1,\ldots,m.$ We let $\boldsymbol T^{(\omega)}_{x,y}$ denote the commuting $n$-tuple $(\omega_1 I + T_{x_1,y_1}, \ldots, \omega_m I + T_{x_m,y_m}).$ We will let $\boldsymbol T_x$ denote the $m$-tuple $ (T_{x_1,x_1}, \ldots,  T_{x_m,x_m}).$
It is easy to see that for $j,k,l=1,\ldots,m,$ 
we have 
$T_{x_j,y_j}T_{x_k,y_k}T_{x_l,y_l}=0$ 
and
\[T_{x_j,y_j}T_{x_k,y_k}=
\left(
 \begin{array}{ccc}
  	 0 & 0 & \big[x_{j}^{\sharp},y_k\big]\\
  	 0 & 0 & 0\\
  	 0 & 0 & 0\\
  \end{array} 
\right).\]
Consequently, for any polynomial $p$ in $m$ variables, we see that 
\begin{equation} \label{FuncCal(p)}
p(\boldsymbol T^{(\omega)}_{x,y})=\left(
 \begin{array}{ccc}
	p(\omega) & D p(\omega)\cdot x^{\sharp} & \frac{1}{2}D^2p(\omega)\cdot A_{x,y}\\
	0 & p(\omega)I & D p(\omega)\cdot y\\
	0 & 0 & p(\omega)\\
  \end{array}
\right),
\end{equation}
where $x^{\sharp}=\big(x_1^\sharp,\ldots,x_m^\sharp\big),$ 
$A_{x,y}=\big(\!\!\big([x_{i}^{\sharp},y_j]\big)\!\!\big)_{m\times m}$. Therefore, extending this definition to functions in ${\rm H}^\infty(\Omega),$ we obtain the homomorphism  $\rho^{(\w)}_{x,y}:
{\rm H}^\infty(\Omega)\to
\mathcal{B}(\mathbb{C}\oplus\mathbb{H}\oplus\mathbb{C}),$ which for any polynomial $p$ is given by the formula 
$\rho^{(\w)}_{x,y}(p) = p(\boldsymbol T^{(\omega)}_{x,y})$ and is defined for $f$ in ${\rm H}^\infty(\W)$ by the same formula. 
The homomorphism $\rho^{(0)}_{x,x}$ will simply be denoted by $\rho_x$.

Suppose $\W=\D^m$ and 
$\|x_j\|\leq 1,\|y_j\|\leq1$ for each $j=1,\ldots,m$. 
Then for $m=1,2$, 
we know that $\rho^{(\w)}_{x,y}$ is contractive homomorphism. 
What about $m>2$?\\
The following example is due to Varopoulos and Kaijser in \cite{V1}. Set  
\[A_1=
\left(
\begin{array}{ccccc}
	0 & 0 & 0 & 0 & 0\\
	1 & 0 & 0 & 0 & 0\\
	0 & 0 & 0 & 0 & 0\\
	0 & 0 & 0 & 0 & 0\\
	0 & 1/\sqrt{3} & -1/\sqrt{3} & -1/\sqrt{3} & 0\\ 
\end{array}
\right),\]
\[A_2=
\left(
\begin{array}{ccccc}
	0 & 0 & 0 & 0 & 0\\
	0 & 0 & 0 & 0 & 0\\
	1 & 0 & 0 & 0 & 0\\
	0 & 0 & 0 & 0 & 0\\
	0 & -1/\sqrt{3} & 1/\sqrt{3} & -1/\sqrt{3} & 0\\ 
\end{array}
\right)\]
and
\[A_3=
\left(
\begin{array}{ccccc}
	0 & 0 & 0 & 0 & 0\\
	0 & 0 & 0 & 0 & 0\\
	0 & 0 & 0 & 0 & 0\\
	1 & 0 & 0 & 0 & 0\\
	0 & -1/\sqrt{3} & -1/\sqrt{3} & 1/\sqrt{3} & 0\\ 
\end{array}
\right).
\]
It is easy to see that 
$A_1,A_2$ and $A_3$ are commuting contractions.
Now, consider the \VK polynomial $p_{\!_V}$  defined earlier. 
Choose $x_1=(\frac{1}{\sqrt{3}},-\frac{1}{\sqrt{3}},-\frac{1}{\sqrt{3}}),$ 
$x_2=(-\frac{1}{\sqrt{3}},\frac{1}{\sqrt{3}},-\frac{1}{\sqrt{3}}),$ 
$x_3=(-\frac{1}{\sqrt{3}},-\frac{1}{\sqrt{3}},\frac{1}{\sqrt{3}})$
and 
$y_1=(1,0,0),$ $y_2=(0,1,0),$ $y_3=(0,0,1)$.
In the notations above  
$T_{x_1,y_1}=A_1,$ $T_{x_2,y_2}=A_2$ and $T_{x_3,y_3}=A_3$. 
We have 
\[\left\|p_{\!_V}(T_{x_1,y_1},T_{x_2,y_2},T_{x_3,y_3})\right\|
=\Big|\sum\limits_{j,k=1}^{3}a_{jk}
\left[x_{j}^{\sharp},y_k\right]\Big|=3\sqrt{3}> 5 = \|p_{\!_V}\|_{\D^3,\infty},\]
where $\big (\!\!\big ( a_{jk} \big )\!\!\big ) = A_{\!_{V}}.$ 
Hence $\rho^{0}_{x,y}$ corresponding to 
$x=(x_1,x_2,x_3),\,y=(y_1,y_2,y_3)$ is not contractive.
In this example the ratio of 
$\|p_{\!_V}(\boldsymbol T_{x,y}^0 )\|$ to $\|p_{\!_V}\|_{\D^3,\infty}$ is approximately $1.04.$ 
In this section we shall show that 
$$
\sup \Big \{\frac{\|p_{\!_{V}}(\boldsymbol T_x)\|}{\|p_{\!_{V}}\|_{\mathbb D^3,\infty}}: \|x\|_2=1 \Big \} 
$$
is $1.2,$ which was proved earlier by Holbrook \cite{HJ}. However, we give many examples of operators of type {\sf{V\,I}} for which this upper bound is attained.  As explained earlier, the hope that we may be able to increase it even further was the motivation behind introducing the set of operators {\sf{V\,I}.}  For the proof, we shall need the following lemma.
\begin{lem}\label{minimizing sum of inner products}
	For $n>1,$ we have  
	\[\min
	\big(\langle x_1,x_2 \rangle + 
	\langle x_2,x_3 \rangle + 
	\langle x_3,x_1 \rangle\big)
	=-\frac{3}{2}, 
	\]
	where the minimum is over the set $\big\{(x_1,x_2,x_3):x_1,x_2,x_3\in\R^n,\,\|x_i\|_2=1,\,i=1,2,3\big\}.$
\end{lem}
\begin{proof}
Let $x_1,x_2,x_3\in\R^n$ with $\|x_i\|_2=1,$ $i=1,2,3.$ 
The following identity is easily verified:  
\[\|x_1+x_2+x_3\|_2^2=\|x_1\|_2^2+\|x_2\|_2^2+\|x_3\|_2^2 + 
2\big(\langle x_1,x_2 \rangle + 
\langle x_2,x_3 \rangle + 
\langle x_3,x_1 \rangle\big).\]  
For $i=1,2,3,$ $\|x_i\|_2=1,$  
therefore
\[\|x_1+x_2+x_3\|_2^2-3
=2\big(\langle x_1,x_2 \rangle + 
\langle x_2,x_3 \rangle + 
\langle x_3,x_1 \rangle\big).\]
Thus $\langle x_1,x_2 \rangle + 
\langle x_2,x_3 \rangle + 
\langle x_3,x_1 \rangle$ 
is minimized at $x_1,x_2,x_3\in\R^n$ 
such that $x_1+x_2+x_3=0$. 
Choose any three points $x_1,x_2,x_3$ 
from the unit sphere of $\R^n$ such that 
the centroid of these points is the origin. 
For example, choose $x_1=(1,0,\ldots,0)$, 
$x_2=(-1/2,\sqrt{3}/2,0,\ldots,0)$ and $x_3=(-1/2,-\sqrt{3}/2,0,\ldots,0)$. 
Thus we have proved the lemma.
\end{proof}
What follows is an easy generalization of the preceding lemma.
\begin{lem}
	For $n>1,$ we have  
	\[\min\left(\sum\limits_{i<j}\langle x_i,x_j \rangle \right)
	=-\frac{m}{2}, 
	\]
	where minimum is over the set 
	$\big\{(x_1,\ldots,x_m):x_1,\ldots,x_m\in\R^n,\,\|x_i\|_2=1,\,i=1,\ldots,m\big\}.$
\end{lem}
Let $x_1,x_2,x_3\in\R^{n}$ be arbitrary vectors of Euclidean norm $1$ and  
set $x=(x_1,x_2,x_3).$  
Consider the algebra homomorphism 
$\rho_{x}$ as in \eqref{FuncCal(p)}, namely, $\rho_{x}(p) = p(\boldsymbol T_x).$ 
Take \VK polynomial $p_{\!_V}$. By the definition of $\rho_{x},$ 
it is easy to see that 
\begin{align*}
	\left\|\rho_{x}(p_{\!_V})\right\|
	&=\Big|\sum_{j,k=1}^{3}a_{jk}[x_{j}^{\sharp},x_k]\Big|
	=\Big|\sum_{j,k=1}^{3}a_{jk}\langle x_{j},x_k \rangle\Big|\\
	&=\sum_{i=1}^{3}a_{ii} + 2a_{12}\langle x_1,x_2 \rangle + 
	2a_{23}\langle x_2,x_3 \rangle + 2a_{31}\langle x_3,x_1 \rangle\\
	&=3-2\left(\langle x_1,x_2 \rangle + 
	\langle x_2,x_3 \rangle + \langle x_3,x_1 \rangle\right).
\end{align*}
From the Lemma \ref{minimizing sum of inner products}, 
it is clear that 
we can choose $x_1,x_2,x_3\in\R^n$
(in fact there are infinitely many choices for $x$ for each $n>1$) 
such that $\|\rho_x(p_{\!_V})\|=6$ 
and $\|x_i\|_2=1$ for each $i=1,2,3$. 
Thus
\[\frac{\|\rho_{x}(p_{\!_V})\|}{\|p_{\!_V}\|_{\D^3,\infty}}=\frac{6}{5}=1.2>1.\]
Hence for this choice of $x$ 
the corresponding ratio of $\|p_{\!_V}(\boldsymbol T_x)\|$ to $\|p_{\!_V}\|_{\D^3,\infty}$ is 1.2.

\medskip

We state ``the zero lemma{'\!'} for a third time, in the form we will use it here. The proof is no different from what has been indicated earlier. 

\begin{lem}\label{Zero lemma for VOFT}
	For $m-$tuple of vectors, $x=(x_1,\ldots,x_m),y=(y_1,\ldots,y_m)$ 
	from $\h,$ we have, 
	$\|\rho^{(\w)}_{x,y}(f)\|\leq 1$ for all $f\in{\rm H}^\infty(\Omega,\D)$ 
	if and only if 
	$\|\rho^{(\w)}_{x,y}(f)\|\leq 1$ 
	for all $f\in{\rm H}^\infty_\w(\Omega,\D).$
\end{lem}

As before, using the Lemma \ref{Zero lemma for VOFT} 
 we may assume that $f(\omega)=0,$ without loss of generality,   
in determining the contractivity of $\rho^{(\w)}_{x,y}$ for $f$ in 
any algebra of holomorphic functions containing the algebra ${\rm H}^\infty(\Omega).$ 

\begin{prop}\label{contractivity VOFT}
	Let $\rho^{(\w)}_{x,y}$ be as defined in \eqref{FuncCal(p)}. 
	For $f\in{\rm H}^\infty_{\w}(\W,\D)$, 
	we get $\|\rho^{(\w)}_{x,y}(f)\|\leq 1$ if and only if
	\[\left|\frac{1}{2}D^2f(\omega)\cdot A_{x,y}\right|^2\leq 
	\left(1-\left\|D f(\omega)\cdot x\right\|^2\right)
	\left(1-\left\|D f(\omega)\cdot y\right\|^2\right). \]
\end{prop}
\begin{proof}
Let $f\in{\rm H}^\infty_{\w}(\W,\D).$ Let $V_1:\mathbb{H}\to\ell^2$ and 
$V_2:\mathbb{H}^{\sharp}\to(\ell^2)^{\sharp}$ 
be isometries taking 
$D f(\omega)\cdot y$ 
to $\|D f(\omega)\cdot y\|e_1$ 
and 
$D f(\omega)\cdot x^{\sharp}$ 
to $\|D f(\omega)\cdot x\|e_{1}^{\rm t}$ respectively, 
where $e_1$ is $(1,0,0,\ldots)^{\rm t}$. 
Then
\begin{align*}
	\|\rho^{(\w)}_{x,y}(f)\|&=\left\|
	\left(
	\begin{array}{cc}
			D f(\omega)\cdot x^{\sharp} & \frac{1}{2}D^2f(\omega)\cdot A_{x,y}\\
		   0 & D f(\omega)\cdot y\\
	  \end{array}
	\right)
	\right\|\\
	&=\left\|
	\left(
		\begin{array}{cc}
			\frac{1}{2}D^2f(\omega)\cdot A_{x,y} & D f(\omega)\cdot x^{\sharp}\\
			D f(\omega)\cdot y & 0\\
		\end{array}
	\right)
	\right\|.
\end{align*}
As norms are preserved under isometries therefore
\[\|\rho^{(\w)}_{x,y}(f)\|= \left\|
\left(
\begin{array}{cc}
 	  1 & 0\\
  	 0 & V_1\\
\end{array}
\right)
\left(
\begin{array}{cc}
	\frac{1}{2}D^2f(\omega)\cdot A_{x,y} & D f(\omega)\cdot x^{\sharp}\\
	D f(\omega)\cdot y & 0\\
\end{array}
\right)
\left(
\begin{array}{cc}
	   1 & 0\\
	   0 & V_2\\
\end{array}
\right)
\right\|,\]
and hence
\begin{align*}
	\|\rho^{(\w)}_{x,y}(f)\|&=\left\|
	\left(
	\begin{array}{cc}
	   \frac{1}{2}D^2f(\omega)\cdot A_{x,y} & \|D f(\omega)\cdot x\|e_{1}^{\rm t}\\
	   \|D f(\omega)\cdot y\|e_{1} & 0\\
	 \end{array}
	\right)
	\right\|\\
	&=\left\|
	\left(
	\begin{array}{cc}
		   \frac{1}{2}D^2f(\omega)\cdot A_{x,y} & \|D f(\omega)\cdot x\|\\
	 	  \|D f(\omega)\cdot y\| & 0\\
	  \end{array}
	\right)
	\right\|.
\end{align*}
Thus we have the proposition.

\end{proof}
Let
\[\mathcal D_{\Omega}^{(\omega)}
:=\Big\{\Big(\frac{1}{2} D^2f(\omega),D f(\omega)\Big)
\mid f\in{\rm H}^\infty_{\w}(\W,\D)\Big\}\]
be a subset of $M_{m}^{s}\times \mathbb{C}^m$, 
where $M_{m}^{s}$ denotes 
the set of all $m\times m$ complex symmetric matrices.

\begin{lem}
The set	$\mathcal D_{\Omega}^{(\omega)}$ can be realized as the unit ball in $M_{m}^{s}\times \mathbb{C}^m$ with respect to some norm, say $\|\cdot\|_\mathcal D.$
\end{lem}

\begin{proof} We will show that $\mathcal D_{\Omega}^{(\omega)}$ is a balanced, convex and absorbing subset of $M_{m}^{s}\times \mathbb{C}^m.$  
\begin{itemize}
\item\textbf{Balanced:} If $\lambda \in \D$ and 
$\left(\frac{1}{2}D^2f(\omega),
D f(\omega)\right)\in \mathcal D_{\Omega}^{(\omega)},$ 
then
\[\lambda\left(\frac{1}{2}D^2f(\omega),
D f(\omega)\right)
=\left(\frac{1}{2}D^2(\lambda f)(\omega),
D (\lambda f)(\omega)\right).\]
The map  
$\lambda f:\Omega \rightarrow \D$ is  analytic with 
$\lambda f(\omega)=0$ and hence
\[\lambda\left(\frac{1}{2}D^2f(\omega),
D f(\omega)\right)\in \mathcal D_{\Omega}^{(\omega)}.\]

\item\textbf{Convex:} Pick 
	\[\left(\frac{1}{2}D^2f(\omega),D f(\omega)\right),
	\left(\frac{1}{2}D^2g(\omega),D g(\omega)\right)
	\in \mathcal D_{\Omega}^{(\omega)}.\]
	For the $h:=t f+ (1-t)g,$ $t\in (0,1),$ we have 
	\begin{align*}
		&t \left(\frac{1}{2}D^2f(\omega),
		D f(\omega)\right) + 
		\left(1-t\right)\left(\frac{1}{2}D^2g(\omega),
		D g(\omega)\right)\\
		&=\left(\frac{1}{2}D^2 h(\omega),
		D h(\omega)\right).
	\end{align*}

	Since $f,g$ are in ${\rm H}^\infty_\omega(\Omega, \mathbb D),$ it follows that 
	$h$ is also in ${\rm H}^\infty_\omega(\Omega, \mathbb D).$
	Hence
	\[t\left(\frac{1}{2}D^2f(\omega),
	D f(\omega)\right) + 
	(1-t)\left(\frac{1}{2}D^2g(\omega),
	D g(\omega)\right)\in \mathcal D_{\Omega}^{(\omega)}.\] 

\item \textbf{Absorbing:} Let $B=(b_{jk})$ be a symmetric matrix 
	of order $m$ and $a=(a_1,...,a_m)$ in $\mathbb{C}^m$. 
	Define 
	\[p(z_1,z_2,...,z_m)=
	\sum_{j=1}^{m}a_j(z_j-\omega_j)+
	\sum_{j,k=1}^{m}b_{jk}(z_j-\omega_j)(z_k-\omega_k).\]
	The function 
	\[f(z_1,z_2,...,z_m)=\frac{p(z_1,z_2,...,z_m)}{\|p\| _{\W,\infty}}.\]
	is clearly in ${\rm H}^\infty_\omega(\Omega, \mathbb D)$ with 
	\[D f(\omega)=\frac{a}{\|p\|_{\W,\infty}}
	\mbox{ and }
	\frac{1}{2}D^2f(\omega)=\frac{B}{\|p\|_{\W,\infty}}.\]
	Hence
	\[\frac{1}{\|p\|_{\W,\infty}}(B,a)\in \mathcal D_{\Omega}^{(\omega)}.\] 
\end{itemize}
\end{proof}
The set 
\[\mathbb U:=
\left\{(z,v_1,v_2):z\in\mathbb{C},v_1,v_2\in\mathbb{H}
\mbox{ with }
|z|^2\leq \left(1-\|v_1\|^2\right)\left(1-\|v_2\|^2\right)\right\}\]
is seen to be the unit ball via the identification $(z,v_1,v_2) \to  \Big(
\begin{smallmatrix}
	   v_{1}^{\sharp} & z\\
	   0 & v_2\\
  \end{smallmatrix}
\Big ).$ Clearly, $\|\Big(
\begin{smallmatrix}
	   v_{1}^{\sharp} & z\\
	   0 & v_2\\
  \end{smallmatrix}
\Big )\| \leq 1$ if and only if $(z,v_1,v_2)$ is in $\mathbb U.$ Thus we have proved the following lemma.
\begin{lem}
The set 	$\mathbb U$ is the unit ball with respect to the norm 
$\|(z,v_1,v_2)\|_\mathbb U:= \|\Big(
\begin{smallmatrix}
	   v_{1}^{\sharp} & z\\
	   0 & v_2\\
  \end{smallmatrix}
\Big )\|.$
\end{lem}
%

For fixed $x=(x_1,...,x_m),y=(y_1,...,y_m)$ in $\mathbb{H}^m$, 
define a linear map 
$L^{(\w)}_{x,y}:M_{m}^{s}\times \mathbb{C}^m
\to
\C\oplus\mathbb{H}\oplus\h$ 
by the formula 
\[L^{(\w)}_{x,y}(B,a)=\left(\frac{1}{2}\mbox{\rm tr}(A_{x,y}B),a\cdot x,a\cdot y\right),\]
where $a\cdot x= a_1 x_1 + \cdots + a_m x_m,$ $a=(a_1,\ldots , a_m)\in \mathbb C^m$ (and $a\cdot
y$ is defined similarly). Proposition \ref{contractivity VOFT} together with what we have said here amounts to the equivalence asserted in the following theorem.   
\begin{thm}
The following statements are equivalent:
\begin{itemize}
	\item[1.] $\rho^{(\w)}_{x,y}$ is a contractive homomorphism.
	\item[2.] $L^{(\w)}_{x,y}:(M_{m}^{s}\times \mathbb{C}^m,\|.\|_{\mathcal D})
	\to(\C\oplus\mathbb{H}\oplus\h,\|. \|_{\mathbb U})$ 
	is a contractive linear map.
\end{itemize}
\end{thm}

Let $E$ be a domain (containing 0) in $\C$. 
For each $k\in\mathbb{N}_{0},$ let 
\begin{align*}
	\mathcal{P}_{k}(\W,E)=
	\left\{p\in\C[Z_1,\ldots,Z_m]:\deg(p)
	\leq k 
	\mbox{ and }
	p(\W)\subset E\right\}.
\end{align*}

For each $\omega\in\W,$ let 
$\mathcal{P}_{k}^{\omega}(\Omega,E)$ 
denote the set of all polynomials 
$p\in\mathcal{P}_{k}(\Omega,E)$ 
such that $p(\omega)=0$. 

Now, suppose $\Omega$ is the unit disc and $\omega=0.$  Then we have the following theorem. 
\begin{thm}
	If $\mathbb{H}$ is a separable Hilbert space 
	and $x\in\mathbb{H}$ with $\|x\|\leq 1,$  
	then for the homomorphism $\rho_{x},$ 
	we have    
	\[\sup\left\{\|\rho_{x}(p)\|:p\in\mathcal{P}_{1}^{0}(\D,\D)\right\}
	=\sup\left\{\|\rho_{x}(f)\|:f\in{\rm H}^\infty_{0}(\D,\D)\right\}.\]
\end{thm}

\begin{proof}
We know that for $f\in{\rm H}^\infty(\D)$ with $f(0)=0$,
\[\|\rho_{x}(f)\|=\left\|
\left(
	\begin{array}{cc}
	   \frac{1}{2}f{'\!'}(0)[x^{\sharp},x] & \|f'(0)x\|\\
	   \|f'(0)x\| & 0\\
    \end{array}
\right)
\right\|,\]
therefore from the formula in \cite{GMSastry},
\[\|\rho_{x}\|=
\left\{|a|\|x\|^2+\frac{1}{2}\left[|a|^2\left|[x^{\sharp},x]\right|^2
+\sqrt{|b|^4\left|[x^{\sharp},x]\right|^4
+4|a|^2\|x\|^2|b|^2\left|[x^{\sharp},x]\right|^2}\right]\right\}^{\frac{1}{2}},\]
where $a=f'(0)$ and $b=f{'\!'}(0)/2$.
Using Cauchy-Schwarz inequality we get
\[\|\rho_{x}\|\leq 
\left\{|a|\|x\|^2+
\frac{1}{2}\left[|a|^2\|x\|^4+
\sqrt{|b|^4\|x\|^8+4|a|^2|b|^2\|x\|^6}\right]\right\}^{\frac{1}{2}}\]
and therefore
\[\|\rho_{x}\|\leq 
\|x\|\left\{|a|+\frac{1}{2}\left[|a|^2
+\sqrt{|b|^4+4|a|^2|b|^2}\right]\right\}^{\frac{1}{2}}.\]
Hence $\sup\left\{\|\rho_{x}(f)\|:f\in{\rm H}^\infty_{0}(\D,\D)\right\}=\|x\|.$ 
It is easy to see that
\[\sup\left\{\|\rho_{x}(p)\|:p\in\mathcal{P}_{1}^{0}(\D,\D)\right\}=\|x\|.\]
Hence the proof is complete.
\end{proof}
The following corollary is now evident. 
\begin{cor}
	Suppose $\mathbb{H}$ is a separable Hilbert space 
	and $x\in\mathbb{H}$ with $\|x\|=1$. 
	Then for homomorphism $\rho_{x}$ defined above, 
	we get 
	\[\sup\left\{\|\rho_{x}(p)\|:p\in\mathcal{P}_{1}^{0}(\D,\D)\right\}
	=\sup\left\{\|\rho_{x}(f)\|:f\in{\rm H}^\infty(\D,\D)\right\}.\]
\end{cor}
\section{The Carath\'{e}odory-Fej\'{e}r Interpolation Problems}
We state the well known interpolation problem in $m$ variables, usually known as the Carath\'{e}odory-Fej\'{e}r (CF) problem. 
\begin{prob}[CF] \label{CF}
Given any polynomial $p$ in $m$ variables of degree $d$, find necessary and sufficient conditions on the co-efficients of $p$ to  ensure the existence of a holomorphic function $h$ defined on the polydsic $\mathbb D^m$ with $h^{(k)}(0)=0$ for all multi indices $k$ of length at most $d,$ such that 
$f:=p+h$ maps the polydisc $\mathbb D^m$ to the unit disc $\mathbb D.$   
\end{prob} 
Without loss of generality one may assume that 
$p(0)=0$ via the transitivity of the unit disc  $\D.$ 
There are several different known solutions to the CF problem when $n=1,$ see (cf. \cite[Page 179]{Nik}). However, repeated attempts to obtain solutions for $n>1$ has remained unsuccessful for the most part, however see (cf. \cite[Chapter 3]{BMW}) for a comprehensive survey of recent results.  
In these notes we shall obtain necessary condition for the CF problem  
for the bi-disc $\D^2.$ (However,  we first discuss the case of the unit disc $\mathbb D,$ which paves the way for the case of the bi-disc $\mathbb D^2.$)  We show that for certain class of polynomials of degree at most $2,$ our necessary conditions turn out to be sufficient as well. None the less, they are not always sufficient as we demonstrate by means of an example.  

We point out that the necessary condition for the CF problem actually works for any $n,$ via an adaptation of a theorem due to Kor\'{a}nyi and Puk\'{a}nszky \cite{KP}. However for $n>2,$ the computations involved in deriving the necessary condition explicitly is cumbersome. Therefore, 
we don't give the details except in the case $n=2.$ 

\subsection{CF problem in one variable}
The CF problem for one variable is stated below for polynomials of degree at most two and with $p(0)=0.$ This is the first non-trivial case of the CF problem and is typical of all other cases.  
\begin{prob} \label{Planar extension}
	Fix $p$ to be the polynomial  $p(z)=az+bz^2.$  
	Find a necessary and sufficient condition 
	for the existence of a holomorphic function $g$ defined on the unit disc $\mathbb D$  with $g^{(k)}(0) = 0,\, k=0,1,2,$ 
	such that  $\|p + g \|_{\D,\infty}\leq 1.$ 
\end{prob}

\medskip

Let
$T_x$ be an operator of the type {\sf{V\,I}} for some $x\in \C.$ 
For any $f\in{\rm H}^\infty_{0}(\D,\D),$ picking $|x| \leq 1$ to ensure contractivity of $T_x,$  we see that 
$\|\rho_{x}(f)\| \leq 1.$ Now, applying 
Corollary \ref{contractivity VOFT}, we find that 
\[\left|\frac{1}{2}f{'\!'}(0)x^2\right|+\left|f'(0)x\right|^2 \leq 1.\]
Taking supremum over all $x$ such that $|x| \leq 1,$ 
we get
\[\left|\frac{1}{2}f{'\!'}(0)\right|+\left|f'(0)\right|^2 \leq 1,\]
which is equivalent to
\[\left\|\mathscr{T} \left(f'(0),\frac{f{'\!'}(0)}{2}\right)\right\| := \Big \|\Big ( \begin{matrix} f^\prime(0) & \frac{f^{\prime\!\prime}(0)}{2} \\
0&  f^\prime(0) \end{matrix} \Big ) \Big \| \leq 1.\]
Thus we have proved the following theorem. 

\begin{thm}
	Suppose $f:\D\to\D$ is an analytic function with $f(0)=0.$ Then
	\[\left\|\mathscr{T} \left(f'(0),\frac{f{'\!'}(0)}{2}\right)\right\| \leq 1.\]
\end{thm}
We answer  the question of the converse in the theorem below. 
\noindent
\begin{thm}\label{single variable}
	If $\alpha,\beta\in\mathbb{C}$ are 
	such that $\|\mathscr{T}(\beta,\alpha)\|\leq 1,$ 
	then there exists an analytic map $f:\D\rightarrow \D$ 
	such that $f(0)=0,f'(0)=\beta$ and $f{'\!'}(0)/2=\alpha$.
\end{thm}
\begin{proof}
Let $\alpha,\beta\in\mathbb{C}$ be such that 
$\|\mathscr{T}(\beta,\alpha)\|\leq 1$ 
i.e. $|\alpha|+|\beta|^2\leq 1$. 
As $\mathcal D_{\D}^{0}$ is a convex and balanced set 
so without loss of generality we assume $\alpha>0$ 
and $|\alpha|+|\beta|^2=1$. 
Define
\begin{eqnarray*}
	g(z):= \left\{
	\begin{array}{ll}
	      \beta & if~z=0\\
	      \frac{f(z)}{z} & otherwise. \\
	\end{array} 
	\right.
\end{eqnarray*}
Let $\phi_{\beta}$ denote the automorphism of $\D$ 
mapping $\beta$ to $0$. 
From chain rule, we get 
\[\left(\phi _{\beta}\circ g\right)'(0)
= \frac{\alpha }{1-|\beta |^2}=1.\]
Hence 
$(\phi _{\beta}\circ g)(z)= e^{i\theta}z$ 
for some $\theta \in [0,2\pi)$. 
Therefore
\[g(z)= \frac{e^{i\theta}z+\beta}{1+\overline{\beta}e^{i\theta}z}.\]
and thus
\[f(z)=z\cdot\frac{e^{i\theta}z+\beta}{1+\overline{\beta}e^{i\theta}z}.\]
\end{proof}
Thus we have found necessary and sufficient condition for the CF problem \ref{Planar extension}.  
A second approach to this problem will be given in Chapter \ref{Koranyi}.

\subsection{CF interpolation problem in two variables}
The complete solution to the CF problem remains a mystery, although, several different partial answers are known. On the other hand, 
Eschmeier, Patton and Putinar \cite{EPP} find a necessary and 
sufficient condition for the CF problem for the bi-disc $\mathbb D^2.$ However, these conditions are somewhat intractable.
\begin{thm}
	Let $d$ be a positive integer and 
	let $P(z)$ be a polynomial of degree less than or equal to $d$ 
	in two complex variables. 
	There exists an analytic function $F:\D^2\to \D$ 
	such that $P \equiv F \mod(z^{d+1})$ 
	if and only if there are Hilbert spaces $\h_1$ and $\h_2$ 
	and a pair of vector valued polynomial functions 
	of degree less than or equal to $d$, 
	$A_k:\D^2\to \h_k$, $k = 1,2$, 
	such that:
	\[1-P(z)P(z)\equiv (1-|z_1|^2)\|A_1(z)\|^2_1+(1-|z_2|^2)\|A_2(z)\|^2_2
	 \mod(z^{d+1},\overline{z}^{d+1}).\] 
\end{thm}

Analogous to the case of one variable, the CF problem in the case of two variables is given below for polynomials of degree at most two with constant term zero. 
This is typical of all other cases.
\begin{prob}\label{two Variables}
	Fix $p\in\mathbb{C}[Z_1,Z_2]$ to be the polynomial  
	\[p(z_1,z_2)=
	a_{1,0}z_1+a_{0,1}z_2+
	a_{2,0}z_{1}^{2}+a_{1,1}z_1z_2+a_{0,2}z_{2}^2.\]
	Find necessary and sufficient conditions 
	for the existence of a holomorphic function 
	$q$ on $\mathbb{D}^2$ with $q^{(k)}(0)=0,$ 
	for multi indices $k$ of length at most $2$,
	such that $\|p+q\|_{\D^2,\infty}\leq 1$. 
\end{prob}

\medskip

Let $T_{x_1}$, $T_{x_2}$ 
be operators of type {\sf{V\,I}} for  $x_1,x_2$ in $\mathbb C,$ 
$|x_1|, |x_2| \leq 1.$ 
Let $f\in \mbox{\rm H}^\infty_0(\mathbb D^2, \mathbb D)$ be any holomorphic function mapping $\mathbb D^2$ to $\mathbb D$ with $f(0) = 0.$  
The von-Neumann inequality in 
Theorem \ref{Theorem for zero as only eigenvalues}  
implies that $\|f(T_{x_1},T_{x_2})\|\leq 1,$ 
which in turn is equivalent to
\[\|D f(0)\cdot x\|^2+ 
\Big|{\rm tr}\left(\frac{1}{2} D^2f(0)\cdot xx^{\rm t}\right)\Big| \leq 1,\] where $x$ is the column vector $\binom{x_1}{x_2}.$
Thus we have the following theorem.

\begin{thm}\label{1var necessary condition}
	If $p$ is any complex valued polynomial in two variables 
	of degree at most $2$ with $p(0)=0,$ then
	\begin{equation}\label{second derivative and VOFT}
		\sup\limits_{x_1,x_2\in\D}\left\|\mathscr{T} \left( \frac{\partial p}{\partial z_1}(0)x_1 
		+ \frac{\partial p}{\partial z_2}(0)x_2,\frac{1}{2}
		\sum_{i,j=1}^{2}\frac{\partial ^2 p}
		{\partial z_i \partial z_j}(0) x_i x_j\right)\right\|
		\leq 1
	\end{equation}
	is a necessary condition 
	for the existence of a holomorphic function $q:\mathbb D^2 \to \mathbb C,$ with $q^{(k)}(0)=0,$ $|k| \leq 2,$ such that  $\|p+q\|_{\mathbb D^2,\infty} \leq  1.$
\end{thm}
In the Chapter \ref{Extremal} we will compute the supremum occurring 
in \eqref{second derivative and VOFT}. We will also  
find conditions on the coefficients of the polynomial $p,$ apart from the ones imposed by \eqref{second derivative and VOFT}, which will ensure the existence of the required function $q.$

\chapter{Varopoulos Operators of Type II}\label{Extremal}
\section{Homomorphisms induced by operators of type \texorpdfstring{{\sf{V\,II}}}{TEXT}~ and order two}
Let $\Omega$ be a bounded domain in $\mathbb{C}^m$ 
and $\omega=(\omega_1,\ldots,\omega_m)\in\Omega$ be fixed. 
Let $\mathbb{H}$ be a separable Hilbert space 
and $ X=(X_1,\ldots,X_m)$ be a tuple of commuting contractions, $X_j\in\mathcal{B}(\mathbb{H}),$ $j=1,\ldots,m.$ 
Let $T_{X_j}$ be of type {\sf{V\,II}} and of order $2,$ $j=1,\ldots,m.$ 
Let $\boldsymbol T_X$ be the $n$-tuple $(\omega_1 I + T_{X_1}, \ldots , \omega_m I + T_{X_m}).$
For $j,k,l=1,\ldots,m,$ we have $T_{X_j}T_{X_k}=T_{X_k}T_{X_j},$ 
$T_{X_j}T_{X_k}T_{X_l}=0$ 
and
\[T_{X_j}T_{X_k}=
\left(
 \begin{array}{ccc}
	   0 & 0 & X_jX_k\\
	   0 & 0 & 0\\
	   0 & 0 & 0\\
  \end{array} 
\right).\]
Let $A_{X}$ denote the block matrix 
$\big(\!\!\big(X_jX_k\big)\!\!\big)_{m\times m}$ of operators. 
Consequently, for any polynomial $p$ in $m$ variables, we see that 
\[p(\boldsymbol T\!_X)=\left(
 \begin{array}{ccc}
	   p(\omega)I & D p(\omega)\cdot X & \frac{1}{2}D^2p(\omega)\cdot A_X\\
	   0 & p(\omega)I & D p(\omega)\cdot X\\
	   0 & 0 & p(\omega)I\\
  \end{array}
\right).
\] 
Therefore, extending this definition to functions in 
${\rm H}^\infty(\Omega),$ 
we obtain the algebra homomorphism 
$\mu^{(\w)}_{X}:{\rm H}^\infty(\Omega)\to\mathcal{B}(\h\oplus\h\oplus\h),$ 
which for any polynomial $p$ is given by the formula 
$\mu^{(\w)}_{X}(p) = p(\boldsymbol T\!_X)$ and is defined for $f$ in ${\rm H}^\infty(\W)$ by the same formula.
Suppose $\W$ is the polydisc $\D^m$. Then, 
for $m=1,2$, we know that 
$\mu^{(0)}_{X}:=\mu_X$ is a contractive homomorphism. 
What about $m>2$?

Consider the operators $A_1,A_2$ and $A_3$ 
as defined in the Section \ref{VOTO}. 
Consider $T_{A_1},T_{A_2}$ and $T_{A_3}$, 
the operators of type {\sf{V\,II}} and of order $2.$ 
Consider the \VK polynomial $p_{\!_V}$.
From the computation in \cite{V1}, 
we get
\[\left\|p_{\!_V}(T_{A_1},T_{A_2},T_{A_3})\right\|
=\Big\|\sum\limits_{j,k=1}^{3}a_{jk} A_k A_j\Big\| 
= 3\sqrt{3},\]
where $\big (\!\!\big ( a_{jk} \big )\!\!\big ) = A_{\!_V}.$
Therefore
\[\left\|p_{\!_V}(T_{A_1},T_{A_2},T_{A_3})\right\|>
\|p_{\!_V}\|_{\D^3,\infty}=5.\]
Hence $\mu_{ X}$ corresponding to the tuple 
$ X=(A_1,A_2,A_3)$ of commuting contractions 
$A_1,$ $A_2$ and $A_3$ is not contractive.\\
We need a version of ``the zero lemma'' one final time which is  adapted to apply directly to the functional calculus for operators of the type {\sf V\,II}. This variant is also proved exactly the same way as before. 
\begin{lem}\label{Zero lemma for VOTT}
	The homomorphism $\mu^{(\w)}_{X}$ is contractive  
	if and only if $\|\mu^{(\w)}_{ X}(f)\|\leq 1$ 
	for all $f$ in ${\rm H}^\infty_{\w}(\W,\D)$.
\end{lem}
If $\Omega$ is the unit disc $\D$ 
and $\omega=0,$ then we have the following theorem.

\begin{thm}
	Suppose $\mathbb{H}$ is a separable Hilbert space 
	and $X\in\mathcal{B}(\mathbb{H})$ with $\|X\|\leq 1$. 
	Then for homomorphism $\mu_{ X},$ we get 
	$\sup\left\{\|\mu_{X}(p)\|:p\in\mathcal{P}_{1}^{0}(\D,\D)\right\}
	=\sup\left\{\|\mu_{\boldsymbol X}(f)\|:f\in{\rm H}^\infty_{0}(\D,\D)\right\}.$
\end{thm}

\begin{proof}
We have 
$\sup\left\{\|\mu_{X}(p)\|:p\in\mathcal{P}_{1}^{0}(\D,\D)\right\}
= \|X\|$ by definition.  
Fix $f\in{\rm H}^\infty_{0}(\D,\D),$ and assume that $f$ is represented in the unit disc $\mathbb D$ by the convergent power series 
$\sum_{j=1}^{\infty}a_n z^n.$ 
Then 
\[\|\mu_{X}(f)\|=\left\|\mathscr{T}\left(a_1 X,a_2 X^2 \right)\right\|
\leq 
\|X\| \left\|\mathscr{T}\left(a_1 I,a_2 X\right)\right\|.\]
Also
$\|f\|_{\D,\infty}
\geq 
\left\|\mathscr{T}\left(a_1,a_2\right)\right\|,$
therefore
\[\frac{\|\mu_{X}(f)\|}{\|f\|_{\D,\infty}}
\leq 
\frac{\|X\|\left\|\mathscr{T}\left(a_1I,a_2 X\right)\right\|}
{\left\|\mathscr{T}\left(a_1,a_2\right)\right\|}
\leq 
\|X\|.\]
Hence 
\[\sup\left\{\|\mu_{X}(p)\|:p\in\mathcal{P}_{1}^{0}(\D,\D)\right\}
=\sup\left\{\|\mu_{X}(f)\|:f\in{\rm H}^\infty_{0}(\D,\D)\right\}=\|X\|.\]
\end{proof}

For two commuting contractions $X_1,X_2\in\mathcal{B}(\h),$ let $X=(X_1,X_2)$
and $T_{X_1}$, $T_{X_2}$ be the operators of type {\sf{V\,II}} and  order 2. Setting $\Omega=\D^2,$ and $\omega=0,$ 
we  see that the homomorphism $\mu_{X}$ is contractive via Ando's theorem.   
Hence for $f\in{\rm H}^\infty_{0}(\D^2,\D)$, we have 
\[\sup_{X}\left\|\mathscr{T}\left(D f(0)\cdot  X,
\frac{1}{2}D^2f(0)\cdot A_{ X}\right)\right\|
\leq 1,\]
where the supremum is taken over all pairs of commuting contractions $X=(X_1,X_2).$ 
Thus we have proved the following theorem.

\begin{thm}\label{var necessary condition}
	Let $f\in{\rm H}^\infty_{0}(\D^2,\C)$. 
	Then
	\[\sup\left\|\mathscr{T}\left(\frac{\partial f}{\partial z_1}(0)X_1 + 
	\frac{\partial f}{\partial z_2}(0)X_2,~ 
	\frac{1}{2}\sum_{i,j=1}^{2}\frac{\partial ^2 f}
	{\partial z_i \partial z_j}(0) X_iX_j \right)\right\|
	\leq 1,\]
	where the supremum is taken over all pairs of commuting contractions 
	$X_1,X_2\in\mathcal{B}(\h),$
	is a necessary condition for $f$ 
	to map $\mathbb{D}^2$ to $\mathbb{D}$.
\end{thm}

In what follows, 
we show that 
the supremum in \eqref{second derivative and VOFT}
is the same as the one appearing in Theorem \ref{var necessary condition}.
We then proceed to compute it explicitly. 
Let $(B)_1$  denote 
the open unit ball of the Banach space $B.$ 
Let $\W$ be a domain in $\C^m$ and $k\in\N$. 
We shall denote the set of all 
$M_k-$valued polynomials in $m$ variables by $\mathcal{P}(\C^m,M_k)$. 
The symbol $\mathcal{P}_{n}(\C^m,M_k)$ will denote 
the set of all polynomials in $\mathcal{P}(\C^m,M_k)$ 
which are of degree at most $n$. 
For any $\w\in\W$, we shall denote 
$\mathcal{P}_{n}^{(\w)}(\C^m,M_k):=
\left\{p\in\mathcal{P}_{n}(\C^m,M_k): p(\w)=0 \right\}$ 
and 
$\mathcal{P}_{n}^{(\w)}\left(\W,(M_k)_1\right):=
\left\{p\in\mathcal{P}_{n}^{(\w)}(\C^m,M_k):\|p\|_{\W,\infty}^{\rm op}\leq 1 \right\},$ 
where $\|p\|_{\W,\infty}^{\rm op}=\sup\{\|p(z)\|_{op}:z\in\W\}$.
 
Our attempt here to find a solution to the extremal problem stated 
in the Theorem \ref{var necessary condition} leads naturally to an independent verification of the von-Neumann inequality for pairs of commuting contractions of type {\sf{V\,II}.} 

\section{The von-Neumann Inequality in One variable}\label{1varvon}
In what follows the following lemma is needed.
\begin{lem}\label{VN for MVP}
	If $F(z)=A_1 + A_2z + A_3z^2+\cdots $ 
	is an analytic map on $\mathbb{D}$ taking values in $(M_k)_1,$ 
	then $\mathscr{T}(A_1,A_2)$ has norm at most 1. 
\end{lem}
\begin{proof}
Suppose $\phi _{-A_1}:\big(M_k\big)_1\to \big(M_k\big)_1$ is an analytic map 
defined by 
\[\phi_{-A_1}(C)=\big(I-A_1 A^{*}_{1}\big)^{-\frac{1}{2}}
(C-A_1)\big(I-A^{*}_{1} C\big)^{-1}\big(I-A^{*}_{1} A_1\big)^{\frac{1}{2}}.\]
Then $\phi_{-A_1}\circ F$ maps 
$\mathbb{D}$ to $\big(M_k\big)_1$ 
with $\big(\phi _{-A_1}\circ F\big)\big(0\big)=0$.
\[\big(\phi _{-A_1}\circ F\big)'\big(0\big)=\phi _{-A_1}'\big(F(0)\big)F'(0)
=\big(I-A_1A^{*}_{1}\big)^{-\frac{1}{2}} A_2 \big(I-A^{*}_{1} A_1\big)^{-\frac{1}{2}}\]
Schwarz's lemma implies that 
$\big(\phi _{-A_1}\circ F\big)'\big(0\big)$ is a contractive linear map 
from $\C$ to $M_k$ and 
therefore
\[\left\|\left(I-A_1 A^{*}_{1}\right)^{-\frac{1}{2}}
 A_2 \left(I-A^{*}_{1} A_1\right)^{-\frac{1}{2}}\right\|
 \leq 1.\]
Now due to Parrott's theorem, 
we conclude that $\|\mathscr{T}(A_1,A_2)\|\leq 1$.
\end{proof}
The theorem below proves the von-Neumann inequality (involving matrix valued polynomials)  
for operators of  type ${\sf{V\,II}}$ and order $2.$   

The homomorphism $\mu_X^{(\omega)}$ naturally extends to the algebra $\mbox{\rm H}^\infty(\Omega) \otimes M_k$ by tensoring with 
the identity map $I_k$ on the $k\times k$ matrices $M_k.$ Thus $$\mu_X^{(\omega)}\otimes I_k: \mbox{\rm H}^\infty(\Omega) \otimes M_k\to \mathcal B(\mathbb H\otimes \C^3) \otimes M_k$$ is given by the formula $\mu_X^{(\omega)}\otimes I_k(F):=\big ( \!\!\big (\mu_X^{(\omega)}(F_{ij})\big )\!\!\big ),$ where $F = \big ( \!\!\big (F_{ij}\big )\!\!\big ) \in \mbox{\rm H}^\infty(\Omega) \otimes M_k.$ In particular, for any $F(z) = A_0 + A_1 z + A_ 2 z^2 + \cdots ,$ it follows that   
\begin{eqnarray}\label{FuncTensor}
\mu_X^{(\omega)}\otimes I_k(F) &=& A_0 \otimes I_k + A_1 \otimes T_X + A_2 \otimes T_X^2 + \cdots \nonumber\\
&=&  \left(
\begin{smallmatrix}
     A_0 & A_1\otimes X & A_2\otimes X^2\\
      0 & A_0 & A_1\otimes X \\
      0 & 0 & A_0\\
\end{smallmatrix} 
\right)
\end{eqnarray}
\begin{thm}
	Let $X$ be a contraction on some Hilbert space $\h$ 
	and $T_{X}$ be the operator of type {\sf{V\,II}} and order 2. 
	If $P\in\mathcal{P}_{n}\left(\D,(M_k)_1\right)$ 
	then, $\|P(T_X)\|\leq 1$.
\end{thm}
\begin{proof}
The zero lemma \ref{Zero lemma for VOTT} is easy to prove for the homomorphism $\mu^{(\omega)}_X\otimes I_k.$ The proof now involves finding an automorphism of the unit ball (with respect to the operator norm) in the $k\times k$ matrices taking $P(\omega)$ to $0.$ Since this group of  automorphisms is known to be transitive, the proof of the zero lemma even for matrix valued polynomials is same in spirit to the ones given before. We therefore assume, without loss generality, that $P(0)=0.$ 
Thus for any polynomial $P(z)=A_1z + A_2z^2+ \cdots + A_nz^n,$ we have  
\[P(T_X):= \mu_X\otimes I_k(P) = \left(
\begin{array}{ccc}
      0 & A_1\otimes X & A_2\otimes X^2\\
      0 & 0 & A_1\otimes X \\
      0 & 0 & 0\\
\end{array} 
\right)
\]
and hence 
\begin{align*}
	\left\| P(T_X)\right\|&=\left\| \left(
	\begin{array}{cc}
	      A_1\otimes X & A_2\otimes X^2\\
	     0 & A_1\otimes X \\
	\end{array} 
	\right)
	\right\|\\
	&=
	\left\| \left(
	\begin{array}{cc}
	      I\otimes X & 0\\
	     0 &I \\
	\end{array}
	\right)
	\left(
	\begin{array}{cc}
	      A_1\otimes I & A_2\otimes I\\
	     0 & A_1\otimes I \\
	\end{array}
	\right)
	\left(
	\begin{array}{cc}
	      I & 0\\
	     0 & I\otimes X \\
	\end{array} 
	\right)
	\right\|.
\end{align*}
Since $\|X\|\leq 1,$ it follows that  
$\left\| P(T_X)\right\|\leq \left\|\mathscr{T}(A_1,A_2)\right\|.$
Now using the Lemma \ref{VN for MVP},  
we get $\left\| P(T_X)\right\|\leq 1.$
\end{proof}

\section{Ando's Theorem for the Operators of type \texorpdfstring{{\sf{V\,II}}}~~and order 2}
The zero lemma \ref{Zero lemma for VOTT}, modified as in the previous section, applies to the case of matrix valued polynomials in any number of variables. In this form, it is stated in the papers \cite{GMSastry,  VP}. Consequently it is enough, without loss of generality, that the homomorphisms we consider below are defined only on polynomials with $P(0) =0.$  
We now recall the commutant lifting theorem (cf.\cite[Theorem 4]{DMP}), which we use in the proof of the lemma below.
\begin{thm}
	Let $T$ be a contraction on a Hilbert space $\h$, 
	$U$ be its minimal co-isometric dilation 
	acting on some Hilbert space $\mathbb{K}$, 
	and $R$ be an operator on $\h$ commuting with $T$. 
	Then there is an operator $S$ on $\mathbb{K}$ 
	commuting with $U$ such that
	\[S\h\subset\h,\,\|S\| = \|R\|\,\mbox{ and } R^{m} T^{n} = P_{\h} S^m U^n \vert_{\h} \; \forall m,n \geq 0.\]
\end{thm}

Let $X_1$ and $X_2$ be two commuting contractions 
on a Hilbert space $\h$ and 
$T_{X_1},T_{X_2}$ be the operators of type {\sf{V\,II}} and order $2$ respectively.  
Let
\[P(z_1,z_2)=\sum_{k=1}^{n}\sum_{p+q=k} A_{pq}z_{1}^{p}z_{2}^{q}\]
be in 
$\mathcal{P}_{n}^{(0)}\left(\C^2,M_l\right)$. 
Evaluating the polynomial $P$ on the commuting pair of contractions $T_{X_1},T_{X_2},$ we get 
\[P(T_{X_1},T_{X_2})=\left(
\begin{array}{ccc}
      0 & \sum\limits_{p+q=1} A_{pq}\otimes X_{1}^{p}X_{2}^{q} & \sum\limits_{p+q=2}A_{pq}\otimes 	X_{1}^{p} X_{2}^{q}\\
      0 & 0 &\sum\limits_{p+q=1} A_{pq}\otimes X_{1}^{p}X_{2}^{q} \\
      0 & 0 & 0\\
\end{array} 
\right)\]
and hence 
\[\left\| P(T_{X_1},T_{X_2})\right\| =\left\| \left(
\begin{array}{cc}
      \sum\limits_{p+q=1} A_{pq}\otimes X_{1}^{p}X_{2}^{q} & \sum\limits_{p+q=2}A_{pq}\otimes 		X_{1}^{p} X_{2}^{q}\\
     0 &\sum\limits_{p+q=1} A_{pq}\otimes X_{1}^{p}X_{2}^{q} \\
\end{array} 
\right)
\right\| .\]

\begin{lem}\label{X_1 can be taken to be unitary}
	For any polynomial $P$ in $\mathcal{P}_{n}^{(0)}\left(\C^2,M_l\right),$
	$\|P(T_{X_1},T_{X_2})\|\leq 1$ 
	for all commuting contractions $X_1,~X_2$ 
	if and only if $\|P(T_{X_1},T_{X_2})\|\leq 1$ 
	for all commuting pairs $X_1,~X_2$ 
	with $X_1$ is co-isometry and 
	$X_2$ contractive.
\end{lem}
\begin{proof}
Let $X_1$ and $X_2$ be any two commuting contractions.
 Let $U:\mathbb{K}\to \mathbb{K}$ be the minimal co-isometric dilation of $X_1$. 
 From commutant lifting theorem there exists an operator 
 $S:\mathbb{K}\to \mathbb{K}$ 
 such that
\[\|S\|=\|X_2\|,~SU=US 
\mbox{ and }
X_{2}^{m}{X_1}^n=P_{\mathbb{H}}S^{m}{U^n}|_{\mathbb{H}} 
\mbox{ for all }m,n\in\mathbb{N}_0.\]
Let $T_U$ and $T_S$ be the operators of type {\sf{V\,II}} and order $2.$
Setting  
$\widetilde{\h}=\mathbb{C}^l\otimes \mathbb{H}\oplus \mathbb{C}^l\otimes \mathbb{H},$ we have 
\[P_{\widetilde{\h}}\mathscr{T}\left(A_{10}\otimes U + A_{01}\otimes S,
A_{20}\otimes U^2 + A_{11}\otimes US + A_{02}\otimes S^2\right)_{\big{|}{\widetilde{\h}}}\]
\[=\mathscr{T}\left(\sum\limits_{p+q=1} A_{pq}\otimes X_{1}^{p}X_{2}^{q},
\sum\limits_{p+q=2}A_{pq}\otimes X_{1}^{p} X_{2}^{q}\right).\]
Thus for any polynomial $P$ of the form 
\[P(z_1,z_2)=\sum_{k=1}^{n}\sum_{p+q=k} A_{pq}z_{1}^{p}z_{2}^{q}\]
mapping $\D^2$ into $(M_l)_1,$, we have
$\|P(T_U,T_S)\|\geq \|P(T_{X_1},T_{X_2})\|$
completing the proof of the lemma. 
\end{proof}
\begin{thm}\label{Ando}
	Let $T_{X_1}$ and $T_{X_2}$ be commuting contractions 
	of type {\sf V\,II} and order $2.$ We have 
	 	 $\|P(T_{X_1},T_{X_2})\|\leq \|P\|_{\D^2,\infty}^{\rm op},$
for any matrix valued polynomial $P$ in two variables. 
\end{thm}
\begin{proof} 
Let $P$ be a polynomial in two variables of the form 
\[P(z_1,z_2)=\sum_{k=1}^{n}\sum_{p+q=k} A_{pq}z_{1}^{p}z_{2}^{q},\, A_{pq} \in M_l,\]
with $\|P\|_{\D^2,\infty}^{\rm op}\leq 1.$ 
For $\lambda\in\D$,
\[p_{\lambda}(z_1):=P(z_1,\lambda z_1)=
(A_{10}+ A_{01}\lambda)z_1 + (A_{20} + A_{11}\lambda + A_{02}\lambda ^2){z_1}^2+...\]
maps 
$\mathbb{D}$ to $(M_l)_1$. 
Therefore for each $\lambda\in\D$,
\[\left\| \left(
\begin{array}{cc}
      A_{10}+ A_{01}\lambda & A_{20} + A_{11}\lambda + A_{02}\lambda ^2\\
     0 & A_{10}+ A_{01}\lambda \\
\end{array} 
\right)
\right\| 
\leq 1 ,\]
which is equivalent to
\[\left\| \mathscr{T}\left(A_{10},A_{20}\right)+
\mathscr{T}\left(A_{01},A_{11}\right)\lambda+
\mathscr{T}\left(0,A_{02}\right)\lambda^2\right\| 
\leq 1,\]
for all $\lambda\in\D$. 
Define $f:\mathbb{D}\to (M_{2l})_1$ by 
\[f(\lambda)=\mathscr{T}\left(A_{10},A_{20}\right)+
\mathscr{T}\left(A_{01},A_{11}\right)\lambda+
\mathscr{T}\left(0,A_{02}\right)\lambda^2.\]
Let $X_1$ be a co-isometry operator 
and $X_2$ be an arbitrary contraction 
such that $X_1X_2=X_2X_1$. 
If $Y:=X_2X_{1}^{*},$ then
\[f(Y)=\mathscr{T}\left(A_{10}\otimes I + 
A_{01}\otimes Y, ~A_{20}\otimes I +
A_{11}\otimes Y + A_{02}\otimes Y^2\right).\]
An easy computation gives
\begin{align*}
	&\left(
	\begin{array}{cc}
	      A_{10}\otimes X_1 + A_{01}\otimes X_2 & A_{20}\otimes {X_1}^2 + A_{11}\otimes X_2 X_1 		+ A_{02}\otimes {X_2}^2\\
	     0 &  A_{10}\otimes X_1 + A_{01}\otimes X_2 \\
	\end{array} 
	\right)\\
	&=
	\left(
	\begin{array}{cc}
	     I\otimes X_1 & 0\\
	     0 & I \\
	\end{array} 
	\right)
	f(Y)
	\left(
	\begin{array}{ll}
	     I & 0\\
	     0 & I\otimes X_1 \\
	\end{array} 
	\right).
\end{align*}
Therefore,
\[\left\|\mathscr{T}\left(\sum\limits_{p+q=1} A_{pq}\otimes X_{1}^{p}X_{2}^{q},
~ \sum\limits_{p+q=2}A_{pq}\otimes X_{1}^{p} X_{2}^{q}\right) \right\| 
\leq \left\| f(X_2{X_1}^{*}) \right\|\]
and since $X_2X_{1}^{*}$ is a contraction, 
therefore by the von-Neumann inequality, 
we have 
\[\left\|\mathscr{T}\left(\sum\limits_{p+q=1} A_{pq}\otimes X_{1}^{p}X_{2}^{q},
~ \sum\limits_{p+q=2}A_{pq}\otimes X_{1}^{p} X_{2}^{q}\right) \right\| 
\leq 1.\]
This completes the proof of the theorem.
\end{proof}

\section{Solution to the Extremal Problem}
In this section we shall calculate the supremum 
occurring in Theorem \ref{1var necessary condition} and 
Theorem \ref{var necessary condition}.\\ 
Let $B$ be the bilateral shift on $\ell^2(\mathbb{Z})$  
and $C^*(B)$ be the commutative unital $C^*-$ algebra generated by $B.$  
Hence $C^*(B)$ is isometrically isomorphic to the $C^*$- algebra of continuous functions $C(\sigma(B)),$ where  $\sigma(B)=\mathbb{T}$ is the unit circle in $\C.$ This isometric isomorphism, which we denote by $\tau,$ is defined by the rule $\tau(f) = f(B^*).$ 
Consequently, for any $k\in \mathbb N,$ 
the map 
\[\tau \otimes I_k : C(\mathbb{T}) \otimes M_k\to\mathcal{B}(\ell^2(\mathbb{Z})) \otimes M_k\]
is also a $*-$ isometric monomorphism. 
In particular, for any $P\in\mathcal{P}(\C,M_k),$ we have
\begin{equation}\label{norm of matrix valued polynomial}
	\|P\|_{\D,\infty}^{\rm op}=\|P(B^*)\|.
\end{equation}
In the proof of the Theorem \ref{Ando}, we have seen that in solving the extremal problem 
\[\sup\left\|\mathscr{T}\left(\frac{\partial f}{\partial z_1}(0)X_1 
+ \frac{\partial f}{\partial z_2}(0)X_2,
~\frac{1}{2} \sum_{i,j=1}^{2}
\frac{\partial ^2 f}{\partial z_i \partial z_j}(0) X_iX_j \right)\right\|\]
over 
all commuting contractions $X_1,X_2\in\mathcal{B}(\h),$ we may assume without loss of generality that $X_1=I.$  
Therefore the supremum in Theorem \ref{var necessary condition} 
is equal to the
\begin{equation}\label{reduced to I}
	\sup\left\|\mathscr{T}\left(\frac{\partial f}{\partial z_1}(0)I 
	+ \frac{\partial f}{\partial z_2}(0)X,
	~ \frac{1}{2}\left(\frac{\partial ^2 f}{\partial z^{2}_{1}}(0) I 
	+ 2\frac{\partial ^2 f}{\partial z_1 \partial z_2}(0)X 
	+\frac{\partial ^2 f}{\partial z_{2}^{2}}(0)X^2\right) \right)\right\|,
\end{equation}
where the supremum  is taken over all contractions $X$.

Let $P$ be the polynomial (taking values in $2\times 2$ matrices $M_2$)  
\[P(\lambda)=\mathscr{T}\left(\frac{\partial f}{\partial z_1}(0),
\frac{1}{2}\frac{\partial ^2 f}{\partial z_{1}^{2}}(0)\right) + 
\mathscr{T}\left(\frac{\partial f}{\partial z_2}(0),
\frac{\partial ^2 f}{\partial z_1 \partial z_2}(0)\right)\lambda + 
\mathscr{T}\left(0,\frac{1}{2}\frac{\partial ^2 f}{\partial z_{2}^{2}}(0)\right)\lambda^2\]
We have 
\[\sup\limits_{x_1,x_2\in\D}\left\|\mathscr{T}\left( \frac{\partial f}{\partial z_1}(0)x_1 + 
\frac{\partial f}{\partial z_2}(0)x_2,
\frac{1}{2}\sum_{i,j=1}^{2}\frac{\partial ^2 f}{\partial z_i \partial z_j}(0) x_i x_j\right)\right\|
=\|P\|_{\D,\infty}^{\rm op}\]
and for any contraction $X,$ an application of the von-Neumann inequality gives 
$\|P(X)\|\leq \|P\|_{\D,\infty}^{\rm op}.$
From \eqref{norm of matrix valued polynomial} 
we have,
$\|P\|_{\D,\infty}^{\rm op}=\|P(B^*)\|.$
Therefore,
$\sup\|P(X)\|=\|P(B^*)\|$
and hence supremum in \eqref{reduced to I}, 
Theorem \ref{1var necessary condition} and 
Theorem \ref{var necessary condition} are 
equal to
\[\left\|\mathscr{T}\left(\frac{\partial f}{\partial z_1}(0)I + 
\frac{\partial f}{\partial z_2}(0)B^*,
~ \frac{1}{2}\frac{\partial ^2 f}{\partial z^{2}_{1}}(0) I + 
\frac{\partial ^2 f}{\partial z_1 \partial z_2}(0)B^* +
\frac{1}{2}\frac{\partial ^2 f}{\partial z_{2}^{2}}(0){B^*}^2 \right)\right\|.\]
Thus Theorem \ref{1var necessary condition} and 
Theorem \ref{var necessary condition} are equivalent to the following theorem.
\begin{thm}\label{explicit necessary condition}
	For any  $f\in{\rm H}^\infty_{0}(\D^2,\D),$  
		\[\left\|\mathscr T\left(\frac{\partial f}{\partial z_1}(0)I + 
	\frac{\partial f}{\partial z_2}(0)B^*,
	~ \frac{1}{2}\frac{\partial ^2 f}{\partial z^{2}_{1}}(0) I + 
	\frac{\partial ^2 f}{\partial z_1 \partial z_2}(0)B^* +
	\frac{1}{2}\frac{\partial ^2 f}{\partial z_{2}^{2}}(0){B^*}^2 \right)\right\|
	\leq 1.\]
\end{thm}
The following corollary is essentially a restatement of Theorem \ref{explicit necessary condition}. However, it is worded to make the 
necessary condition for the CF problem inherent in this theorem, evident. 

\begin{cor}\label{finalCF}
	If $p$ is any complex valued polynomial in two variables 
	of degree at most $2$ with $p(0)=0,$ then
	$$	\left\|\mathscr T\left( \frac{\partial p}{\partial z_1}(0) I  
		+ \frac{\partial p}{\partial z_2}(0) B^*, \frac{1}{2}\frac{\partial ^2 p}{\partial z^{2}_{1}}(0) I + 
	\frac{\partial ^2 p}{\partial z_1 \partial z_2}(0)B^* +
	\frac{1}{2}\frac{\partial ^2 p}{\partial z_{2}^{2}}(0){B^*}^2 \right)\right\|
		\leq 1 $$
	is a necessary condition 
	for the existence of a holomorphic function $q:\mathbb D^2 \to \mathbb C,$ with $q^{(k)}(0)=0,$ $|k| \leq 2,$ such that  $\|p+q\|_{\mathbb D^2,\infty} \leq  1.$
\end{cor}

In the Chapter \ref{Koranyi},
we will give another proof of 
Theorem \ref{explicit necessary condition} 
and investigate the question of the converse.

\chapter{The Kor\'{a}nyi-Puk\'{a}nszky Theorem and CF Problems}\label{Koranyi}
\section{The Kor\'{a}nyi-Puk\'{a}nszky Theorem}
We recall the following theorem of Kor\'{a}nyi and Puk\'{a}nszky proved 
in \cite[Corollary, Page 452]{KP}. This gives a necessary and sufficient condition for the range of a  holomorphic function defined on the polydisc $\mathbb D^n$ to  be in the right half plane $H_+.$ 
\begin{thm}[Kor\'{a}nyi-Puk\'{a}nszky Theorem] \label{Koranyi Theorem}
	If the power series 
	$\sum_{\alpha \in \mathbb{N}_{0}^{n}}a_{\alpha}z^{\alpha}$ represents a holomorphic function $f$ on the polydisc $\mathbb D^n,$ then   $\Re(f(z))\geq 0$ for all $z\in \mathbb{D}^n$ 
	if and only if the map 
	$\phi:\mathbb{Z}^n\to \mathbb{C}$ 
	defined by 
	\begin{eqnarray*}
		\phi (\alpha )=\left\{
		\begin{array}{ll}
		      2\Re a_{\alpha} & \mbox{\rm if }\alpha =0\\
		      a_{\alpha} & \mbox{\rm if }\alpha > 0\\
		      a_{-\alpha} & \mbox{\rm if }\alpha < 0\\
		      0 & \mbox{\rm otherwise }\\
		\end{array} 
		\right.
	\end{eqnarray*}
	is positive, that is, the $k\times k$ matrix $\big (\!\!\big ( \phi(\scriptstyle{m_i-m_j}) \big )\!\!\big )$ is non-negative definite for every choice of $m_1, \ldots, m_k\in \mathbb Z^n.$  
\end{thm}
We will call the function 
$\phi,$ the Kor\'{a}nyi-Puk\'{a}nszky function 
corresponding to the coefficients 
$(a_{\alpha})_{\alpha\in \mathbb{N}_0^n}.$

Let us revisit the (CF) problem of realizing a polynomial 
$p\in\mathbb{C}[Z_1,\ldots,Z_n]$ of degree $d$ 
as the first $d$ terms of the power series expansion 
of an analytic function $f\in{\mbox{\rm H}^{\infty}}(\mathbb{D}^n)$ 
with  
$\|f\|_{\D^n,\infty}\leq 1.$ 
\section{The planar Case}
Although, we state the problem below for polynomials of degree $2,$ our methods apply to the general case.
\begin{prob} \label{planar extension}
	Fix $p$ to be the polynomial  $p(z)=az+bz^2.$  
	Find a necessary and sufficient condition 
	for the existence of a holomorphic function $g$ 
	defined on the unit disc $\mathbb D$  
	with $g^{(k)}(0) = 0,\, k=0,1,2,$ 
	such that  $\|p + g \|_{\D,\infty}\leq 1.$ 
\end{prob}

We note that the condition on the range of a holomorphic function 
given in the theorem of Kor\'{a}nyi and Puk\'{a}nszky can be easily converted into a condition where the range is required to be in the unit disc $\mathbb D.$ For this  
consider the Cayley map 
$\chi :\mathbb{D}\to H_{+}$ 
into the right half plane defined by 
\[\chi(z)=\frac{1+z}{1-z},\] 
which is a bi-holomorphism. 
Let $f\in{\mbox{\rm H}^{\infty}}(\mathbb{D})$ be given by the power series 
$f(z)=\sum_{n=1}^{\infty}a_nz^n$. 
Assume that $f$ maps $\mathbb{D}$ to $\mathbb{D}$. 
This happens if and only if $\chi \circ f$ 
maps $\mathbb{D}$ to $H_{+}$. 
Also,
\begin{equation}
 	\chi \circ f(z) = \frac{1+f(z)}{1-f(z)}
 	=2\left(c_0+\sum_{n=1}^{\infty}c_nz^n\right),  \label{c_n's}
\end{equation}
where $c_0 = 1/2$ and the new coefficients $c_n$ are as in the lemma below. 
In this section, we set $c_0=1/2,$ wherever it occurs. 
\begin{lem} \label{recursive coefficients}
	The coefficient $c_n$ in equation \eqref{c_n's} 
	is given by 
	$a_n+\sum\limits_{\substack{j=1}}^{n-1}a_jc_{n-j}$ 
	for $n \geq 1$.
\end{lem}

\begin{proof}
Consider the expression 
\[\chi \circ f(z) =2\left(c_0+\sum_{n=1}^{\infty}c_nz^n\right) 
= 2\left(\frac{1}{2}+f(z)+f(z)^2+f(z)^3+\cdots \right).\]
Rewriting, we get
\[ \frac{1}{1-f(z)} = 1 + \sum\limits_{n = 1}^{\infty}c_nz^n.\]
Hence, we have 
\[\left(1+\sum_{n=1}^{\infty}c_n z^n\right)
\left(1-\sum_{n=1}^{\infty}a_n z^n\right) =1.\]
A comparison of the coefficients completes the verification. 
\end{proof}
\begin{rem}
Applying Theorem \ref{Koranyi Theorem} to 
$\chi \circ f$, we conclude that 
$f$ maps $\mathbb{D}$ to $\mathbb{D}$ 
if and only if the Kor\'{a}nyi-Puk\'{a}nszky function $\phi$ 
corresponding to coefficients  
$(c_n)_{n=0}^{\infty}$ is positive.\\
\end{rem}
\textbf{Matrix of $\phi$ :} 
The matrix $\left(\phi(j-k)\right)_{j,k}$  is given by
\begin{equation}
	\bordermatrix{
		~ & \cdots & -1 & 0 & 1 & \cdots\cr
		\vdots & & \vdots & \vdots & \vdots & \cr
		-1 & \cdots & 1 & \overline{c}_1 & \overline{c}_2 & \cdots\cr
		0 & \cdots & c_1 & 1 & \overline{c}_1 & \cdots \cr
		1 & \cdots & c_2 & c_1 & 1 & \cdots \cr
		\vdots &  & \vdots & \vdots & \vdots &  \cr}. \label{Matrix of phi}
\end{equation}
Therefore, we can rewrite the 
Problem \ref{planar extension} in the equivalent form:
\begin{prob}
Let $p$ be  a polynomial of the form $p(z) = a_1 z + a_2z^2.$ There exists a holomorphic function $q,$ $q^{(k)}(0)=0$ for $k=0,1,2,$ defined on the unit disc $\D,$ such that $\|p+q\|_{\mathbb D, \infty} \leq 1$ if and only if 
\[\left(
	\begin{array}{ccc}
		1 & \overline{c}_1 & \overline{c}_2\\
		c_1 & 1 & \overline{c}_1 \\
		c_2 & c_1 & 1 \\
	\end{array}
\right)\]
	is non-negative definite  
	and for $j > 2,$ there exists $c_j \in \mathbb{C}$ 
	such that the Kor\'{a}nyi-Puk\'{a}nszky function 
	$\phi$ corresponding to $(c_n)_{n\in\mathbb{N}_0}$ is positive.	
\end{prob}
\section{Alternative proof of Theorem \ref{single variable}}
Suppose $f$ is an analytic function on the unit disc $\mathbb{D}$ 
with $\|f\|_{\D,\infty}\leq 1$ and that 
$f(z)=\sum_{n=1}^{\infty}a_nz^n$ is its power series expansion in the unit disc $\D.$ 
Then 
$\chi\circ f(z)$ has the power series $2(c_0+\sum_{n=1}^{\infty}c_n z^n)$ 
in the unit disc $\D$, where $c_0=1/2$ and $c_n$ is of the form prescribed  in the Lemma \ref{recursive coefficients}. 
In this section also, we set $c_0=1/2$, wherever it occurs. 
Let $C_n, A_n$ and $P_n$ denote  
\[(C_n:=)\,\, \left(
{
\begin{array}{ccccccc}
	1 & \overline{c}_1 & \overline{c}_2 & \cdots & \overline{c}_n \\
	c_1 & 1 & \overline{c}_1 & \cdots & \overline{c}_{n-1}\\ 
	c_2 & c_1 & 1 & \cdots & \overline{c}_{n-2}\\
	\vdots & \vdots & \vdots & \ddots & \vdots \\
	c_n & c_{n-1} & c_{n-2} & \cdots & 1 \\
\end{array} }
\right),
\,(A_n:=)\,\, \left(
{
\begin{array}{ccccc}
	a_1 & a_2 & a_3 &\cdots & a_n \\
	0 & a_1 & a_2 & \cdots & a_{n-1}\\ 
	0 & 0 & a_1 & \cdots & a_{n-2}\\
	\vdots & \vdots & \vdots & \ddots & \vdots \\
	0 & 0 & 0 & \cdots & a_1 \\
\end{array} }
\right)
\]
and 
\[(P_n:=)\,\, \left(
{
\begin{array}{ccccc}
	1 & -a_1 & -a_2 &\cdots & -a_n \\
	0 & 1 & -a_1 & \cdots & -a_{n-1}\\ 
	\vdots & \vdots & \ddots & \ddots & \vdots \\
	0 & 0 & 0 & \cdots & -a_1 \\
	0 & 0 & 0 & \cdots & 1 \\
\end{array} }
\right)
\]
respectively for each $n\in \mathbb N.$ 
\begin{lem}\label{n=2 case}
	If $a_1,a_2\in\mathbb{C},$ then $|a_1|^2+|a_2|\leq 1 $
	if and only if the matrix $C_2$ is non-negative definite.
\end{lem}
\begin{proof}
Since $|a_1|^2+|a_2|\leq 1$, 
it follows that $\|A_2\|\leq 1$. 
It is equivalent to the positivity of the following matrix 
\[\left(
\begin{array}{cc}
	I- A_2 A_{2}^{*} & 0\\
	0 & 1\\
\end{array}
\right)
=
\left(
{
\begin{array}{ccc}
	1-(|a_1|^2+|a_2|^2) & -a_2\overline{a}_1 & 0 \\
	-a_1\overline{a}_2 & 1-|a_1|^2 & 0\\ 
	0 & 0 & 1\\
\end{array} }
\right)=P_2 C_2^{\rm t} P_{2}^{*}.
\]
Since $P_2$ is an invertible matrix, 
the positivity of $P_2 C_2^{\rm t} P_{2}^{*}$ is equivalent to 
$C_2^{\rm t}$ being non-negative definite. 
Thus, we conclude that $C_2$ is non-negative definite.
\end{proof}
\begin{lem}\label{contractivity vs positivity}
	For all $n\in\mathbb{N},$ $P_n C_{n}^{\rm t} P_{n}^{*}=(I-A_n A_{n}^{*})\oplus 1.$ 
	\end{lem}
\begin{proof}
We shall prove the result by induction on $n$. 
The case $n=1$ follows from the Lemma \ref{n=2 case}. 
Assume the result upto $n-1$ for $n>1$. 
For each $n\in \N,$ let  
\[\tilde{P}_n:=\left(-a_n,-a_{n-1},\ldots ,-a_1\right)^{\rm t} 
\mbox{ \rm and } 
\tilde{C}_n:=\left(c_n,c_{n-1},\ldots ,c_1\right)^{\rm t}.\] 
The verification of the identity   
\[P_nC_n^{\rm t}P_{n}^{*} =\left(
\begin{array}{cc}
	P_{n-1} & \tilde{P}_n\\
	0 & 1\\
\end{array}
\right) 
\left(
\begin{array}{cc}
	C_{n-1}^{\rm t} & \tilde{C}_n\\
	\tilde{C}_{n}^{*} & 1\\
\end{array}
\right) 
\left(
\begin{array}{cc}
	P_{n-1}^{*} & 0\\
	\tilde{P}_{n}^{*} & 1\\
\end{array}
\right)\]
is easy. Hence 
\[P_nC_{n}^{\rm t}P_{n}^{*} = 
\left(
\begin{array}{cc}
	P_{n-1}C_{n-1}^{\rm t}P_{n-1}^{*} + \tilde{P}_n\tilde{C}_{n}^{*}P_{n-1}^{*} + 
	\tilde{P}_{n}^{*}\left(P_{n-1}\tilde{C}_n+
	\tilde{P}_n\right) & P_{n-1}\tilde{C}_n + \tilde{P}_n\\
	\left(P_{n-1}\tilde{C}_n +\tilde{P}_n\right)^* & 1\\
\end{array} 
\right).\]
From the Lemma \ref{recursive coefficients}, 
we have 
$P_{n-1}\tilde{C}_n + \tilde{P}_n = 0$ 
and therefore we conclude that
\[P_nC_nP_{n}^{*}=
\left(
\begin{array}{cc}
	P_{n-1}C_{n-1}^{\rm t}P_{n-1}^{*}+ \tilde{P}_n\tilde{C}_{n}^{*}P_{n-1}^{*} & 0\\
	0 & 1\\
\end{array}
\right).
\]
Now 
\[\tilde{P}_n \tilde{C}_{n}^{*}P_{n-1}^{*}
=\left(
\begin{array}{c}
	-a_n \\
	\vdots \\
	-a_1\\
\end{array}
\right)
\left(
\begin{array}{cccc}
	\overline{c}_n-\sum\limits_{i=1}^{n-1}a_ic_{n-i} & 
	\overline{c}_{n-1}-\sum\limits_{i=1}^{n-2}a_ic_{n-i} & 
	\cdots & \overline{c}_1\\
\end{array}
\right).
\]
From the Lemma \ref{recursive coefficients}, 
we get
\[\tilde{P}_n\tilde{C}^{*}_{n}P_{n-1}^{*}
=\left(
\begin{array}{c}
	-a_n \\
	\vdots \\
	-a_1\\
\end{array}
\right)
\left(
\begin{array}{ccc}
	\overline{a}_n  & \cdots & \overline{a}_1\\
\end{array}
\right)=
\left(-a_{n-i}\overline{a}_{n-j}\right)_{i,j=0}^{n-1}.\]
Also,
\[I-A_k A^{*}_{k} = 
\left(
\begin{array}{cccc}
	1-\sum\limits_{j=1}^{k}|a_j|^2 & 
	-\sum\limits_{j=2}^{k}a_j\overline{a}_{j-1} & 
	\cdots & -a_k\overline{a}_1 \\
	-\sum\limits_{j=2}^{k}\overline{a}_ja_{j-1} & 
	1-\sum\limits_{j=1}^{k-1}|a_j|^2 & \cdots & 
	-a_{k-1}\overline{a}_1\\
	\vdots & \vdots & \ddots & \vdots \\
	-a_1\overline{a}_{k} & 
	-a_1\overline{a}_{k-1} & 
	\cdots & 1-|a_1|^2 \\ 
\end{array}
\right)\]
and therefore
$I-A_nA_{n}^{*} = \left((I-A_{n-1}A_{n-1}^*)\oplus 1\right) + 
\left(-a_{n-j}\overline{a}_{n-l}\right)_{1\leq j,l\leq {k-1}}.$
Thus
$I-A_nA_{n}^{*} = P_{n-1}C_{n-1}^{\rm t}P_{n-1}^{*} + 
\tilde{P}_n\tilde{C}_nP_{n-1}^{*},$
which completes the proof.
\end{proof}
An immediate corollary is the following proposition.
\begin{prop}
	The matrix $C_n$ is non-negative definite 
	if and only if $A_n$ satisfies $\|A_n\|\leq 1$.
\end{prop}
In the theorem below we provide an alternative proof of the Theorem \ref{single variable}. The technique involved here is from the Section 3 of the paper of Parrott \cite{SP}. 
\begin{thm}
	Suppose $a_1$ and $a_2$ are two complex numbers.  
	Then there exists $f\in{\mbox{\rm H}^{\infty}}(\mathbb{D})$, 
	with $\|f\|_{\D,\infty}\leq 1$, 
	such that
	$f(0)=0,~f'(0)=a_1,~f{'\!'}(0)=2a_2$ 
	if and only if
	$|a_1|^2+|a_2|\leq 1.$
\end{thm}
\begin{proof}
Assume $f$ and $\chi \circ f$ are 
as in \eqref{c_n's}. 
Then by Kor\'{a}nyi-Puk\'{a}nszky theorem \ref{Koranyi Theorem}, 
every principal submatrix of finite size  of the matrix in 
\eqref{Matrix of phi} is  non-negative definite. 
In particular, the matrix $C_2$
is non-negative definite. 
From the Lemma \ref{n=2 case} 
we conclude that 
$|a_1|^2+|a_2|\leq 1.$
Conversely, 
assume $a_1,a_2\in\mathbb{C}$ are such that 
$|a_1|^2+|a_2|\leq 1$. 
Then,
\[A_2=\left(
{
\begin{array}{cc}
	a_1 & a_2\\
	0 & a_1\\
\end{array}}
\right)
\]
satisfies $\|A_2\|\leq 1$. 
Using Parrott's theorem, 
there exists $a_3\in\mathbb{C}$ 
such that 
\[A_3=\left(
{
\begin{array}{ccc}
	a_1 & a_2 & a_3\\
	0 & a_1 & a_2\\
	0 & 0 & a_1\\
\end{array}}
\right)
\]
has operator norm less than or equal to $1$.
Using the Lemma \ref{recursive coefficients}, we see that  
the matrix $C_3$ is non-negative definite. 
Repeatedly using the Parrott's theorem and the 
Lemma \ref{recursive coefficients}, one may ensure 
the existence of non-negative definite matrices $C_n,$ for all 
$n>3.$ 

Hence the Kor\'{a}nyi-Puk\'{a}nszky function 
corresponding to $(c_n)_{\N_0}$ is positive. 
Thus the function $g(z)=\sum_{n}c_n z^n$ maps 
$\mathbb{D}$ to $\mathbb{H}_+$ 
by Kor\'{a}nyi-Puk\'{a}nszky theorem \ref{Koranyi Theorem}. 
Hence from the Lemma \ref{recursive coefficients}, 
the function $f=\chi^{-1} \circ g$ satisfies all the required conditions.
\end{proof}

\section{The case of two Variables}
\begin{prob}\label{Two Variables}
	Fix $p\in\mathbb{C}[Z_1,Z_2]$ to be the polynomial defined by 
	$$p(z_1,z_2)=a_{10}z_1+a_{01}z_2+
	a_{20}z_{1}^{2}+a_{11}z_1z_2+a_{02}z_{2}^2.$$
	Find necessary and sufficient conditions 
	for the existence of a holomorphic function function 
	$q$ defined on the bi-disc $\mathbb D^2$ 
	with $q^{(k)}(0)=0$ for $|k|\leq 2,$
	such that
	$\|p+q\|_{\D^2,\infty}\leq 1.$
\end{prob}
Let $f$ be an analytic function on $\mathbb{D}^2.$ 
Suppose $f$ is represented  by the power series 
\[f(z)=\sum\limits_{\substack{{m,n=0}}}^{\infty}a_{mn}z_{1}^{m}z_{2}^{n}\]
and $a_{00}=0.$ 
Also assume that $f$ maps $\mathbb{D}^2$ into $\mathbb{D}.$ 
This happens if and only if 
$\chi \circ f$ maps $\mathbb{D}^2$ to $H_{+},$ where 
\[\chi \circ f(z)= (1+f(z))(1-f(z))^{-1}= 
2\left(c_{00}+\sum_{m,n=1}^{\infty}c_{mn}z_{1}^{m}z_{2}^{n}\right),\]
$c_{00}=1/2$ and the coefficients $c_{mn}$ are from the Lemma \ref{recurrence 2 variables}. 
In this section, we set $c_{00}=1/2$, wherever it occurs.  
If $\phi$ denotes the Kor\'{a}nyi-Puk\'{a}nszky function 
corresponding to the coefficients $(c_{mn}),$ 
then $\phi$ is positive.

\textbf{The matrix of $\phi$:} 
For a fixed $k\in\mathbb{Z}$, define 
$P_k:=\left\{\left(x,y\right)|x+y=k \right\}.$ 
The sequence $(P_k)$ is a sequence of disjoint subsets of $\mathbb{Z}^2.$ Besides 
\[\bigsqcup_{k\in\mathbb{Z}}P_k=\mathbb{Z}^2.\]
An order on $\mathbb{Z}^2,$ which we call the {\tt D-slice} ordering, is  defined below. Clearly, it is different from the usual co-lexicographic order. The matrix computations that follow are transparent because of the D-slice ordering that we use in describing the matrix of the Kor\'{a}nyi-Puk\'{a}nszky function $\phi.$ 
\begin{defn}[D-slice ordering] \label{D-slice}
Suppose $(x_1,y_1)\in P_l$ and $(x_2,y_2)\in P_m$ 
are two elements in $\mathbb{Z}^2$. 
Then 
\begin{enumerate}
	\item If $l=m,$ then $(x_1,y_1) < (x_2,y_2)$ is determined by the lexicographic ordering on $P_l\subseteq \mathbb Z^2$ and    
	\item  if $l < m$ (resp., if $l > m$), then $(x_1,y_1) < (x_2,y_2)$ (resp., $(x_1,y_1) > (x_2,y_2)$).  
\end{enumerate}
\end{defn}
The following theorem describes the Kor\'{a}nyi-Puk\'{a}nszky function $\phi$ with respect to the D-slice ordering on $\mathbb Z^2.$ 
\begin{thm}\label{Matrix phi}
Let $(c_{mn})_{m,n\in\N_0}$ be an infinite array of complex numbers. The matrix of the  Kor\'{a}nyi-Puk\'{a}nszky function $\phi$ in the D-slice ordering   corresponding to this array is of the form   
$$\bordermatrix{
	~ & \cdots & P_{-1} & P_0 & P_1 & \cdots\cr
	\vdots &  & \vdots & \vdots & \vdots & \cr
	P_{-1} & \cdots & I & C_1^{*} & C_2^{*} & \cdots\cr
	P_{0} & \cdots & C_1 & I & C_1^{*} & \cdots \cr
	P_1 & \cdots & C_2 & C_1 & I & \cdots \cr
	\vdots &  & \vdots & \vdots & \vdots &  \cr},$$
where $C_n:=c_{n0}I+c_{n-1,1}B^{*}+\cdots +c_{0n}{B^*}^n,$ 
$n\in \mathbb N.$ 
\end{thm} 
\begin{proof}
With respect to the D-slice ordering on $\mathbb{Z}^2,$ 
the matrix corresponding to the function $\phi$ 
is a doubly infinite block matrix,  
where $(k,n)$ element in $(l,m)$ block,  
which is $\phi\left((k,-k+l)-(n,-n+m)\right).$ 
is computed as follows, separately, in three different cases:

First, let $k-n<0.$	
	
	The quantity $\phi\left((k,-k+l)-(n,-n+m)\right)$ is non-zero 
	only if $k-n \geq l-m.$ 
	Hence if $l\geq m,$ 
	then $\phi\left((k,-k+l)-(n,-n+m)\right)=0.$ 
	Now, assume $l<m.$  
	In this case, the possible values 
	for $k-n$ are $l-m,l-m+1,\ldots,-1,$
	otherwise $\phi\left((k,-k+l)-(n,-n+m)\right) = 0.$ 
	For $p\in\{0,1,\ldots,-l+m-1\}$ and $k-n=l-m+p,$   
	we have
	\[\phi\left((k,-k+l)-(n,-n+m)\right)=\overline{c}_{m-l-p,p}.\]
	
	Second, let $k-n=0.$
		\begin{eqnarray*}
		\phi (0,l-m)=\left\{
		\begin{array}{ll}
 		     c_{0,l-m} & \mbox{ if } l\geq m\\
		      \overline{c}_{0,m-l} & \mbox{ if } l< m\\
	\end{array} 
	\right.
	\end{eqnarray*}
	
	Finally, let $k-n>0.$
	 
	The quantity $\phi\left((k,-k+l)-(n,-n+m)\right)$ 
	is non-zero only if $k-n \leq l-m$. 
	Hence if $l\leq m,$ 
	then $\phi\left((k,-k+l)-(n,-n+m)\right)=0$. 
	Now, assume $l>m$. 
	In this case the possible values 
	for $k-n$ are $l-m,l-m-1,\ldots,1$ otherwise 
	$\phi\left((k,-k+l)-(n,-n+m)\right).$ 
	For $p\in\{0,1,\ldots,l-m-1\}$ and $k-n=l-m-p,$ 
	we have
	\[\phi\left((k,-k+l)-(n,-n+m)\right)=c_{l-m-p,p}.\]

Therefore, the $(l,m)$ block in the matrix of $\phi$ 
is given exactly by the following rule:  
\begin{enumerate}
	\item $C_{m-l}^{*}$ if $l<m$,
	\item $C_{l-m}$ if $l>m$,
	\item $I$ if $m=l$.
\end{enumerate}
Hence the matrix of the  Kor\'{a}nyi-Puk\'{a}nszky function $\phi$ in the D-slice ordering corresponding to the array $(c_{mn})$ is of the form 
$$\bordermatrix{
	~ & \cdots & P_{-1} & P_0 & P_1 & \cdots\cr
	\vdots &  & \vdots & \vdots & \vdots & \cr
	P_{-1} & \cdots & I & C_1^{*} & C_2^{*} & \cdots\cr
	P_{0} & \cdots & C_1 & I & C_1^{*} & \cdots \cr
	P_1 & \cdots & C_2 & C_1 & I & \cdots \cr
	\vdots &  & \vdots & \vdots & \vdots &  \cr}$$
\end{proof}
Assume that the power series 
	$\sum_{j,k=0}^{\infty}a_{jk}z_{1}^{j}z_{2}^{k},$ 
	with $a_{00}=0,$ represents a holomorphic function $f$ defined on the bi-disc and that $\|f\|_{\D^2,\infty}\leq 1.$ 
Let $2\big(c_{00}+\sum_{j,k=1}^{\infty}c_{jk}z_{1}^{j}z_{2}^{k}\big)$
be the power series representation for $\chi \circ f$ on the bi-disc $\mathbb D^2$ for some choice of complex number $c_{mn}$ which are determined from the coefficients $a_{mn}$ of the function $f.$
\begin{lem} \label{recurrence 2 variables}
For all $n\in \mathbb N,$ setting $A_n:=a_{n0}I+a_{n-1,1}B^*+\cdots +a_{0n}{B^*}^n,$ $C_n=c_{n0}I+c_{n-1,1}B^*+ \cdots +c_{0n}{B^*}^n,$ we have 
		$$C_n=A_n+\sum\limits_{\substack{j=1}}^{n-1}A_jC_{n-j}.$$
\end{lem}
\begin{proof}
Let 
$C(z_1,z_2):=\sum_{i,j=0}^{\infty}c_{ij}z_{1}^{i}z_{2}^{j}.$
We have 
\[1+f(z_1,z_2)+f(z_1,z_2)^2+\cdots=\frac{(\chi\circ f )(z_1,z_2)}{2}+c_{00}= C(z_1,z_2).\]
Thus 
$C(z_1,z_2)(1-f(z_1,z_2))=1,$ which is the same as  
\begin{eqnarray*} 
	\lefteqn{(1+c_{10}z_1+c_{01}z_2+c_{20}z_{1}^{2}+
	c_{11}z_1z_2+c_{02}z_{2}^{2}+\cdots )\times}\\&\phantom{\times \times}& (1-a_{10}z_1-a_{01}z_2-
	a_{20}z_{1}^{2}-a_{11}z_1z_2-a_{02}z_{2}^{2}+\cdots)\\&=&1.
\end{eqnarray*}
Now comparing the coefficient of 
$z_{1}^{n-k}z_{2}^k$ we have
\[c_{n-k,k}=\sum\limits_{\substack{p=0}}^{k}
\sum\limits_{\substack{j=k}}^{n}a_{n-j,p}c_{j-k,k-p},\]
where $a_{00}=0$.\\
The coefficient of ${B^*}^k$ in $A_n+\sum\limits_{\substack{i=1}}^{n-1}A_iC_{n-i}$ 
is 
\begin{eqnarray*}
\lefteqn{a_{n-k,k}c_{00}+a_{n-k,k-1}c_{01}+a_{n-k-1,k}c_{10}+
	a_{n-k,k-2}c_{02}}\\
 &\phantom{++++}& + a_{n-k-1,k-1}c_{11}+a_{n-k-2,k}c_{20}+\cdots\\
\end{eqnarray*} 
\vspace{-56pt}
\begin{eqnarray*}
\lefteqn{= (a_{n-k,k}c_{00}+a_{n-k,k-1}c_{01}+\cdots +
 	a_{n-k,0}c_{0k})+}\\
 	&\phantom{++++}&(a_{n-k-1,k}c_{10}+a_{n-k-1,k-1}c_{11}+
 	\cdots +a_{n-k-1,0}c_{1,k})+\cdots 
 	\end{eqnarray*}
\vspace{-32pt}
\begin{eqnarray*}	
\lefteqn{ \cdots +(a_{0k}c_{n-k,0}+a_{0,k-1}c_{n-k,1}+\cdots +a_{00}c_{n-k,k})}\\
 	&\phantom{====}& = \sum\limits_{\substack{p=0}}^{k}\sum\limits_{\substack{j=k}}^{n}a_{n-j,p}c_{j-k,k-p}
\end{eqnarray*}
completing the proof of the claim.
\end{proof}

In view of the Theorem \ref{Matrix phi} and the Lemma \ref{recurrence 2 variables}, the CF Problem \ref{Two Variables} takes the following form:
\begin{thm}{\label{equivalent form}}For any polynomial $p$  of the form $$p(z) = a_{10} z_1 + a_{01}z_2 + a_{20}z_1^2 + a_{11} z_1z_2 + a_{02} z_2^2,$$ there exists a holomorphic function $q,$ defined on the bi-disc $\D^2,$ with $q^{(k)}(0)=0$ for $|k|=0,1,2,$ such that $$\|p+q\|_{\mathbb D^2, \infty} \leq 1$$ if and only if 
	\[\left(
	\begin{array}{ccc}
		I & C_1^{*} & C_2^{*}\\
		C_1 & I & C_1^{*} \\
		C_2 & C_1 & I \\
	\end{array}
	\right)
	\]
	is non-negative definite and for each $k\geq 3,$ there exists 
	$C_k=c_{k0}I+c_{k-1,1}B^*+ \cdots +c_{0k}{B^*}^k$ 
	such that the Kor\'{a}nyi-Puk\'{a}nszky function $\phi$ 
	corresponding to $(c_{mn})_{m,n\in \N_0}$ is positive.  
\end{thm}

\begin{lem}\label{Contractivity and Positivity}
	If $A_n$ and $C_n$ are as defined above, then   
	\[\left(
	{
	\begin{array}{ccccccc}
		I & C_{1}^* & C_{2}^* & \cdots & C_{n}^* \\
		C_1 & I & C_{1}^* & \cdots & C_{n-1}^*\\ 
		C_2 & C_1 & I & \cdots & C_{n-2}^*\\
		\vdots & \vdots & \vdots & \ddots & \vdots \\
		C_n & C_{n-1} & C_{n-2} & \cdots & I \\
	\end{array} }
	\right)
	\geq 0\]
	 if and only if 
	 \[\left\| \left(
	{
	\begin{array}{ccccc}
		A_1 & A_2 & A_3 &\cdots & A_n \\
		0 & A_1 & A_2 & \cdots & A_{n-1}\\ 
		0 & 0 & A_1 & \cdots & A_{n-2}\\
		\vdots & \vdots & \vdots & \ddots & \vdots \\
		0 & 0 & 0 & \cdots & A_1 \\
	\end{array} }
	\right)\right\|
	\leq 1.\]
\end{lem}
\begin{proof}
For each $n\in\mathbb N,$ 
$C_n$ commutes with $C_m$ and $A_m$ 
for all $m\in\mathbb{N}$ and 
hence we can adapt the proof of the Lemma \ref{contractivity vs positivity} 
to complete the proof in this case.  
\end{proof}
Since the adjoint of the bilateral shift $B^*$ on $\ell^2(\Z)$ is unitarily equivalent to 
the multiplication operator $M_z$ on $L^2(\T)$, it follows that 
$A_n$ and $C_n$ are unitarily equivalent to the multiplication operators 
$M_a$ and $M_c$ respectively, where $a(z)=a_{n0}+a_{n-1,1}z + \cdots +a_{0n}z^n,$ and $c(z)=c_{n0}+ c_{n-1,1}z + \cdots +c_{0n}z^n.$  
Now the Theorem \ref{equivalent form} takes the equivalent form given below, where for the polynomial $p$ of the form  $p(z) = a_{10} z_1 + a_{01}z_2 + a_{20}z_1^2 + a_{11} z_1z_2 + a_{02} z_2^2,$ we have set
\begin{equation}\label{p1p2}
 p_1(z):= a_{10}+a_{01}z \mbox{\rm ~and~} p_2(z) = a_{20}+a_{11}z+a_{02}z^2.
\end{equation}
\begin{thm}\label{Final form}
	For any polynomial $p$ of the form 
	$$p(z) = a_{10} z_1 + a_{01}z_2 + a_{20}z_1^2 + a_{11} z_1z_2 + a_{02} z_2^2,$$ 
	there exists a holomorphic function $q,$ 
	defined on the bi-disc $\D^2,$ with $q^{(k)}(0)=0$ for $|k|=0,1,2,$ such that 
	$$\|p+q\|_{\mathbb D^2, \infty} \leq 1$$ if and only if
	 	$|p_2|\leq 1-|p_1|^2$ 
	and there exists a holomorphic function  
	$f:\mathbb{D}\to \mathcal B(L^2(\mathbb{T}))$
	with 
	\[\|f\|_{\D,\infty}^{\rm op}\leq 1 
	\mbox{ and }
	\frac{f^{(k)}(0)}{k!}=M_{p_k}\mbox{ for all }k\geq 0,\]
	where $p_0=0$ and for $k\geq 3,$ 
	$p_k\in\mathbb{C}[Z]$ 
	is a polynomial of degree less than or equal to $k.$ 
	Here $M_{p_k}$ is the multiplication operator on $L^2(\mathbb{T})$ 
	induced by the polynomial $p_k.$ 
	\end{thm}
	Thus the Problem \ref{Two Variables} has been reduced to a one variable problem except it now involves holomorphic functions taking values in $\mathcal B(L^2(\mathbb{T})).$ To discuss this variant of the CF problem, the following definition will be useful. 
\begin{defn}[Completely Polynomially Extendible]
	Suppose $k\in\N$ and $\{p_j\}_{j=1}^{k}$ is a sequence of polynomials, 
	with $\deg(p_j)\leq j$ for all $j=1,2,\ldots,k$. 
	Then $\mathscr{T}(M_{p_1},\ldots,M_{p_k})$ will be called 
	$n$-{\tt{polynomially~extendible}} 
	if $\|\mathscr{T}(M_{p_1},\ldots,M_{p_k})\|\leq 1$ 
	and there exists a sequence of polynomials 
	$\{p_{l}\}_{l=k+1}^{n}$, 
	with $\deg(p_{l})\leq l$,  
	such that $\|\mathscr{T}(M_{p_1},\ldots,M_{p_n})\|
	\leq 1$. 
	Also, $\mathscr{T}(M_{p_1},\ldots,M_{p_k})$ 
	will be called {\tt{completely~polynomially~extendible}} 
	if the operator $\mathscr{T}(M_{p_1},\ldots,M_{p_k})$ is 
	$n$-polynomially extendible for all $n\in\N$.
\end{defn}

For $p_1,p_2\in\mathbb{C}[Z],$ polynomials of degree at most $1$ and $2$ respectively,  
let $P$ denote the polynomial
$P(z)=M_{p_1}z + M_{p_2}z^2.$ 
We shall call $P$ to be a polynomial in the CF class if given these polynomials
$p_1,\,p_2,$ there is a holomorphic function $f:\mathbb{D}\to \mathcal B(L^2(\mathbb{T}))$ 
satisfying properties stated in the Theorem \ref{Final form}. Such a function $f$ is called a CF-extension of the polynomial $P$. It follows that a solution to the Problem \ref{Two Variables} exists if and only if the polynomial $P$ is in the CF class. We have therefore proved the following theorem.
\begin{thm}
A solution to the Problem \ref{Two Variables} exists if and only if  	
the corresponding one variable operator valued polynomial $P$ is in the CF class. 
Or, equivalently, 
	$\mathscr{T}(M_{p_1},M_{p_2})$ 
	is completely polynomially extendible.
\end{thm}

%
It is clear, from the Theorem \ref{Final form}, that $|p_1|^2+|p_2|\leq 1$ is a necessary condition for the existence of a  solution to the Problem \ref{Two Variables}. This condition via Parrott's theorem is also equivalent to the condition $\|\mathscr T(M_{p_1}, M_{p_2})\| \leq 1.$ 

	We now give some instances, where this necessary condition 
	is also sufficient for the existence of a solution to the Problem \ref{Two Variables}. This would amount to find condition for   $\mathscr{T}(M_{p_1},M_{p_2})$ 
	to be completely polynomially extendible. 

\begin{thm}
	Let $p_1(z)=\gamma + \delta z$ and $p_2(z)=(\alpha + \beta z)(\gamma + \delta z)$ for some choice of complex numbers $\alpha,~\beta,~\gamma$ and $\delta.$ Assume that 
	$|p_1|^2+|p_2|\leq 1.$ 
	If either $\alpha \beta \gamma \delta =0$ or 
	$\arg(\alpha)-\arg(\beta)=\arg(\gamma)-\arg(\delta),$ 
	then $\mathscr{T}(M_{p_1},M_{p_2})$ is 
	completely polynomially extendible.
\end{thm}
\begin{proof} All through this proof, for brevity of notation, we will let $\|f\|$ stand for the norm $\sup\{\|f(z)\|_{\rm op}:z\in\mathbb D\},$ for any holomorphic function $f:\mathbb D \to \mathcal B(L^2(\mathbb T)).$  

\textbf{Case 1:} Suppose $\beta=0$. 
Then $P(z)=M_{p_1}(z + \alpha z^2 ).$
Let $p(z)=z + \alpha z^2$. 
Using Nehari's theorem, we extend $p$ to the function 
$\tilde{p}(z)=z + \alpha z^2 + \alpha_3 z^3 +\cdots$ 
such that
$\|\tilde{p}\|_{\D,\infty}=\left\|\mathscr{T}(1,\alpha)\right\|.$
Define
$f(z)=M_{p_1}\tilde{p}(z)=M_{p_1}z + M_{p_2}z^2 + M_{p_3}z^3 +\cdots, $
where 
$p_k=\alpha_k p_1$. 
Also,
\[\|f\|=\sup_{z\in\D}\|M_{p_1}\tilde{p}(z)\|
=\|M_{p_1}\|\sup_{z\in\D}|\tilde{p}(z)|=\|M_{p_1}\| \|\mathscr{T}(1,\alpha)\|.\]
Thus 
$\|f\|=\|M_{p_1}\otimes \mathscr{T}(1,\alpha)\|=\|\mathscr{T}(M_{p_1},M_{p_2})\|\leq 1.$
Hence $f$ is a required CF-extension of $P$.

\textbf{Case 2:} Suppose $\alpha=0$. 
Then, $P(z)=M_{p_1}(z + \beta M_z z^2)$. 
Let $Q(z)=z + \beta M_z z^2$ and 
$r(z_1,z_2)=z_1(1 + \beta z_2)$. 
Define $s(z_2)=1+\beta z_2$. 
Suppose
$\tilde{s}(z_2)=s(z_2) + \beta_2 z^{2}_{2} + \beta_3z_{2}^{3} + \cdots$
be such that 
$\|\tilde{s}\|_{\D,\infty}=\|\mathscr{T}(1,\beta)\|$. 
If $\tilde{r}:=z_1\tilde{s}(z_2),$ then
$\|\tilde{r}\|=\|\tilde{s}\|=\|\mathscr{T}(1,\beta)\|.$
If $\tilde{Q}(z)=z + M_{\beta z}z^2 + M_{\beta_2 z^2}z^2 + 
\cdots$ and 
$f(z)=M_{p_1}\tilde{Q}(z),$
then
$\|f\|=\|M_{p_1}\tilde{Q}\|\leq \|M_{p_1}\|\|\tilde{Q}\|.$ 
Since $\tilde{s}(M_z)=\tilde{Q}(z)$, from the von-Neumann inequality it follows that 
$\|\tilde{Q}\|\leq \|\tilde{s}\|.$ 
Therefore,
$\|f\|\leq \|M_{p_1}\|\|\mathscr{T}(1,\beta)\|
=\|\mathscr{T}(M_{p_1},\beta M_{p_1})\|.$
Hence
\[\|f\|
\leq 
\left\|\left(
\begin{array}{cc}
	M_z & 0\\
	0 & I\\
\end{array}
\right)
\left(
\begin{array}{cc}
	M_{p_1} & \beta M_{p_1}\\
	0 & M_{p_1}\\
\end{array}
\right)
\left(
\begin{array}{cc}
	M_{z}^{*} & 0\\
	0 & I\\
\end{array}
\right)\right\|
=
\|\mathscr{T}(M_{p_1},M_{p_2})\|
\leq 1. \]
Therefore $f$ is a CF-extension of $P.$

\textbf{Case 3:} Suppose $\alpha\neq 0$ and  
$\beta\neq 0$. 
Then, $P(z)=M_{p_1}\left(z + M_{\alpha + \beta z}z^2\right)$. 
Let $Q(z):=z + M_{\alpha + \beta z}z^2$. 
Define
$r(z_1,z_2):=z_1 + \alpha z^{2}_{1} + \beta z_1z_2=
z_1\left(1 + \alpha z_1 + \beta z_2\right).$
Let $\lambda:=|\alpha|/|\beta|$ and 
$a:=\lambda/(1+\lambda)$. 
Define
$s(z_1,z_2):=1 + \alpha z_1 + \beta z_2=
\left(a + \alpha z_1\right) + \left(1-a + \beta z_2\right).$
If $h_1(z_1):= a + \alpha z_1$ and 
$h_2(z_2):= 1-a + \beta z_2,$ then 
there exist
$\tilde{h}_1=a + \alpha z_1 + \alpha_2 z_{1}^{2} +\cdots$
and $\tilde{h}_2=1-a + \beta z_2 + \beta_2 z_{2}^{2} +\cdots$
with
$\left\|\tilde{h}_1\right\|=\big \|\mathscr{T}(a,\alpha)\big \| 
\mbox{ and } 
\left\|\tilde{h}_2\right\|=\left \|\mathscr{T}(1-a,\beta)\right\|.$
If
\[\tilde{r}(z_1,z_2):=z_1\left(\tilde{h}_1(z_1) + \tilde{h}_2(z_2)\right)=
z_1 + \alpha z_{1}^{2} + \beta z_1z_2 + 
\alpha_2 z_{1}^{3} + \beta_2z_1z_{2}^{2} +\cdots,\]
then $\|\tilde{r}\|\leq \|\tilde{h}_1\| + \|\tilde{h}_2\|$. 
Let
$\tilde{Q}(z)=Iz + M_{\alpha + \beta z}z^2 + M_{\alpha_2 + \beta_2 z^2}z^3 +\cdots$
and 
$f(z)=M_{p_1}\tilde{Q}(z)$ $=\sum_{j}M_{p_j}z^j,$
where 
$p_{k+1}(z)=\big(\alpha_k + \beta_kz^k\big)p_1$ 
for all $k>1$. 
Thus $\|f\|\leq \|M_{p_1}\|\|\tilde{Q}\|.$ 
Since $\tilde{Q}(z)=\tilde{r}(z,M_z),$ from the von-Neumann inequality, it follows that 
\[\|f\|\leq \|M_{p_1}\|\|\tilde{r}\|\leq \|M_{p_1}\|
\left(\|\tilde{h}_1\| + \|\tilde{h}_2\|\right).\]
As $\mathscr{T}(a,|\alpha|)=\lambda \mathscr{T}(1-a,|\beta|)$, 
therefore
$\left(\|\tilde{h}_1\| + \|\tilde{h}_2\|\right)=
\|\mathscr{T}\left(1,|\alpha|+|\beta|\right)\|$
and hence
\[\|f\|\leq \|M_{p_1}\|\|\mathscr{T}\left(1,|\alpha|+|\beta|\right)\|=
\left\|\mathscr{T}\left(\|p_1\|,(|\alpha|+|\beta|)\|p_1\|\right)\right\|.\]
\textbf{subcase 1:} Suppose $\gamma\neq 0$, $\delta \neq 0$ 
and 
$\mbox{ arg }(\alpha) - \mbox{ arg }(\beta) = \mbox{ arg }(\gamma) - \mbox{ arg }(\delta)$. 
Then
\[(|\alpha|+|\beta|)\|p_1\|= \|(\alpha + \beta)p_1\|=\|p_2\|.\]
Our hypothesis clearly implies that
$\|p_2\| + \|p_1\|^2 \leq 1.$
Hence $\|f\|\leq 1$.\\
\textbf{subcase 2:} Suppose $\gamma=0$ or $\delta=0$. 
Then
\[(|\alpha|+|\beta|)\|p_1\|= \|(\alpha + \beta)p_1\|=\|p_2\|.\]
As in subcase 1, 
here also $\|f\|\leq 1$ can be inferred easily. 
\end{proof}
\begin{rem}
	In Problem \eqref{Two Variables}, 
If either $p_1\equiv 0$ or $p_2\equiv 0,$ and $\|\mathscr T(M_{p_1}, M_{p_2})\| \leq 1,$ then $\|P\|\leq 1$ 
and hence $f$ in the Theorem \eqref{Final form} 
can be taken to be $P$ itself. 	
\end{rem}

Having verified that the necessary condition 
$\|\mathscr{T}(M_{p_1},M_{p_2})\|\leq 1$ is also  sufficient 
for $P$ to be in the CF class in several cases, we expected it to be sufficient in general.  
But unfortunately this is not the case. 
We give an example of a polynomial $P$ 
for which $\|\mathscr{T}(M_{p_1},M_{p_2})\|\leq 1$ 
but $P$ is not in the CF class. 

\textbf{An Example:} Let $p_1(z)=1/\sqrt{2}$ and 
$p_2(z)=z^2/2$. 
We show that $\mathscr{T}(M_{p_1},M_{p_2})$ 
is not even $3-$polynomially extendible.

It can easily be seen that 
$\|\mathscr{T}(M_{p_1},M_{p_2})\|\leq 1$. 
Now suppose there exists a polynomial $p_3$ of degree at most $3$ such that  
$\|\mathscr{T}(M_{p_1},M_{p_2},M_{p_3})\|\leq 1.$ 
Then Parrott's theorem 
guarantees the existence of a contraction $V\in\mathcal{B}\big(L^2(\T)\big)$ such that
\[M_{p_3}=\left(I-M_{|p_1|^2}-M_{p_2}\left(I-M_{|p_1|^2}\right)^{-1}M_{p_2}^{*}\right)V-M_{p_2}\left(I-M_{|p_1|^2}\right)^{-\frac{1}{2}}M_{p_1}^{*}\left(I-M_{|p_1|^2}\right)^{-\frac{1}{2}}M_{p_2}.\]
As we have $(1-|p_1|^2)^{2}-|p_2|^2\equiv 0,$ 
therefore operator in the first bracket is zero and hence
\[p_3=\frac{-p_{2}^{2}\overline{p_1}}{1-|p_1|^2}=\sqrt{2}z^4.\]
Thus $p_3$ is a polynomial of degree more than $3$ which is a contradiction. 
Hence $\mathscr{T}(M_{p_1},M_{p_2})$ is not even 
$3-$ polynomially extendible.

We close this section with an open question: What are the properties we must impose on the polynomials  $p_1$ and $p_2$  in addition to the requirement $\|\mathscr{T}(M_{p_1},M_{p_2})\|\leq 1$ to ensure 
that $P$ is in the CF class?

\section{Kor\'{a}nyi-Puk\'{a}nszky Theorem as an Application of Spectral Theorem}
In this section we show that 
Kor\'{a}nyi-Puk\'{a}nszky Theorem \ref{Koranyi Theorem} 
is an application 
of the well known spectral theorem. 
For simplicity we consider $n=2$.
\begin{thm}[Kor\'{a}nyi-Puk\'{a}nszky Theorem]
	If the power series 
	$\sum_{\alpha \in \mathbb{N}_{0}^{n}}c_{\alpha}z^{\alpha}$ represents a holomorphic function $g$ on the polydisc $\mathbb D^n,$ then   $\Re(g(z))\geq 0$ for all $z\in \mathbb{D}^n$ 
	if and only if the map 
	$\phi:\mathbb{Z}^n\to \mathbb{C}$ 
	defined by 
	\begin{eqnarray*}
		\phi (\alpha )=\left\{
		\begin{array}{ll}
		      2\Re c_{\alpha} & \mbox{\rm if }\alpha =0\\
		      c_{\alpha} & \mbox{\rm if }\alpha > 0\\
		      c_{-\alpha} & \mbox{\rm if }\alpha < 0\\
		      0 & \mbox{\rm otherwise }\\
		\end{array} 
		\right.
	\end{eqnarray*}
	is positive, that is, the $k\times k$ matrix $\big (\!\!\big ( \phi(\scriptstyle{m_i-m_j}) \big )\!\!\big )$ is non-negative definite for every choice of $m_1, \ldots, m_k\in \mathbb Z^n.$ 
\end{thm} 
\begin{proof}
Suppose the power series $2\sum_{\alpha \in \mathbb{N}_{0}^{n}}c_{\alpha}z^{\alpha}$ 
represents a holomorphic function $g$ defined on the bi-disc $\D^2$ with the property that 
$\Re(g(z))\geq 0$ for all $z\in\D^2$(This extra factor $2$ has been put in to the power series just to make previous computation matched). 
Without loss of generality, we assume that $g(0)=1$. 
The function $g$ maps $\D^2$ 
to the right half plane $H_+$ if and only if 
$f:=\chi^{-1}\circ g$ maps $\D^2$ to the unit disc $\D$. 
Suppose $\sum_{j,k=0}^{\infty}a_{jk}z_{1}^{j}z_{2}^{k}$ represents the function $f$. 
Then, $a_{00}=0$ and the array of coefficients $\big(\!\!\big(a_{jk}\big)\!\!\big)$ 
and $\big(\!\!\big(c_{jk}\big)\!\!\big)$ are related by the formula obtained in 
the Lemma \ref{Contractivity and Positivity}. The operators $I\otimes B^*$ and $B^*\otimes B^*$ 
are commuting unitaries 
and they have $\T^2$ as their joint spectrum.  
Now, applying spectral theorem and maximum modulus principle, 
we get the following: 
\begin{equation}\label{Spectral}
\|f(I\otimes B^*,B^*\otimes B^*)\|=\|f\|_{\D^2,\infty}.
\end{equation}

Also, we note that 
\[f(I\otimes B^*,B^*\otimes B^*)=A_1\otimes B^* + A_2\otimes {B^*}^2+\cdots,\]
where $A_n:=a_{n0}I+a_{n-1,1}B^*+\cdots +a_{0n}{B^*}^n$ and $C_n:=c_{n0}I+c_{n-1,1}B^*+\cdots +c_{0n}{B^*}^n$ as in the Lemma \ref{Contractivity and Positivity}. 
Since $\|f\|_{\D^2,\infty}\leq 1,$ it follows from \eqref{Spectral} that 
$\|\mathscr{T}(A_1,\ldots,A_n)\|\leq 1$ for all $n\in\N$. Now, from the Lemma \ref{Contractivity and Positivity}, we conclude that 
\begin{equation}\label{Principal block}
\left(
	{
	\begin{array}{ccccccc}
		I & C_{1}^* & C_{2}^* & \cdots & C_{n}^* \\
		C_1 & I & C_{1}^* & \cdots & C_{n-1}^*\\ 
		C_2 & C_1 & I & \cdots & C_{n-2}^*\\
		\vdots & \vdots & \vdots & \ddots & \vdots \\
		C_n & C_{n-1} & C_{n-2} & \cdots & I \\
	\end{array} }
	\right)
\end{equation} 
is non-negative for all $n\in\N$. 
Hence from the Theorem \ref{Matrix phi}, we get that the Kor\'{a}nyi-Puk\'{a}nszky 
function $\phi$ corresponding to the array $\big(\!\!\big(c_{jk}\big)\!\!\big)$ is positive.

Conversely, suppose the Kor\'{a}nyi-Puk\'{a}nszky 
function $\phi$ corresponding to the array 
$\big(\!\!\big(c_{jk}\big)\!\!\big)$ is positive, where $c_{00}$ is assumed to be $1/2.$ 
Then, from the Theorem \ref{Matrix phi}, 
we get that operator in \eqref{Principal block} 
is non-negative for all $n\in \N$. 
Thus, from the Lemma \ref{Contractivity and Positivity} and the equation \eqref{Spectral}, 
we conclude that $\|\chi\circ g\|_{\D^2,\infty}\leq 1,$ where $g(z_1,z_2)=2\sum_{m,n=0}^{\infty}c_{mn}z_1^mz_2^n.$ This is so if and only if $g$ maps $\D^2$ to the right half plane $H_+.$
Hence the theorem is proved.
\end{proof}

\chapter{A generalization of Nehari's Theorem}
For a closed subspace $M$ and 
a point $x$ in a Hilbert space $\h$, 
the distance of $M$ from $x$ is attained at $P(x),$ 
where $P$ is the orthogonal projection of $\h$ onto $M$. 
Nehari considered a similar problem 
but in the space $L^{\infty}(\T)$. 
He evaluated the distance of a function 
$f$ in $L^{\infty}(\T)$ from 
the closed subspace $H^{\infty}(\T)$. 
Before stating Nehari's theorem 
we shall give some definitions.

\section{The Hankel Operator}
Let $H^2(\T)$ denote the Hardy space, namely the closed subspace of $L^2(\T):$ 
$$H^2(\mathbb T):= \big \{f\in L^2( \mathbb T)\vert\, \hat{f}(-n) = 0, n\in \mathbb N \big\},$$
where $\hat{f}(-n)$ is the Fourier coefficient of $f$ with respect to the standard orthonormal basis $z^n,$ $n\in\mathbb Z$ and $z\in\mathbb T,$ of $L^2(\mathbb T).$ Let $P_{-}$ denote the orthogonal projection 
of $L^2(\T)$ onto $L^2(\T)\ominus H^2.$
\begin{defn}[Multiplication Operator]
	For $\phi\in L^{\infty}(\T),$ 
	we define the multiplication operator 
	$M_{\phi}:L^2(\T)\to L^2(\T)$ 
	by the rule $M_{\phi}(f)=\phi  f,$ 
	where $(\phi  f)(z) = \phi(z) f(z),\, z\in \mathbb T.$  
\end{defn}
For any $\phi\in L^{\infty}(\T)$ and 
$f\in L^2(\T),$ it is easy to see that 
$\phi f\in L^2(\T)$ and  
that $M_{\phi}$ is bounded.  
Indeed $\|M_{\phi}\|=\|\phi\|_{\infty}$ (cf. \cite[Theorem 13.14]{NY}). 

\begin{defn}[Hankel Operator]
	For $\phi\in L^{\infty}(\T),$ define the  
	Hankel operator with symbol $\phi$ 
	to be the operator $P_{-}\circ M_{\phi}\vert_{H^2}$ and  
	denote it by $H_{\phi}.$
\end{defn}
We recall the well known theorem of Nehari.
\begin{thm}[Nehari]
	If $\phi\in L^{\infty}(\T)$ and 
	$H_{\phi}$ is the corresponding Hankel operator, 
	then
	\[\inf\left\{\|\phi-g\|_{\T,\infty}:g\in H^{\infty}(\T)\right\}
	=\|H_{\phi}\|_{op}.\]
\end{thm}

\section{Nehari's theorem for \texorpdfstring{$L^2(\mathbb T^2)$}{TEXT}}
In this section, 
we will give a possible generalization of Nehari's theorem for $L^2(\mathbb T^2).$ 
This generalization is most conveniently stated in terms of the D-slice ordering on $\Z^2,$ which we now recall.
For a fixed $k\in\mathbb{Z}$, define 
$P_k:=\left\{\left(x,y\right)|x+y=k \right\}.$ 
The sequence $\big(P_k\big)$ is a sequence of disjoint subsets of $\mathbb{Z}^2$ 
and $\bigsqcup_{k\in\mathbb{Z}}P_k=\mathbb{Z}^2.$
The D-slice ordering on $\mathbb{Z}^2$ is the ordering:

Suppose $(x_1,y_1)\in P_l$ and $(x_2,y_2)\in P_m$ 
are two elements in $\mathbb{Z}^2$. 
Then 
\begin{enumerate}
	\item If $l=m,$ then $(x_1,y_1) < (x_2,y_2)$ is determined by the lexicographic ordering on $P_l\subseteq \mathbb Z^2$ and    
	\item  if $l < m$ (respectively, if $l > m$), then $(x_1,y_1) < (x_2,y_2)$ (respectively, $(x_1,y_1) > (x_2,y_2)$).  
\end{enumerate}
Let $A_1=\bigsqcup_{k\in\mathbb{N}_0}P_k$ and 
$A_2=\bigsqcup_{k\in\mathbb{N}}P_{-k}$. 
Define
\[H_1:=\left\{f:=\!\!\!\!\!\sum\limits_{(m,n)\in A_1}\!\!\!\!\!\!\!\!a_{m,n}z_{1}^{m}z_{2}^{n}|f\in L^{\infty}(\T^2)\right\},\,
H_2:=\left\{f:=\!\!\!\!\!\sum\limits_{(m,n)\in A_2}\!\!\!\!\!\!\!\!a_{m,n}z_{1}^{m}z_{2}^{n}|f\in L^{\infty}(\T^2)\right\}.\]
$H_1$ and $H_2$ are two closed and disjoint subspaces 
of $L^{\infty}(\T^2)$ satisfying 
$L^{\infty}(\T^2)=H_1 \oplus H_2$. 
Now the answer to the following question on $L^2(\mathbb T^2)$ would be a natural generalization of the Nehari's theorem.
\begin{Qn}
	 	For $\phi\in L^{\infty}(\T^2),$ what is $\mbox{\rm dist}_{\infty}(\phi,H_1),$ the distance of $\phi$ 
	from the subspace $H_1?$  
\end{Qn}

\section{The Hankel Matrix corresponding to \texorpdfstring{$\phi$}{TEXT}}
Any $\phi\in L^{2}(\T^2)$ 
can be written as
\[\phi(z_1,z_2)=\sum\limits_{m,n\in\Z}\!\!\!\!\!\!\!\!a_{m,n}z_{1}^{m}z_{2}^{n}
=\sum\limits_{m,n\in A_1}\!\!\!\!\!\!\!\!a_{m,n}z_{1}^{m}z_{2}^{n} + 
\sum\limits_{m,n\in A_2}\!\!\!\!\!\!\!\!a_{m,n}z_{1}^{m}z_{2}^{n}.\]
Suppose $z_2=\lambda z_1$. 
Then
\[\phi(z_1,\lambda z_1)=\sum\limits_{k\geq 0}\left(\sum\limits_{m+n=k}
\!\!\!\!\!\!\!\!a_{m,n}\lambda^n\right)z_{1}^{k} + 
\sum\limits_{k<0}\left(\sum\limits_{m+n=k}\!\!\!\!\!\!\!\!a_{m,n}\lambda^n\right)z_{1}^{k}.\]
Setting 
$f^{\phi}_{k}(\lambda):=\sum_{m+n=k}a_{m,n}\lambda^n,$ 
we have  
$$\phi(z_1,\lambda z_1)=\sum_{k\in\Z}f^{\phi}_{k}(\lambda)z_{1}^{k}.$$
In this way, $L^{2}(\T^2)$ 
is first identified  
with $L^{2}(\T)\otimes L^{2}(\T)$ and then  
a second time with 
$L^{2}(\T)\otimes \ell^2(\Z)$, the identifications in both cases  
being isometric. 
For any $\phi\in L^{\infty}(\T^2),$ 
define the multiplication operator 
$M_{\phi}:L^{2}(\T)\otimes \ell^2(\Z)\to L^{2}(\T)\otimes \ell^2(\Z)$ 
as follows 
\[M_{\phi}\left(\sum\limits_{j\in\Z}g_j\otimes e_j\right):=
\sum\limits_{k\in\Z}\left(\sum\limits_{q\in\Z}g_qf_{q+k}\right)e_k.\] 

\begin{lem}\label{Norm of Multiplication Operator}
	For any $\phi\in L^{\infty}(\T^2),$ we have $\|M_{\phi}\|=\|\phi\|_{\T^2,\infty}.$ 
\end{lem}
\begin{proof}
Let $\phi\in L^{\infty}(\T^2)$ 
be an arbitrary element. 
From what we have said above, it follows that 
$\phi(z,\lambda z)=\sum_{k\in\Z}f^{\phi}_{k}(\lambda)z^{k}$ for some $f_k^\phi$ in $L^2(\mathbb T).$ 
The set of vectors $\left\{z^i \otimes e_j:(i,j)\in\Z^2\right\}$ 
is an orthonormal basis 
in $L^{2}(\T)\otimes \ell^2(\Z)$. 
The matrix of the operator $M_{\phi}$ with respect to this basis 
and the D-slice ordering on its index set is of the form 
\[\left(
\begin{array}{ccccc}
	 & \vdots & \vdots & \vdots & \\
	\cdots & M_{f^{\phi}_{-1}} & M_{f^{\phi}_{0}} & M_{f^{\phi}_{1}} & \cdots\\
	\cdots & M_{f^{\phi}_{-2}} & M_{f^{\phi}_{-1}} & M_{f^{\phi}_{0}} & \cdots\\
	\cdots & M_{f^{\phi}_{-3}} & M_{f^{\phi}_{-2}} & M_{f^{\phi}_{-1}} & \cdots\\
	 & \vdots & \vdots & \vdots & \\ 
\end{array}
\right).
\] 
We know that 
$\|\phi\|_{\T^2,\infty}=\sup_{\lambda\in\T}\sup_{z\in\T}\left|\sum_{k\in\Z}f^{\phi}_{k}(\lambda)z^{k}\right|$. 
Thus
\begin{align*}
	\|\phi\|_{\T^2,\infty}&=\sup_{\lambda\in\T}
	\left\|
	\left(
	\begin{array}{ccccc}
		 & \vdots & \vdots & \vdots & \\
		\cdots & f^{\phi}_{-1}(\lambda) & f^{\phi}_{0}(\lambda) & f^{\phi}_{1}(\lambda) & \cdots\\
		\cdots & f^{\phi}_{-2}(\lambda) & f^{\phi}_{-1}(\lambda) & f^{\phi}_{0}(\lambda) & \cdots\\
		\cdots & f^{\phi}_{-3}(\lambda) & f^{\phi}_{-2}(\lambda) & f^{\phi}_{-1}(\lambda) &\cdots\\
		 & \vdots & \vdots & \vdots & \\ 
	\end{array}
	\right)
	\right\| \\
	&=
	\left\|
	\left(
	\begin{array}{ccccc}
		 & \vdots & \vdots & \vdots & \\
		\cdots & M_{f^{\phi}_{-1}} & M_{f^{\phi}_{0}} & M_{f^{\phi}_{1}} & \cdots\\
		\cdots & M_{f^{\phi}_{-2}} & M_{f^{\phi}_{-1}} & M_{f^{\phi}_{0}} & \cdots\\
		\cdots & M_{f^{\phi}_{-3}} & M_{f^{\phi}_{-2}} & M_{f^{\phi}_{-1}} & \cdots\\
		 & \vdots & \vdots & \vdots & \\ 
	\end{array}
	\right)
	\right\|.
\end{align*}
Hence $\|\phi\|_{\T^2,\infty}=\|M_{\phi}\|$ completing the proof.
\end{proof}

The Hilbert space $\ell^{2}(\N_0)$ and the normed linear subspace 
$$\big \{(\ldots,0,x_0,x_1,\ldots):\sum_{i\geq 0}|x_i|^2 < 
\infty \mbox{ with }x_0 \mbox{ at the } 0^{th} 
\mbox{ position}\big \}$$
of $\ell^2(\mathbb Z)$ 
are naturally isometrically isomorphic.
Let $H:=L^2(\T)\otimes \ell^2(\N_0)$. 
The space $H$ is a closed subspace of 
$L^{2}(\T)\otimes \ell^{2}(\Z)$. 
We define Hankel operator 
$H_{\phi}$, with symbol $\phi$, 
to be the operator 
$P_{H^{\perp}}\circ {M_{\phi}}_{|H}$. 
Writing down the matrix for $H_{\phi}$ 
with respect to the bases 
$\{z^i \otimes e_j:i\in\Z,j=0,1,2,\ldots\}$ 
and 
$\{z^i \otimes e_{-j}:i\in\Z,j=1,2,\ldots\}$ 
in the spaces $H$ 
and $H^{\perp}$ respectively, 
we get 
\[H_{\phi}=
\left(
\begin{array}{cccc}
	M_{f^{\phi}_{-1}} & M_{f^{\phi}_{-2}} & M_{f^{\phi}_{-3}} & \cdots\\
	M_{f^{\phi}_{-2}} & M_{f^{\phi}_{-3}} & M_{f^{\phi}_{-4}} & \cdots\\
	M_{f^{\phi}_{-3}} & M_{f^{\phi}_{-4}} & M_{f^{\phi}_{-5}} & \cdots\\
	\vdots & \vdots & \vdots & \\ 
	\end{array}
\right).
\]
We call this the Hankel matrix with symbol $\phi.$

Let $\h$ be a Hilbert space. 
For any $(T_n)_{n\in\N}\subset \mathcal{B}(\h),$ 
define a operator $H(T_1,T_2,\ldots)$ as follows: 
\[H(T_1,T_2,\ldots)=
\left(
\begin{array}{cccc}
	T_1 & T_2 & T_3 & \cdots\\
	T_2 & T_3 & T_4 & \cdots\\
	T_3 & T_4 & T_5 & \cdots\\
	\vdots & \vdots & \vdots & \\ 
\end{array}
\right).
\]
\begin{lem}\label{Upper bound for Hankel matrix}
For $\phi$ in $L^\infty(\mathbb T^2),$ we have $\|H_{\phi}\|\leq \mbox{\rm dist}_{\infty}(\phi,H_1)$.
\end{lem}
\begin{proof} 
From the definition of $H_{\phi}$ and the Lemma \ref{Norm of Multiplication Operator}, it can easily be seen that
\[\big\|H_{\phi}\big\|=\big\|P_{H^{\perp}}\circ {M_{\phi}}_{|H}\big\|
\leq \big\|M_{\phi}\big\| = \big\|\phi\big\|_{\T^2,\infty}.\] 
Thus $\|H_{\phi}\|\leq \|\phi\|_{\T^2,\infty}$. 
From the matrix representation of $H_{\phi}$ 
it is clear that for any 
$g$ in $H_1,$ $H_{\phi-g}=H_{\phi}$. 
Hence 
$\|H_{\phi}\|=\|H_{\phi -g}\|\leq \|\phi-g\|_{\T^2,\infty}$. 
Thus  the proof of the lemma is complete. 
\end{proof}
For 
$n\in\N$, $a_0,a_1,\ldots,a_{n-1}\in\C$ 
and 
$(b_m)_{m\in\N}\subset \C,$ 
define the following operator  
\[T_n\left((b_m),a_0,a_1,\ldots,a_{n-1}\right):=
\left(
\begin{array}{cccc}
	a_0 & a_1 & \cdots & a_{n-1}\\
	b_1 & a_0 & \cdots & a_{n-2}\\
	\vdots & \vdots & & \vdots\\
	b_{n-1} & b_{n-2} & \cdots & a_0\\
	\vdots & \vdots &  & \vdots\\
\end{array}
\right).
\]

\begin{lem}\label{Existence of essentially bounded function}
	Suppose  
	$f_0,f_1,\ldots,f_{n-1}\in L^{\infty}(\T)$ 
	and 
	$(g_m)\subset L^{\infty}(\T)$ 
	are such that 
	\[\sup_{\lambda\in\T}
	\|T_n\left((g_m(\lambda)),f_0(\lambda),\ldots,f_{n-1}(\lambda)\right)\|\leq 1.\]
	Then, there exists $f_n\in L^{\infty}(\T)$ 
	such that 
	\[\sup_{\lambda\in\T}
	\|T_{n+1}\left((g_m(\lambda)),f_0(\lambda),\ldots,f_{n}(\lambda)\right)\|\leq 1.\]
\end{lem}
\begin{proof}
Let \[Q(\lambda)=
\left(\!\!
\begin{array}{ccc}
	f_0(\lambda)\!\!\!&\!\!\cdots\!\!&\!\!\!f_{n-1}(\lambda)\\
\end{array}
\right),R(\lambda)=
\left(
\begin{array}{ccccc}
	f_{n-1}(\lambda)\!\!\! &\!\! \cdots \!\!&\!\!\! f_0(\lambda) & \!\!\!g_1(\lambda) \!\!&\!\! \cdots\\
\end{array}
\right)^{\rm t}
\]
and 
\[S(\lambda)=
\left(
\begin{array}{cccc}
	g_1(\lambda) & f_0(\lambda) & \cdots & f_{n-3}(\lambda)\\
	g_2(\lambda) & g_1(\lambda) & \cdots & f_{n-4}(\lambda)\\
	\vdots & \vdots & & \vdots\\
	g_{n-1}(\lambda) & g_{n-2}(\lambda) & \cdots & g_1(\lambda)\\
	\vdots & \vdots & & \vdots\\
\end{array}
\right).
\]
All possible choices of $f_n(\lambda)$ for which  $T_{n+1}\left((g_m(\lambda)),f_0(\lambda),\ldots,f_{n}(\lambda)\right)$ is a contraction are given, via Parrott's theorem (cf. \cite[Chapter 12, page 152]{NY}), by the formula
\begin{equation}\label{All solutions in Parrot's theorem}
	f_n(\lambda)=(I-ZZ^*)^{1/2}V(I-Y^*Y)^{1/2}-ZS(\lambda)^*Y, 
\end{equation}
where $V$ is an arbitrary contraction and $Y,$ $Z$ are obtained from the 
formulae $R(\lambda)=(I-S(\lambda)S(\lambda)^*)^{1/2}Y$, 
$Q(\lambda)=Z(I-S(\lambda)^*S(\lambda))^{1/2}.$ 

Every entry of $I-S(\lambda)^*S(\lambda)$ 
is in $L^{\infty}$ as function of $\lambda.$ 
Thus all entries in $(I-S(\lambda)^*S(\lambda))^{1/2}$ 
are measurable functions which are essentially bounded. Consequently, so are entries of $Z.$ A similar assertion can be made for $Y.$
Therefore choosing $V=0$ 
in equation \eqref{All solutions in Parrot's theorem}, 
we get $f_n$ with the required property. 
In fact, one can choose $V$ to be any contraction 
whose entries are $L^{\infty}$ functions. 
\end{proof}

\begin{thm}[Nehari's theorem for $L^2(\mathbb T^2)$] 
	If $\phi\in L^{\infty}(\T^2),$ 
	then $\|H_{\phi}\|=\mbox{\rm dist}_{\infty}(\phi,H_1)$.
\end{thm}
\begin{proof}
From the Lemma \ref{Upper bound for Hankel matrix},  
we know that 
$\|H_{\phi}\|\leq {\rm dist}_{\infty}(\phi,H_1)$. 
Without loss of generality 
we assume that $\|H_{\phi}\|=1$.
Using the Lemma \ref{Existence of essentially bounded function}, 
we find 
$f^{\phi}_{0}\in L^{\infty}(\T)$ 
such that the norm of 
$H\big (M_{f^{\phi}_{0}},M_{f^{\phi}_{-1}},\ldots \big )$ 
is at most 1. 
Repeated use of 
the Lemma \ref{Existence of essentially bounded function} 
proves the theorem.
\end{proof}

\section{CF problem in view of Nehari's theorem for \texorpdfstring{$L^2(\T)$}{TEXT}}
Fix $p\in\mathbb{C}[Z_1,Z_2]$ to be the polynomial defined by 
	$$p(z_1,z_2)=a_{10}z_1+a_{01}z_2+
	a_{20}z_{1}^{2}+a_{11}z_1z_2+a_{02}z_{2}^2.$$
Denote $\phi(z_1,z_2):=\overline{z}_1^3p(z_1,z_2)=a_{10}\overline{z}_1^2+a_{01}\overline{z}_1^3z_2+a_{20}\overline{z}_1+a_{11}\overline{z}_1^2z_2+a_{02}\overline{z}_1^3z^2_2.$ 
Suppose $p_1(\lambda)=a_{10}+a_{01}\lambda$ and $p_2(\lambda)=a_{20}+a_{11}\lambda + a_{02}\lambda^2.$
Then, $\|H_{\phi}\|={\rm dist}_{\infty}(\phi,H_1),$ where 
\[H_\phi=
\left(
	\begin{array}{cccc}
		M_{p_2} & M_{p_1}& 0 & \cdots\\
		M_{p_1} & 0 & 0 & \cdots\\
		0 & 0 & 0 & \cdots\\
		\vdots&\vdots&\vdots&\ddots\\
	\end{array}
\right)
\]	
Thus, if there exists a holomorphic function $q:\D^2\to \C$ with $q^{(k)}(0)=0$ for $|k|\leq 2$ such that $\|p+q\|_{\D^2,\infty}\leq 1,$ then $\|H_\phi\|\leq \|p+q\|_{\D^2,\infty}$. Hence $\|H_\phi\|\leq 1$ is a necessary condition for such a $q$ to exist.

\chapter{Operator Space Structures on \texorpdfstring{$\ell ^{1}(n)$}{TEXT}}
\section{Operator space}
\begin{defn}
	An abstract operator space 
	is a normed linear space $V$ 
	together with a 
	norm $\|\cdot\|_k$ defined on the linear space 
$$M_k(V):=\big \{\big (\!\! \big (v_{ij} \big )\!\! \big )\vert v_{ij} \in V, 1\leq i,j \leq k \big \},\,\, k\in \mathbb N,$$	
with the understanding that $\|\cdot\|_1$ is the norm of $V$ and 
	the family of norms $\|\cdot\|_k$ satisfies the compatibility conditions:
	\begin{itemize}
		\item[1.] $\|T\oplus S\|_{p+q} = \max\big\{\|T\|_p,\|S\|_q\big\}$ and
		\item[2.] $\|ASB\|_q \leq \|A\|_{op} \|S\|_p \|B\|_{op}$
	\end{itemize} 
	for all $S\in M_q(V),T\in M_p(V),$ 	
	$A\in M_{q\times p}(\C)$ and $B\in M_{p\times q}(\C).$
\end{defn}

Let $(V,\|\cdot\|_k )$ and 
$(W,\|\cdot\|_k)$ be two operator spaces. 
A linear bijection 
$T:V\to W$ is said to be a complete isometry 
if $T\otimes I_k:(M_k(V),\|\cdot\|_k)\to (M_k(W),\|\cdot\|_k)$
is an isometry for every $k\in \N$. 
Operator spaces $(V,\|\cdot\|_k )$ and 
$(W,\|\cdot\|_k)$ are said to be completely isometric
if there is a linear complete isometry $T:V\to W$. 
A  well known theorem of Ruan 
says that any operator space $(V,\|\cdot\|_k)$ 
can be embedded, completely isometrically, in to $C^*$-algebra $\mathcal{B}(\h)$ 
for some Hilbert space $\h$.
There are two natural operator space structures
on any normed linear space $V,$ which may coincide.  
These are the MIN and the MAX operator space structures defined below.  
\begin{defn}[MIN]
	The MIN operator space structure denoted by MIN($V$) 
	on a normed linear space $V$ 
	is obtained by the isometric embedding of $V$ 
	in to the $C^*$-algebra $C((V^*)_1)$, 
	the space of continuous functions 
	on the unit ball $(V^*)_1$ 
	of the dual space $V^*$. 
	Thus for $\big (\!\! \big (v_{ij} \big )\!\! \big )$ in $M_k(V)$, 
	we set
	\[\big\|\big (\!\! \big (v_{ij} \big )\!\! \big )\big\|_{MIN} = \sup\left\{\big\|\big (\!\! \big (f(v_{ij}) \big )\!\! \big )\big\| : f\in (V^*)_1\right\},\]
	where the norm of a scalar matrix $\big (\!\! \big (f(v_{ij}) \big )\!\! \big )$ 
	is the operator norm in $M_k$. 
\end{defn}

\begin{defn}[Max]
	Let $V$ be a normed linear space and  
	$\big (\!\! \big (v_{ij} \big )\!\! \big )\in M_k(V)$. 
	Define
	\[\big\|\big (\!\! \big (v_{ij} \big )\!\! \big )\big\|_{MAX} = \sup\left\{\big\|\big (\!\! \big (Tv_{ij} \big )\!\! \big )\big\|: T:V\to \mathcal{B}(\h)\right\},\]
	where the supremum is taken over 
	all isometries $T$ and all Hilbert spaces $\h$. 
	This operator space structure is denoted by MAX($V$).
\end{defn}

These two operator space structures 
are extremal in the sense 
that for any normed linear space $V$, 
MIN($V$) and MAX($V$) are 
the smallest and largest 
operator space structures on $V$  respectively. 
For any normed linear space $V,$ 
Paulsen \cite{VP} associates a very interesting constant, namely,   $\alpha(V):$ 
\[\alpha(V):=\sup \left\{\|I_V\otimes I_k\|_{(M_k(V), \|\cdot \|_{\rm MIN}) \to (M_k(V), \|\cdot \|_{\rm MAX})}: k\in \mathbb N \right \}. \]
%
%
The constant $\alpha(V)$ is equal to $1$ 
if and only if 
$V$ has only one operator space structure on it. 
There are only a few examples of  normed linear spaces 
for which $\alpha(V)$ is known to be $1.$  These include   
$\alpha(\ell^{\infty}(2))=\alpha(\ell^1(2))=1.$ 
In fact, it is  known (cf. \cite[Page 79]{GP}) 
that $\alpha(V)>1$ if $\dim(V)\geq 3$.

The map $\phi:\ell^\infty(n)\to \mathcal{B}(\C^n)$ 
defined by $\phi(z_1,\ldots,z_n)={\rm diag}(z_1,\ldots,z_n),$  is an isometric embedding of the normed linear space $\ell^{\infty}(n)$ 
in to the finite dimensional C*-algebra $\mathcal{B}(\C^n).$ Clearly, this is the MIN structure of the normed linear space $\ell^\infty(n).$
%
We shall, however prove that 
there is no such finite dimensional isometric embedding for 
the dual space $\ell^1(n).$ 
Never the less, we shall construct, explicitly, a number of possibly different isometric infinite dimensional embeddings of $\ell^1(n)$.   



\section{\texorpdfstring{$\ell^{1}(n)$}{TEXT} has no isometric embedding into any \texorpdfstring{$M_k$}{TEXT}}
In this section, we will show that 
there does not exist 
an isometric embedding of $\ell^{1}(n),$ $n>1,$ into any finite dimensional matrix algebra $M_k,$ $k\in\mathbb N.$ 
Without loss of generality, 
we prove this for the case of  $n=2.$ 

\begin{lem}\label{Reason for non finite dimensinality}
	For $m\in\N$ and 
	$\theta_1,\ldots,\theta_m\in [0,2\pi),$ 
	there exists $a_1,a_2\in\C$ 
	such that 
	\[\max\limits_{j=1,\ldots, m}\big|a_1 + e^{i\theta_j}a_2\big|< \big|a_1\big| + \big|a_2\big|.\]
\end{lem}
\begin{proof} For any two non-zero complex numbers $a_1,a_2,$ we have  
\[\max\limits_{j=1,\ldots, m}\big|a_1 + e^{i\theta_j}a_2\big|
=\max\limits_{j=1,\ldots, m}\big||a_1| + e^{i\theta_j+\phi_2-\phi_1}|a_2|\big|,\] 
where $\phi_1$ and $\phi_2$ are the arguments of 
$a_1$ and $a_2$ respectively. 
Setting $\alpha_j=\theta_j+\phi_2-\phi_1,$ we have 
\begin{align*}
\max\limits_{j=1,\ldots, m}\big|a_1 + e^{i\theta_j}a_2\big|^2
&=\max\limits_{j=1,\ldots, m}\big||a_1| + e^{i\alpha_j}|a_2|\big|^2\\
&=\max\limits_{j=1,\ldots, m}\big||a_1|^2 |a_2|^2+ 2|a_1a_2|cos\alpha_j\big|.
\end{align*}
Therefore
\[\max\limits_{j=1,\ldots, m}\big|a_1 + e^{i\theta_j}a_2\big|
=\big|a_1\big|+\big|a_2\big|
\] 
if and only if 
$\cos \alpha_j=1$ for some $j,$ 
that is, if and only if 
$\alpha_j=0$ for some $j$. 
Choose $a_1$ and $a_2$ such that 
$\phi_1-\phi_2 \neq \theta_j$ 
for all $j=1,\ldots, m$. 
The existence of such a pair $a_1$ and $a_2$ 
proves  the lemma. 
\end{proof}

\begin{thm}\label{No finite dimensional embedding}
	There is no isometric embedding of $\ell^{1}(2)$ into $M_n$ for any $n\in \N.$.
\end{thm}
\begin{proof}
Suppose there is a $n-dimensional$ isometric embedding $\phi$ of $\ell^{1}(2)$. 
Then this embedding $\phi$ is induced by a pair of operators $T_1, T_2\in M_n$ of norm $1,$ defined by the rule    
$\phi(a_1,a_2)=a_1T_1 + a_2T_2.$  
Let $U_1$ and $U_2$ in $M_{2n}$ be the pair of unitaries: 
$$
U_i:= \begin{pmatrix}T_i & D_{T_i^*} \\ D_{T_i} & -T_i^*, \end{pmatrix} i=1,2,
$$
where $D_{T_i}$ is the positive square root of the (positive) operator $I-T_i^* T_i.$   
Now, we have 
\[P_{\C^n}(a_1U_1 + a_2U_2)_{|\C^n}=a_1T_1 + a_2T_2.\]
(This dilating pair of unitaries is not necessarily commuting nor is 
 it a power dilation!)  
Thus 
$\psi:\ell^{1}(2)\to M_{2n}(\C)$ 
defined by 
$\psi(a_1,a_2)=a_1U_1 + a_2U_2$
is also an isometry. 
Since norms are preserved under unitary operations, without loss of generality,    
we assume $U_1=I$ and $U_2$ 
to be a diagonal unitary, say, $D.$ 
Let 
$D={\rm diag}\big(e^{i\theta_1},\ldots,e^{i\theta_{2n}}\big)$. 
Now applying the Lemma \ref{Reason for non finite dimensinality}, 
we obtain complex numbers $a_1$ and $a_2$ 
such that 
\[\max\limits_{j=1\ldots 2n}\big|a_1 + e^{i\theta_j}a_2\big|
< \big|a_1\big| + \big|a_2\big|.\]
Hence $\psi$ cannot be an isometry contradicting the 
hypothesis that 
$\phi$ is an isometry.    
\end{proof}
\begin{rem}
	An amusing corollary of this theorem is that the two spaces $\ell^{\infty}(n)$ and $\ell^{1}(n)$ 
	cannot be isometrically isomorphic for $n>1.$
\end{rem}
\section{Infinite dimensional embeddings of \texorpdfstring{$\ell^{1}(n)$}{TEXT}}
Let $\h_1,\ldots,\h_n$ be Hilbert spaces and $T_i$ 
be a contraction on 
$\h_i$ for $i=1,\ldots,n.$ 
Assume that the unit circle $\T$ is contained in $\sigma(T_i),$ the spectrum of $T_i,$ for $i=1,\ldots,n$. 
Denote 
$$\tilde{T_1}=T_1\otimes I^{\otimes (n-1)},
\tilde{T_2}=I\otimes T_2\otimes I^{\otimes (n-2)},\ldots ,
\tilde{T_n}=I^{\otimes (n-1)}\otimes T_n.$$
\begin{thm}\label{Infiniteoperator}
Suppose the operators $\tilde{T}_1,\ldots,\tilde{T}_n$ are defined as above. Then, the function 
$$f:\ell ^1(n)\to \mathcal B(\h_1 \otimes\cdots \otimes \h_n)$$
defined by 
\begin{align*}
	f(a_1,a_2,\ldots , a_n):=a_1\tilde{T_1}+ a_2\tilde{T_2} +\cdots + a_n \tilde{T_n}.
\end{align*}
is an isometry.
\end{thm}
\begin{proof}
Since 
$\mathbb{T}\subset \sigma (T_i)$ and 
$T_i$ is a contraction for $i=1,\ldots,n$, 
it follows that 
$\mathbb{T}\subset \partial\sigma (T_i)$ 
for $i=1,\ldots,n.$ 
From (cf. \cite[Proposition 6.7, page 210]{Conway}), we have 
$\mathbb{T}\subset \sigma _a(T_i)$ (approximate point spectrum of $T_i$) for $i=1,\ldots,n.$ 
Thus for any $i\in\{1,\ldots,n\}$ and $\lambda \in \mathbb{T},$ 
there exists a sequence of unit vectors 
$(x_{m}^{i})_{m\in\mathbb{N}}$ 
in $\h_i$ 
such that 
$$\| (T_i-\lambda)(x_{m}^{i})\| \longrightarrow 0\mbox{ as }  m\longrightarrow \infty.$$ 
Now, applying the Cauchy-Schwarz's inequality, 
we have 
\begin{align*}
	|\langle (T_i-\lambda)(x_{m}^{i}),(x_{m}^{i})\rangle |&
	\leq  \| (T_i-\lambda)(x_{m}^{i})\| \| (x_{m}^{i}) \| \\
	&= \| (T_i-\lambda)(x_{m}^{i})\| \longrightarrow 0.   
\end{align*}
as $m\longrightarrow \infty.$ Hence  
$\langle T_i(x_{m}^{i}),(x_{m}^{i})\rangle \longrightarrow \lambda$ as $m\longrightarrow \infty.$
Let $(a_1,\ldots ,a_n)$ 
be any vector in $\ell ^1(n)$ such that none of its co-ordinates zero.
Let $\lambda _1=e^{-i\arg(a_1)},\lambda _2=e^{-i\arg(a_2)},\ldots ,\lambda _n=e^{-i\arg(a_n)}.$
 Now for each $j\in \{1,\ldots ,n\}$, 
 we have $(x_{m}^{j})_{m\in\mathbb{N}}$, 
 a sequence of unit vectors from  $\h_j$, 
 such that 
 $$\langle T_j(x_{m}^{j}),(x_{m}^{j})\rangle \longrightarrow \lambda _j\mbox{ as } m\longrightarrow \infty .$$
As $m$ goes to $\infty$, we have  
\begin{align*}
	&\big|\langle (a_1T_1\otimes I^{\otimes (n-1)}+\cdots 
	+ a_nI^{\otimes (n-1)}\otimes T_n)(x_{m}^{1}\otimes\cdots\otimes x_{m}^{n}),
	(x_{m}^{1}\otimes \cdots\otimes x_{m}^{n}) \rangle \big|\\
	 &= \big|a_1\langle T_1 (x_{m}^{1}),(x_{m}^{1}))\rangle + 
	 \cdots + a_n\langle T_n (x_{m}^{n}),(x_{m}^{n}))\rangle \big|
	\longrightarrow  \big|a_1 \lambda _1 + \cdots + a_n \lambda _n\big|\\ 
	& =  |a_1| +  \cdots + |a_n| =  \| (a_1,\ldots ,a_n)\| _1.
\end{align*} 
Hence 
$\| a_1\tilde{T_1}+ a_2\tilde{T_2} +\cdots + a_n \tilde{T_n}\| \geq \| (a_1,a_2,\ldots ,a_n)\| _1.$
Also 
	$$\|a_1\tilde{T_1}+ a_2\tilde{T_2} +\cdots + a_n \tilde{T_n}\|
	\leq  |a_1|\| T_1 \| + |a_2|\| T_2 \| + \cdots + |a_n|\| T_n \|.$$
Hence 
$ \|a_1\tilde{T_1}+ a_2\tilde{T_2} +\cdots + a_n \tilde{T_n} \| = \| (a_1,a_2,\ldots ,a_n)\| _1,$
proving that $f$ is an isometry.

If some of the co-ordinates in the vector $(a_1, \ldots , a_n)$ are zero, the same argument, as above, remains valid after dropping those co-ordinates. 
\end{proof}

An adaptation of the technique involved in the proof of the Theorem \ref{Infiniteoperator},
also proves the following theorem. 
\begin{thm}
For $i=1,\ldots,n,$ let $T_i$ be a contraction on a Hilbert space $\h_i$ and 
$\T\subseteq \sigma(T_i).$  Denote $\tilde{T}_i=T_1\otimes \cdots\otimes T_i\otimes I_{\h_{i+1}}\otimes\cdots \otimes I_{\h_n}$.
Then, the function 
$$f:\ell ^1(n)\to \mathcal B(\h_1 \otimes\cdots \otimes \h_n)$$
defined by 
\begin{align*}
	f(a_1,a_2,\ldots , a_n):=a_1\tilde{T_1}+ a_2\tilde{T_2} +\cdots + a_n \tilde{T_n}.
\end{align*}
is an isometry.
\end{thm}

\begin{rem}\label{MINon2}
We have already noted that $\alpha(\ell^1(2))=1.$ 
Therefore all the operator space structures on $\ell^1(2)$, defined in the Theorem \ref{Infiniteoperator}, must be completely isometric to the MIN operator space structure. 

Suppose $T_1$ and $T_2$ are contractions 
on Hilbert spaces $\h_1$ and $\h_2$ respectively 
with the property that 
$\T\subseteq \sigma(T_i)$ for $i=1,2$.  
Denote $\tilde{T}_1=T_1\otimes I_{\h_2}$ and 
$\tilde{T}_2=I_{\h_1}\otimes T_2.$ 
Then the map $f$ defined as in the 
Theorem \ref{Infiniteoperator} is an isometry. 
The dilation theorem due to Sz.-Nagy
(cf. \cite[Theorem 1.1, page 7]{Paulsen}), 
gives  unitaries 
$U_1:\mathbb{K}_1\to \mathbb{K}_1$ and 
$U_2:\mathbb{K}_2\to \mathbb{K}_2$ dilating the contractions 
$T_1$ and $T_2$ respectively. 
The operator space structure defined by the isometry 
$g:\ell^1(2)\to \mathcal{B}(\mathbb{K}_1\otimes \mathbb{K}_2),$
where $g(a_1,a_2)=a_1U_1\otimes I_{\mathbb{K}_2}+a_2I_{\mathbb{K}_1}\otimes U_2,$ 
is no lesser than that of $f$. Since $U_1$ is a unitary map, it follows that  
the map 
$a_1U_1\otimes I_{\mathbb{K}_2}+a_2I_{\mathbb{K}_1}\otimes U_2\mapsto a_1I_{\mathbb{K}_1}\otimes I_{\mathbb{K}_2}+a_2U_1^*\otimes U_2$ is a complete isometry. 
Therefore,  
without loss of generality, 
for all operator space structures, 
defined in the Theorem \ref{Infiniteoperator}, 
we can assume that $T_1$ is $I_{\h_1}.$ 
Now suppose $k\in \N$ and $A_1,A_2\in M_k.$ The von-Neumann inequality implies that 
\[\|A_1\otimes I_{\h_1}\otimes I_{\h_2}+A_2\otimes I_{\h_1}\otimes T_2\|\leq \|A_1+A_2z\|_{\D,\infty}^{\rm op}=\|A_1\otimes (1,0)+A_2\otimes (0,1)\|_{MIN}.\]
Since $MIN$ is the smallest operator space structure, 
it follows that all operator space structures on $\ell^1(2)$, defined in the Theorem \ref{Infiniteoperator} are completely isometric to the $MIN$ structure.   
\end{rem}

\begin{rem}
Here we note that all the operator space structures on $\ell^1(3)$, defined in the Theorem \ref{Infiniteoperator}, are completely isometric to the $MIN$ structure. Suppose $T_1,\,T_2$ and $T_3$ are contractions 
on Hilbert spaces $\h_1,\,\h_2$ and $\h_3$ respectively, 
with the property that 
$\T\subseteq \sigma(T_i),$ for $i=1,2,3.$  
Then the map $f$ defined as in the 
Theorem \ref{Infiniteoperator} is an isometry. 
Using the same arguments as in the Remark \ref{MINon2}, 
here also we can assume that $T_1=I_{\h_1}.$ 
Let $k\in\N$ and $A_1,A_2,A_3\in M_k$. 
Since $U_1^*\otimes U_2 \otimes I_{\mathbb K_3}$ and $U_1^*\otimes I_{\mathbb K_2}\otimes U_3$ commute, therefore via Ando's theorem, we conclude that 
\[\|A_1\otimes I_{\h_1\otimes \h_2\otimes \h_3}+A_2\otimes I_{\h_1}\otimes T_2 \otimes I_{\h_3}+ A_3\otimes I_{\h_1}\otimes I_{\h_2}\otimes T_3\|\leq \|A_1+A_2z_2+A_3z_3\|_{\D^2,\infty}^{\rm op}.\]
The right hand quantity in this inequality is $\|A_1\otimes (1,0,0)+A_2\otimes (0,1,0)+A_3\otimes (0,0,1)\|_{MIN}$. Since $MIN$ is the smallest operator space structure, therefore all the operator space structure on $\ell^1(3),$ defined in the Theorem \ref{Infiniteoperator}, are completely isometric to the $MIN$ structure.
\end{rem}
\section{Operator space structures on $\ell ^1(n)$ different from the MIN structure}
Due to Parrott's example \cite{Pexample}, it is known that a linear contractive map on $\ell^1(3)$ may not be completely contractive. An explicit example for this is give also in the paper of G. Misra \cite{GMParrott}. This example explains that there are more than one operator space structure on $\ell^1(3)$. In this section, using the example in \cite{GMParrott}, we give an explicit operator space structure on $\ell^1(3)$, which is different from the $MIN$ structure. 

Consider the following $2\times 2$ unitary operators:
\[I_2=\left(
	\begin{array}{cc}
		1 & 0\\
		0 & 1\\
	\end{array}
\right),\,
U:=\left(
	\begin{array}{cc}
		\frac{1}{2} & \frac{\sqrt{3}}{2}\\
		\frac{\sqrt{3}}{2} & -\frac{1}{2}\\
	\end{array}
\right)\,{ \rm and }\,
V:=\left(
	\begin{array}{cc}
		\frac{1}{2} & -\frac{\sqrt{3}}{2}\\
		\frac{\sqrt{3}}{2} & \frac{1}{2}\\
	\end{array}
  \right).
\]
It is clear that the map $h:\ell^1(3)\to M_2,$ defined by $h(z_1,z_2,z_3)=z_1I+z_2U+z_3V,$ is of norm at most $1.$ The computations done in \cite{GMParrott} includes the following:
\begin{equation}\label{equal3}
\|I\otimes I+U\otimes U+V\otimes V\|=3. 
\end{equation}
and 
\begin{equation}\label{less3}
\sup_{z_1,z_2,z_3\in \D}\|z_1 I+z_2 U+z_3 V\|<3.
\end{equation}
Choose a diagonal operator 
$D$ on $\ell ^2(\mathbb{Z})$ 
such that 
$\| D \|\leq 1$ and $\mathbb{T}\subset \sigma(D)$. 
Define 
$$\tilde{T_1}:=\left[
\begin{array}{ll}
      I & 0\\
     0 & D \\
\end{array} 
\right],\tilde{T_2}:=\left[
\begin{array}{ll}
      U & 0\\
     0 & D \\
\end{array} 
\right],\tilde{T_3}:=\left[
\begin{array}{ll}
      V & 0\\
     0 & D \\
\end{array} 
\right]$$
and 
$$\hat{T_{1}}=\tilde{T_1}\otimes I\otimes I,\,\hat{T_{2}}=I\otimes\tilde{T_2}\otimes I,\,\hat{T_{3}}=I\otimes I\otimes \tilde{T_n}.$$
Let $$S_1:=\hat{T_1}\oplus I,\,S_2:=\hat{T_2}\oplus U,\,S_3:=\hat{T_3}\oplus V$$
be operators on a Hilbert space $\mathbb{K}$. 

Define 
$$S:(\ell ^1(3),MIN)\longrightarrow B(\mathbb{K})$$
by 
$$S(e_1)=S_1,\,S(e_2)=S_2,\,S(e_3)=S_3$$
and extend it linearly.

From the Theorem \ref{Infiniteoperator}, we know that the function $(z_1,z_2,z_3)\mapsto z_1\hat{T}_1+z_2\hat{T}_2+z_3\hat{T}_3$ is an isometry  
and since $h$ is of norm at most $1,$ it follows that the map $(z_1,z_2,z_3)\mapsto z_1S_1+z_2S_2+z_3S_3$ is also an isometry.   
%
%
%
Consequently, there is an operator space structure on $\ell ^1(3)$ 
for which $S$ is a complete isometry.  Also from \eqref{equal3}, we have
\begin{align*}
	\left \|S_1\otimes I+S_2\otimes U+S_3\otimes V\right\|
	& \geq \left \|I\otimes I+U\otimes U+V\otimes V\right \| =3.
\end{align*}
On the other hand from \eqref{less3}, we have 
\[\sup_{z_1,z_2,z_3\in \D}\|z_1 I+z_2 U+z_3 V\|<3\]
and hence the operator space structure 
induced by $S$ is different from the MIN structure.

\backmatter

\chapter{List of Symbols}
$
\begin{array}{ll}
\N & \mbox{The set of all positive integers}\\
\Z & \mbox{The set of all integers}\\
\mathbb{N}_{0} & \mbox{The set of all non-negative integers}\\
\C & \mbox{Complex plane }\\
\C[Z_1,\ldots,Z_m] & \mbox{The set of all polynomials in }m \mbox{ variables}\\
z\cdot x & \sum\limits_{j=1}^{l}z_jx_j \mbox{ for }x=(x_1,\ldots,x_l)\in B^l,~z=(z_1,\ldots,z_l)\in\C^l\\
\|f\|_{\Omega,\infty} & \sup\{|f(z)|:z\in\Omega\}\\
\mathbb{H} & \mbox{A separable Hilbert space }\\
\mathcal{B}(\h) & \mbox{The set of all bounded operators on }\h \\
\sigma(T) & \mbox{ Spectrum of }T\\
H^{\infty}(\Omega) & \mbox{The set of all bounded holomorphic functions on }\Omega\\ 
\D & \mbox{ unit disk in }\C\\
\T & \mbox{ Circle of unit length }\\
\|p\|_{\W,\infty}^{\rm op} & \sup\{\|p(z)\|_{op}:z\in\W\}\\
K_G^{\C} & \mbox{ Complex Grothendieck Constant}\\
\mathcal{P}_k[Z_1,\ldots,Z_n] & \mbox{The set of all polynomials of degree }k\mbox{ in }n\mbox{ variables}\\
H^{\infty}(\W) & \mbox{The set of all complex valued bounded holomorphic function on }\W\\
H^{\infty}(\W,\D) & \left\{f\in H^{\infty}(\W):\|f\|_{\W,\infty}\leq 1 \right\}\\
H^{\infty}_{\omega}(\Omega,\D) & \left\{f\in H^{\infty}(\Omega,\D):f(\omega)=0\right\}\\
D f(\omega) & \left(\frac{\pa}{\pa z_1}f(\omega),\ldots,\frac{\pa}{\pa z_m}f(\omega)\right)\\		
\mathcal{L}[Z_1,\ldots,Z_n] & \{a_1z_1+\cdots + a_nz_n:a_i\in\C,i=1,\ldots,n\}\\
\mbox{}[x^{\sharp},y] & \sum_{j}x_jy_j\\
x^{\sharp} & \C-\mbox{valued linear map on }\mathbb{H} \mbox{ defined by }x^{\sharp}(y)=[x^{\sharp},y]\\
\h^{\sharp} & \{x^{\sharp}:x\in\h\}\\
L^2(\mathbb{T}) & \mbox{The set of all square integrable functions on } \T \mbox{ with respect to }\\
\end{array}$
\newpage
$
\begin{array}{ll}
H^2(\T) & \mbox{Hardy space}\\

L^{\infty}(\T^n) & \mbox{The set of all essentially bounded functions } 
\T^n \mbox{ with respect to  }\\
& \mbox{ Lebsegue measure}\mbox{ and equipped with the essential sup norm}\\

M_{\phi} & \mbox{Multiplication operator corresponding to }\phi\\

H_{\phi} & \mbox{Hankel matrix corresponding to symbol }\phi\\

\|\boldsymbol T\|_{\infty} & \max\big\{\|T_1\|,\ldots,\|T_n\|\big\}\\

C_k(n) & \sup\big\{\|p(\boldsymbol T)\|:\|p\|_{\D^n,\infty}\leq 1, p \mbox{~\rm is of degree at most~} k,\,\|\boldsymbol T\|_{\infty} \leq 1 \big\}\\
C(n) & \lim_{k\to \infty} C_k(n)\\
\mathscr{T}(A_1,\ldots,A_n) & \left(
{\begin{array}{ccccc}
A_1 & A_2 & A_3 &\cdots & A_n \\
0 & A_1 & A_2 & \cdots & A_{n-1}\\ 
0 & 0 & A_1 & \cdots & A_{n-2}\\
\vdots & \vdots & \vdots & \ddots & \vdots \\
0 & 0 & 0 & \cdots & A_1 \\
\end{array} } 
\right)\\
M_{k}(B) & \mbox{The set of all }k\times k \mbox{ matrices with entries in Banach space }B\\
{\rm dist}_{\infty}(\phi,K) & \mbox{Distance of }\phi \mbox{ from }K\mbox{with respect to essential sup norm}\\
A_1 & \bigsqcup_{k\in\mathbb{N}_0}P_k\\
A_2 & \bigsqcup_{k\in\mathbb{N}}P_{-k}\\
H_1 & \Big\{f:=\sum\limits_{(m,n)\in A_1}a_{m,n}z_{1}^{m}z_{2}^{n}|f\in L^{\infty}(\T^2)\Big\}\\
H_2 & \Big\{f:=\sum\limits_{(m,n)\in A_2}a_{m,n}z_{1}^{m}z_{2}^{n}|f\in L^{\infty}(\T^2)\Big\}\\
\ell^2(\N_0) & \big\{(a_0,a_1,\ldots):a_j\in\C\mbox{ for all }j\in\N_0,\sum\limits_{j\geq 0}|a_j|^2< \infty\big\}\\
p_{\!_A} & \sum_{i,j=1}^n a_{ij} z_i w_j, \mbox{ where }A=\big(\!\!\big(a_{ij}\big)\!\!\big)\\
\vartriangle & \big \{(z_1,\ldots, z_n, z_1,\ldots, z_n): |z_i| < 1, 1\leq i \leq n \big \}\\
p_{_{\!A,\vartriangle}} & \mbox{ the restriction of }p_{\!_A} \mbox{ to the diagonal set }
\vartriangle\\
\|(z_1,\ldots,z_n)\|_p &  (|z_1|^p+\cdots+|z_n|^p)^{1/p}\\
\ell^p(n) & (\C^n,\|.\|_p)\\ 
\|A\|_{\ell^{\infty}(n)\to \ell^1(n)} & \mbox{ Operator norm of }A:\ell^{\infty}(n)\to \ell^1(n)\\
A_{xy} & \mbox{The matrix }([x_j^{\sharp},y_k])_{m\times m}\\
\mathcal{D}_{\Omega}^{(\omega)} & \left\{\left(\frac{1}{2} D^2f(\omega),D f(\omega)\right): f:\Omega\rightarrow \D\mbox{ is a analytic map with }f(\omega)=0\right\}\\
M_{m}^{s} & \mbox{The set of all }m\times m \mbox{ complex symmetric matrices}\\
\|\cdot\|_\mathcal{D} & \mbox{The norm in }M_{m}^{s}\times \mathbb{C}^m \mbox{ corresponding to the unit ball }\mathcal{D}_{\Omega}^{(\omega)}\\
\mathbb{U} & \left\{(z,v):z\in\mathbb{C},v\in\mathbb{H}\mbox{ such that }|z|+\|v\|^2\leq 1\right\}\\
\|\cdot\|_\mathbb{U} & \mbox{The norm in }\C\oplus\mathbb{H} \mbox{ corresponding to the unit ball }\mathbb{U}\\
\mathcal{P}_{k}(\W,E) & \left\{p\in\C[Z_1,\ldots,Z_n]:\deg(p)\leq k \mbox{ and }p(\W)\subset E\right\}.\\
\mathcal{P}_{k}^{\w}(\W,E) & \mbox{The set of all polynomials }p\in \mathcal{P}_{k}(\W,E)\mbox{ with }p(\w)=0.\\ 
\end{array}$
\newpage
$
\begin{array}{ll}
M_k & \mbox{The set of all }k\times k\mbox{ complex matrices}\\
(X)_1 & \mbox{Open Unit Ball of Banach Space }X\\
\mathcal{P}(\C^m,M_k) & \mbox{ The set of all }M_k \mbox{ valued polynomials in }m \mbox{ variables}\\
\mathcal{P}_{k}(\C^m,M_k) & \left\{p\in\mathcal{P}(\C^m,M_k):\deg(p)\leq k \right\}.\\
\mathcal{P}_{k}^{(\omega)}(\C^m,M_k) & \mbox{The set of all polynomials }p\in\mathcal{P}_{k}(\C^m,M_k) \mbox{ such that }p(\omega)=0\\
\mathcal{P}_{n}^{(\w)}\left(\W,(M_k)_1\right) &
\left\{p\in\mathcal{P}_{n}^{(\w)}(\C^m,M_k):\|p\|_{\W,\infty}^{\rm op}\leq 1 \right\}\\

B^* & \mbox{Adjoint of the bilateral shift}\\
\ell^{2}(\Z) & \big\{(\ldots,a_{-1},a_0,a_1,\ldots):a_j\in\C\mbox{ for all }j\in\Z,\sum\limits_{j\in\Z}|a_j|^2< \infty\big\}\\
C^*(a) & C^*-\mbox{algebra geneated by }1,a \mbox{ and }a^*\\
\sigma(a) & \left\{\lambda\in\C:a-\lambda 1\mbox{ is not invertible}\right\}\\
C(X) & \mbox{The set of all complex valued continuous functions on }X\\
I_k & \mbox{Identity operator on }\C^k\\
I & \mbox{Identity operator on }\h\\
\mathfrak{R}(z) & \mbox{Real part of }z\\
H_{+} & \mbox{The right half plane}\\
P_k & \left\{(x,y)\in\Z^2:x+y=k\right\}\\
& \mbox{normalized Lebsegue measure and equipped with the }L^2\mbox{ norm}\\
arg(\alpha) & \mbox{Argument of the complex number }\alpha\\
H(T_1,T_2,\ldots) & \left(
\begin{array}{cccc}
T_1 & T_2 & T_3 & \cdots\\
T_2 & T_3 & T_4 & \cdots\\
T_3 & T_4 & T_5 & \cdots\\
\vdots & \vdots & \vdots & \\ 
\end{array}
\right)\\
T_n\left((b_m),a_0,a_1,\ldots,a_{n-1}\right) &\left(
\begin{array}{cccc}
a_0 & a_1 & \cdots & a_{n-1}\\
b_1 & a_0 & \cdots & a_{n-2}\\
\vdots & \vdots & & \vdots\\
b_{n-1} & b_{n-2} & \cdots & a_0\\
\vdots & \vdots &  & \vdots\\
\end{array}
\right)\\
M_k(V) & M_k\otimes V\\
MIN(V) & \mbox{MIN operator space structure on }V\\ 
MAX(V) & \mbox{MAX operator space structure on }V\\
\alpha(V)& \left\{\frac{\|(v_{ij})\|_{MAX}}{\|v_{ij}\|_{MIN}}:(v_{ij})\in M_{k,l}(V),k
\mbox{ and }l
\mbox{ are arbitrary positive integers}\right\}\\
{\rm dim}(V) & \mbox{Dimension of the vector space }V\\
D_{T} & \mbox{ The positive square root of the operator }I-T^*T\\
\end{array}
$



\bibliographystyle{amsalpha}\bibliography{Thesis}

\def\cprime{$'$}
\providecommand{\bysame}{\leavevmode\hbox to3em{\hrulefill}\thinspace}
\providecommand{\MR}{\relax\ifhmode\unskip\space\fi MR }
\providecommand{\MRhref}[2]{%
  \href{http://www.ams.org/mathscinet-getitem?mr=#1}{#2}
}
\providecommand{\href}[2]{#2}
\begin{thebibliography}{HWH14}

\bibitem[Agl90]{Agler}
Jim Agler, \emph{Operator theory and the {C}arath\'eodory metric}, Invent.
  Math. \textbf{101} (1990), no.~2, 483--500. \MR{1062972 (91f:47005)}

\bibitem[And63]{ando}
T.~And{\^o}, \emph{On a pair of commutative contractions}, Acta Sci. Math.
  (Szeged) \textbf{24} (1963), 88--90. \MR{0155193 (27 \#5132)}

\bibitem[Arv69]{WA1}
William~B. Arveson, \emph{Subalgebras of {$C^{\ast} $}-algebras}, Acta Math.
  \textbf{123} (1969), 141--224. \MR{0253059 (40 \#6274)}

\bibitem[Arv72]{WA2}
William Arveson, \emph{Subalgebras of {$C^{\ast} $}-algebras. {II}}, Acta Math.
  \textbf{128} (1972), no.~3-4, 271--308. \MR{0394232 (52 \#15035)}

\bibitem[BW11]{BMW}
Mih{\'a}ly Bakonyi and Hugo~J. Woerdeman, \emph{Matrix completions, moments,
  and sums of {H}ermitian squares}, Princeton University Press, Princeton, NJ,
  2011. \MR{2807419 (2012d:47003)}

\bibitem[Con90]{Conway}
John~B. Conway, \emph{A course in functional analysis}, second ed., Graduate
  Texts in Mathematics, vol.~96, Springer-Verlag, New York, 1990. \MR{1070713
  (91e:46001)}

\bibitem[DMP68]{DMP}
R.~G. Douglas, P.~S. Muhly, and Carl Pearcy, \emph{Lifting commuting
  operators}, Michigan Math. J. \textbf{15} (1968), 385--395. \MR{0236752 (38
  \#5046)}

\bibitem[EPP00]{EPP}
J{\"o}rg Eschmeier, Linda Patton, and Mihai Putinar,
  \emph{Carath\'eodory-{F}ej\'er interpolation on polydisks}, Math. Res. Lett.
  \textbf{7} (2000), no.~1, 25--34. \MR{1748285 (2001j:47017)}

\bibitem[FF90]{FFE}
Ciprian Foias and Arthur~E. Frazho, \emph{The commutant lifting approach to
  interpolation problems}, Operator Theory: Advances and Applications, vol.~44,
  Birkh\"auser Verlag, Basel, 1990. \MR{1120546 (92k:47033)}

\bibitem[Hol01]{HJ}
John~A. Holbrook, \emph{Schur norms and the multivariate von {N}eumann
  inequality}, Recent advances in operator theory and related topics ({S}zeged,
  1999), Oper. Theory Adv. Appl., vol. 127, Birkh\"auser, Basel, 2001,
  pp.~375--386. \MR{1902811 (2003e:47016)}

\bibitem[HWH14]{HMLZ}
Ming-Hsiu Hsu, Lih-Chung Wang, and Zhen He, \emph{Interpolation problems for
  holomorphic functions}, Linear Algebra Appl. \textbf{452} (2014), 270--280.
  \MR{3201101}

\bibitem[KP63]{KP}
A.~Kor{\'a}nyi and L.~Puk{\'a}nszky, \emph{Holomorphic functions with positive
  real part on polycylinders}, Trans. Amer. Math. Soc. \textbf{108} (1963),
  449--456. \MR{0159029 (28 \#2247)}

\bibitem[Mis84]{Gmthesis}
Gadadhar Misra, \emph{Curvature inequalities and extremal properties of bundle
  shifts}, J. Operator Theory \textbf{11} (1984), no.~2, 305--317. \MR{749164
  (86h:47057)}

\bibitem[Mis94]{GMParrott}
G.~Misra, \emph{Completely contractive {H}ilbert modules and {P}arrott's
  example}, Acta Math. Hungar. \textbf{63} (1994), no.~3, 291--303. \MR{1261472
  (96a:46097)}

\bibitem[MNS90]{GMSastry}
Gadadhar Misra and N.~S. Narasimha~Sastry, \emph{Completely bounded modules and
  associated extremal problems}, J. Funct. Anal. \textbf{91} (1990), no.~2,
  213--220. \MR{1058968 (91m:46074)}

\bibitem[MP93]{GMPati}
Gadadhar Misra and Vishwambhar Pati, \emph{Contractive and completely
  contractive modules, matricial tangent vectors and distance decreasing
  metrics}, J. Operator Theory \textbf{30} (1993), no.~2, 353--380. \MR{1305512
  (95j:46092)}

\bibitem[Nik86]{Nik}
N.~K. Nikol{\cprime}ski{\u\i}, \emph{Treatise on the shift operator},
  Grundlehren der Mathematischen Wissenschaften [Fundamental Principles of
  Mathematical Sciences], vol. 273, Springer-Verlag, Berlin, 1986, Spectral
  function theory, With an appendix by S. V. Hru{\v{s}}{\v{c}}ev [S. V.
  Khrushch{\"e}v] and V. V. Peller, Translated from the Russian by Jaak Peetre.
  \MR{827223 (87i:47042)}

\bibitem[Par70]{Pexample}
Stephen Parrott, \emph{Unitary dilations for commuting contractions}, Pacific
  J. Math. \textbf{34} (1970), 481--490. \MR{0268710 (42 \#3607)}

\bibitem[Par78]{SP}
\bysame, \emph{On a quotient norm and the {S}z.-{N}agy\thinspace -\thinspace
  {F}oia\c s lifting theorem}, J. Funct. Anal. \textbf{30} (1978), no.~3,
  311--328. \MR{518338 (81h:47006)}

\bibitem[Pau92]{VP}
Vern~I. Paulsen, \emph{Representations of function algebras, abstract operator
  spaces, and {B}anach space geometry}, J. Funct. Anal. \textbf{109} (1992),
  no.~1, 113--129. \MR{1183607 (93h:46001)}

\bibitem[Pau02]{Paulsen}
Vern Paulsen, \emph{Completely bounded maps and operator algebras}, Cambridge
  Studies in Advanced Mathematics, vol.~78, Cambridge University Press,
  Cambridge, 2002. \MR{1976867 (2004c:46118)}

\bibitem[Pis01]{Pisier}
Gilles Pisier, \emph{Similarity problems and completely bounded maps}, expanded
  ed., Lecture Notes in Mathematics, vol. 1618, Springer-Verlag, Berlin, 2001,
  Includes the solution to ``The Halmos problem''. \MR{1818047 (2001m:47002)}

\bibitem[Pis03]{GP}
\bysame, \emph{Introduction to operator space theory}, London Mathematical
  Society Lecture Note Series, vol. 294, Cambridge University Press, Cambridge,
  2003. \MR{2006539 (2004k:46097)}

\bibitem[Rud08]{RW}
Walter Rudin, \emph{Function theory in the unit ball of {$\Bbb C^n$}}, Classics
  in Mathematics, Springer-Verlag, Berlin, 2008, Reprint of the 1980 edition.
  \MR{2446682 (2009g:32001)}

\bibitem[Var74]{V1}
N.~Th. Varopoulos, \emph{On an inequality of von {N}eumann and an application
  of the metric theory of tensor products to operators theory}, J. Functional
  Analysis \textbf{16} (1974), 83--100. \MR{0355642 (50 \#8116)}

\bibitem[Var76]{V2}
\bysame, \emph{On a commuting family of contractions on a {H}ilbert space},
  Rev. Roumaine Math. Pures Appl. \textbf{21} (1976), no.~9, 1283--1285.
  \MR{0430824 (55 \#3829)}

\bibitem[vN51]{vN}
Johann von Neumann, \emph{Eine {S}pektraltheorie f\"ur allgemeine {O}peratoren
  eines unit\"aren {R}aumes}, Math. Nachr. \textbf{4} (1951), 258--281.
  \MR{0043386 (13,254a)}

\bibitem[Woe02]{Woe}
Hugo~J. Woerdeman, \emph{Positive {C}arath\'eodory interpolation on the
  polydisc}, Integral Equations Operator Theory \textbf{42} (2002), no.~2,
  229--242. \MR{1870441 (2002i:47028)}

\bibitem[You88]{NY}
Nicholas Young, \emph{An introduction to {H}ilbert space}, Cambridge
  Mathematical Textbooks, Cambridge University Press, Cambridge, 1988.
  \MR{949693 (90e:46001)}

\end{thebibliography}

\end{document}